\documentclass[11pt,a4paper,reqno]{amsart}

\usepackage[utf8]{inputenc}
\usepackage[T1]{fontenc}

\usepackage[left=1in,right=1in]{geometry}

\usepackage{amsmath}
\usepackage{amssymb}
\usepackage{amsfonts}
\usepackage{graphicx}
\usepackage{amsthm}
\usepackage{physics}
\usepackage{xfrac}
\usepackage[colorlinks=true,linkcolor=cyan,citecolor=magenta]{hyperref}

\usepackage{tikz}
\usetikzlibrary{hobby}

\usepackage{mathrsfs}
\usepackage{sobolev}
\DefaultSet{\Gs}
\usepackage{esint}
\usepackage{bbm}
\usepackage{mathtools}
\usepackage{cases}
\usepackage{enumerate}
\usepackage{stmaryrd}

\theoremstyle{plain}
\newtheorem{thm}{Theorem}[section]
\newtheorem{prop}[thm]{Proposition}

\newtheorem{lem}[thm]{Lemma}

\theoremstyle{definition}
\newtheorem{defi}[thm]{Definition}

\theoremstyle{remark}
\newtheorem*{rk}{Remark}

\newcommand\mrm[1]{\ensuremath{\mathrm{#1}}}

\newcommand{\wt}[1]{\widetilde{#1}}
\newcommand{\wh}[1]{\widehat{#1}}
\newcommand{\R}{\mathbb{R}}

\newcommand\pd{\partial}

\DeclareMathOperator{\dist}{\mathrm{dist}}
\DeclareMathOperator{\Div}{\mathrm{div}}
\newcommand*\lap{\mathop{}\!\mathbin{\triangle}}
\newcommand*\DD{\mathop{}\!\mathbin{\mathrm{D}}}

\newcommand*{\txfrac}[2]{\textstyle{\frac{#1}{#2}}}

\providecommand\given{} 
\newcommand\SetSymbol[1][]{\nonscript\:#1\vert \allowbreak \nonscript\: \mathopen{}}
\DeclarePairedDelimiterX\Set[1]\{\}{ \renewcommand\given{\SetSymbol[\delimsize]} #1 }

\newcommand{\ini}[1]{\qty(#1)_{\mathrm{I}}}
\newcommand{\itn}[1]{{#1}^{(n)}}
\newcommand{\itm}[1]{{#1}^{(n+1)}}

\newcommand{\vv}{{\vb{v}}}
\newcommand{\vh}{{\vb{h}}}
\newcommand{\vbu}{{\vb{u}}}

\newcommand{\vnu}{{\vb*{\nu}}}
\newcommand{\vmu}{{\vb*{\mu}}}

\newcommand{\vf}{{\vb{f}}}

\newcommand{\vj}{{\vb{j}}}
\newcommand{\veta}{{\vb*{\eta}}}
\newcommand{\vxi}{{\vb*{\xi}}}
\newcommand{\vV}{{\vb{V}}}
\newcommand{\vH}{{\vb{H}}}
\newcommand{\vW}{{\vb{W}}}
\newcommand*{\vhh}{{\vu{h}}}
\newcommand*{\vuu}{{\vu{u}}}
\newcommand*{\hh}{\qty(\txfrac{1}{2}\abs*{\vhh}^2)}

\newcommand{\Om}{{\Omega}}
\newcommand{\Gt}{{\Gamma_t}}
\newcommand{\vn}{{\vb{N}}}

\newcommand{\II}{{\vb{I\!I}}}
\newcommand{\Gs}{{\Gamma_\ast}}
\newcommand{\kk}{{\frac{3}{2}k}}
\newcommand{\vom}{{\vb*{\omega}}}
\newcommand{\gt}{{\gamma_\Gt}}

\newcommand{\OGs}{{\Omega\setminus\Gs}}

\newcommand{\X}{\mathcal{X}}
\newcommand{\id}{\mathrm{Id}}
\newcommand{\ls}{\Lambda_*}
\newcommand{\gmgm}{{\gamma_\Gamma}}

\newcommand{\K}{{\mathfrak{K}}}
\newcommand{\ka}{{\varkappa_a}}

\newcommand{\LL}{{\mathscr{L}}}
\newcommand{\f}{{\mathfrak{f}}}
\newcommand{\g}{{\mathfrak{g}}}
\newcommand{\baf}{\bar{\f}}
\newcommand{\E}{\mathfrak{E}}

\newcommand{\h}{{\mathcal{H}}}
\newcommand{\n}{{\mathcal{N}}}

\newcommand{\opA}{{\mathscr{A}}}
\newcommand{\opR}{{\mathscr{R}}}
\newcommand{\opC}{{\mathscr{C}}}
\newcommand{\opB}{{\mathscr{B}}}
\newcommand{\Pb}{{\mathbb{P}}}
\newcommand{\opb}{{\mathbf{B}}}
\newcommand{\opf}{{\mathbf{F}}}
\newcommand{\opg}{{\mathbf{G}}}
\newcommand{\opF}{{\mathscr{F}}}
\newcommand{\opG}{{\mathscr{G}}}
\newcommand{\opd}{{\mathcal{D}}}

\newcommand{\Dt}{{\mathbb{D}_t}}

\newcommand{\Dts}{{\mathbb{D}_{t\ast}}}
\newcommand{\Dbt}{{\mathbb{D}_{\beta}}}

\newcommand{\less}{\lesssim_{\Lambda_*}}
\newcommand{\lesq}{\lesssim_{Q}}

\numberwithin{equation}{section}

\usepackage[type1]{garamondlibre}

\begin{document}
	
	\title[Plasma-Vacuum Problems]{On the Free Boundary Problems for the Ideal Incompressible MHD Equations}
	\author[S. Liu]{Sicheng Liu}
	\address{The Institute of Mathematical Sciences, the Chinese University of Hong Kong, Hong Kong.}
	\email{scliu@link.cuhk.edu.hk}
	
	\thanks{This is part of the Ph.D. thesis of the first author written under the guidance of the second author at the Institute of Mathematical Sciences, the Chinese University of Hong Kong. This research is supported in part by Zheng Ge Ru Foundation, Hong Kong RGC Earmarked Research Grants CUHK-14301421, CUHK-14300819, CUHK-14302819, CUHK-14300917, and the key project of NSFC (Grant No. 12131010).}
	
	\author[Z. Xin]{Zhouping Xin}
	\address{The Institute of Mathematical Sciences, the Chinese University of Hong Kong, Hong Kong.}
	\email{zpxin@ims.cuhk.edu.hk}
	
	\date{\today}
	
	\subjclass[2020]{76W05 (76B03, 76B47, 35Q35, 76E25)}
	
	\keywords{Plasma-vacuum interface, free boundary, ideal MHD, local well-posedness}

	\begin{abstract}
		We investigate the general plasma-vacuum interface problems for the ideal incompressible MHD equations with or without surface tension and prove their nonlinear local well-posedness in standard Sobolev spaces under either non-zero surface tension or the stability condition that the magnetic fields are everywhere non-collinear on the interface. In particular, the results show that both capillary forces and tangential magnetic fields can stabilize the motion of the plasma-vacuum interfaces. Moreover, the vanishing surface tension limit results are established under the Rayleigh-Taylor sign condition or the non-collinearity condition. All these results hold with no graph assumption on the free interface.
	\end{abstract}
	
	\maketitle

	\tableofcontents
	
	\section{Introduction}
	\subsection{Formulations of the Problems}
	Magnetohydrodynamics (MHD) concerns the correlation between magnetic fields and electrically conductive fluids (like plasmas, liquid metals, electrolytes, and saltwater), whose applications span various domains, including astrophysics, geophysics, and engineering. It is imperative to comprehend the fundamental principles and intrinsic phenomena of these intricate systems, which necessitates the mathematical theories on MHD. When studying the dynamics of a plasma enclosed within a vacuum, such as in those scenarios related to fusion reactors or metallurgy, it is natural to analyze the free boundary problems.
	
	In this paper, we study the following plasma-vacuum interface problems for the ideal incompressible MHD equations with or without surface tension (c.f. \cite{Landau-Lifshitz-Vol8, book_Davidson}):
	\begin{equation}\label{MHD}
		(\mathrm{MHD})\, \begin{cases*}
			\pd_t \vv + \qty(\vv \vdot \grad) \vv + \grad p = \qty(\vh \vdot \grad) \vh &in $ \Om_t^+ $, \\
			\pd_t \vh + \qty(\vv \vdot \grad) \vh = \qty(\vh \vdot \grad) \vv &in $ \Om_t^+ $, \\ 
			\div \vv = 0 = \div \vh &in $ \Om_t^+ $;
		\end{cases*}
	\end{equation}
	\begin{equation}\label{eqn pM}
		(\mathrm{pre-Maxwell})\, \begin{cases*}
			\div \vhh = 0 \qc \curl \vhh = \vb{0} &in $ \Om_t^- $, \\
			\div \vu{E} = 0 \qc \curl \vu{E} = - \pd_t \vhh &in $ \Om_t^- $;
		\end{cases*}
	\end{equation}
	\begin{equation}\label{BC}
		(\mathrm{BC}) \begin{cases*}
			\vv \vdot \vn = \theta &on $ \Gt $, \\
			p = \alpha^2 \kappa + \txfrac{1}{2}\abs*{\vhh}^2 &on $ \Gt $, \\
			\vh \vdot \vn = 0 = \vhh \vdot \vn &on $ \Gt $, \\
			\wt{\vn} \cp \vhh = \vu{J} \qc \wt{\vn} \vdot \vu{E} = {0} &on $ \pd\Om $.
		\end{cases*}
	\end{equation}
	
	Here $ \Om \subset \R^3 $ is a bounded simply-connected domain with a smooth boundary, $ \Om_t^+ \subsetneq \Om $ is the plasma region, evolving with the plasma velocity $ \vv $, $ \Gt = \pd\Om_t^+ $ is the free interface, $\vh$ is the magnetic field in the plasma region, and $p$ is the effective pressure. In addition, $ \Om_t^- \coloneqq \Om \setminus \Om_t^+ $ is the vacuum region, $\vu{E}$ and $ \vhh $ are respectively the electric and magnetic fields in the vacuum, and $ \pd\Om $ is a solid conducting wall. Furthermore, $ \vn $ is the outer normal of $ \pd\Om_t = \Gt $, $ \theta $ is the normal speed of $ \Gt $ in the direction of $ \vn $, $ \kappa $ is the mean curvature of $ \Gt $ with respect to $ \vn $, $ \alpha \in [0, 1] $ is the surface tension coefficient, $ \wt{\vn} $ is the unit outer normal of $ \pd\Om $, and $ \vu{J} \in \mathrm{T}\pd\Om $ is the surface current (which can be generated by a coil) satisfying the compatibility condition $ \Div_{\pd\Om} \vu{J} = 0 $. Assume further that $ \Om_t^+ $ and $ \Om_t^- $ are both simply-connected.
	
	
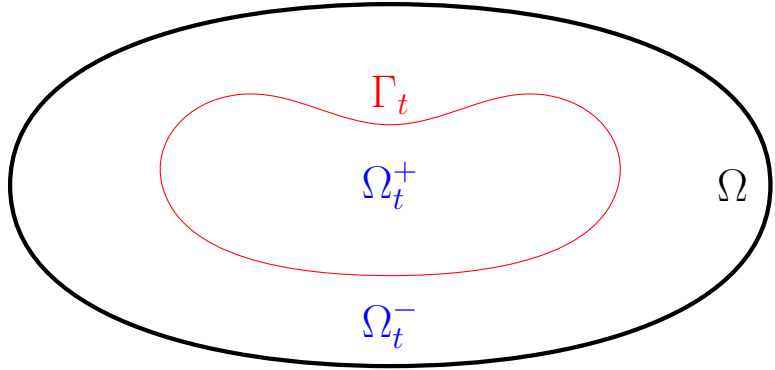
\begin{center}
		\begin{tikzpicture}[use Hobby shortcut, yscale=0.8]
			\path
			(5,0) coordinate (a0)
			(0,3) coordinate (a1)
			(-5,0) coordinate (a2)
			(0,-3) coordinate (a3);
			\draw[closed,ultra thick] (a0) .. (a1) .. (a2) .. (a3);
			\node at (4.5,0) {\LARGE $ \Omega $};
			
			\path
			(0,1) coordinate (b0)
			(2,1.5) coordinate (b1)
			(3,0) coordinate (b2)
			(0,-1.5) coordinate (b3)
			(-3,0) coordinate (b4)
			(-2,1.5) coordinate (b5);
			\draw[red,closed] (b0) .. (b1) .. (b2) .. (b3) .. (b4) .. (b5);
			\node[red] at (0, 1.5) {\LARGE$ \Gamma_t $};
			\node[blue] at (0,0) {\LARGE$ \Omega_t^+ $};
			\node[blue] at (0,-2.3) {\LARGE $ \Om_t^- $};
		\end{tikzpicture}
	\end{center}
	
	The equations \eqref{MHD} are the ideal incompressible MHD equations, which describe the motion of an inviscid perfectly conducting fluid (plasma) under the influence of magnetic fields. The first line is the combination of the incompressible Euler equations in hydrodynamics and the Maxwell equations in electrodynamics, for which the Lorentz force serves as the exterior body force. Here $p$ is the effective pressure, which is the sum of the fluid one and the magnetic one. Note that the effect of displacement currents is neglected here, due to the assumption that the velocity scale of the plasma is much less than the speed of light, i.e., the plasma motion is non-relativistic. The second line is the combination of Faraday's Law and Ohm's Law, together with the assumption that the plasma is a perfect conductor. The third line is the incompressibility of the plasma and Gauss's Law for magnetism.
	
	The system \eqref{eqn pM} is the pre-Maxwell system, which differs from the full Maxwell system on the curl of magnetic fields. Actually, for $\vu{B} \coloneqq \mu_0\vhh$, where $\mu_0$ is the permeability of the vacuum, Maxwell's equations imply that
	\begin{equation*}
		\curl \vu{B} = \frac{1}{c^2} \pdv{\vu{E}}{t}.
	\end{equation*}
	If the scale of the characteristic velocity satisfies $v \ll c$, then one has
	\begin{equation*}
		\abs{\curl \vu{B}} \sim \frac{1}{c^2} \frac{E}{\tau} \ll 1,
	\end{equation*}
	where $\tau$ is the scale of time. In other words, the Maxwell equations can be reduced to the pre-Maxwell ones \eqref{eqn pM}, as long as the plasma motion is non-relativistic.
	
	The first boundary condition in \eqref{BC} is also known as the kinematic boundary condition, which means that the free interface evolves with the plasma, or, in other words, the particles on the boundary will not enter the interior of the plasma region. The second condition is derived from the balance of stress tensor on both sides of the free interface, and the third one follows from Gauss's Law for magnetism and the physical character of the plasma. The fourth condition can be derived from Ampere's Law and the the assumption that the solid wall is a perfect conductor. Next, we shall show the conservation of physical energy to demonstrate the reasonableness of these boundary conditions.
	
	\subsection{Conservation of the Physical Energy}
	The expression of the physical energy is:
	\begin{equation}
		{E}(t) \coloneqq \int_{\Om_t^+} \frac{1}{2} \qty(\abs{\vv}^2 + \abs{\vh}^2) \dd{x} + \int_{\Om_t^-} \frac{1}{2}\abs*{\vhh}^2 \dd{x} + \int_\Gt \alpha^2 \dd{S_t},
	\end{equation}
	which consists of the kinetic part, the magnetic part, and the surface part.
	
	Indeed, it follows from \eqref{MHD} and the transport formula (c.f. \cite[p. 5]{Majda-Bertozzi2002}) that
	\begin{equation}\label{energy conserv 1}
		\begin{split}
			\dv{t}\frac{1}{2}\int_{\Om_t^+} \abs{\vv}^2 + \abs{\vh}^2 \dd{x} &= \int_{\Om_t^+} \vv \vdot \qty[\pd_t \vv + (\vv\vdot\grad) \vv] + \vh \vdot \qty[\pd_t\vh + (\vv \vdot \grad)\vh] \dd{x} \\
			&= \int_{\Om_t^+} \vv \vdot \qty[(\vh\vdot\grad)\vh] - \vv \vdot \grad p + \vh \vdot \qty[(\vh\vdot\grad)\vv]  \dd{x} \\
			&= \int_{\Om_t^+} -\div(p\vv) + \div[\vh \qty(\vv\vdot\vh)] \dd{x} \\
			&= \int_\Gt -p \vv\vdot\vn + (\vv\vdot\vh)\cdot\qty(\vh\vdot\vn) \dd{S_t} \\
			&= -\int_\Gt \alpha^2 \kappa \theta \dd{S_t} - \int_\Gt \frac{1}{2}\abs{\vhh}^2 \theta \dd{S_t}.
		\end{split}
	\end{equation}
	The transportation formula for a moving surface implies (here "$ \Div_\Gt $" is the divergence operator on $ \Gt $, and $ \vv^\top $ is the tangential projection of $ \vv $ onto $ \Gt $; for details, see \textsection~\ref{sec geo} or \cite{Ecker2004, Shatah-Zeng2008-Geo}):
	\begin{equation}\label{energy conserv 2}
		\begin{split}
			\dv{t} \int_\Gt \dd{S_t} &= \int_{\Gt} \Div_\Gt \vv \dd{S_t}  = \int_{\Gt} \Div_\Gt \vv^\top + \Div_\Gt \qty(\theta\vn) \dd{S_t} \\
			&= \int_\Gt \theta \Div_\Gt \vn + \Div_\Gt \vv^\top + \grad^\top \theta \vdot \vn \dd{S_t} \\
			&= \int_\Gt \theta\kappa \dd{S_t}.
		\end{split}
	\end{equation}
	Thus, it remains only to compute the rate of change of the magnetic energy in the vacuum. Since the interface $ \Gt $ is evolving with normal speed $ \theta $, we may assume that there is a visual fluid with velocity $ \vu{v} $ in the vacuum so that $ \vu{v} \vdot \vn = \theta $ on $ \Gt $ and $\vu{v} \vdot \wt{\vn} = 0$ on $\pd\Om$. Hence,
	\begin{equation*}\label{key}
		\begin{split}
			\dv{t} \int_{\Om_t^-} \frac{1}{2}\abs{\vhh}^2 \dd{x} &= \int_{{\Om_t^-}} \pd_t \qty(\frac{1}{2}\abs{\vhh}^2) + \div(\vu{v} \frac{1}{2}\abs{\vhh}^2) \dd{x} \\
			&= \int_{{\Om_t^-}} \vhh \vdot \pd_t \vhh \dd{x} + \int_{\Gt} \qty(-\vn \vdot \vu{v}) \cdot \frac{1}{2}\abs{\vhh}^2 \dd{S_t} \\
			&= \int_{{\Om_t^-}} \vhh \vdot \pd_t \vhh \dd{x} - \int_{\Gt} \frac{1}{2}\abs{\vhh}^2\theta \dd{S_t}.
		\end{split}
	\end{equation*}
	We note that the above relation can also be shown without introducing the visual velocity $ \vu{v} $ (c.f. \cite[p. 1326]{Wang-Xin2021}).
	Furthermore, it follows from \eqref{eqn pM} that
	\begin{equation*}\label{key}
		\begin{split}
			\int_{{\Om_t^-}} \vhh \vdot \pd_t\vhh \dd{x} &= \int_{{\Om_t^-}} - \vhh \vdot \qty(\curl\vu{E}) \dd{x} \\
			&= \int_{{\Om_t^-}} - \div(\vu{E} \cp \vhh) - \vu{E} \vdot \qty(\curl\vhh) \dd{x} \\
			&= \int_\Gt \vn \vdot (\vu{E} \cp \vhh) \dd{S_t} - \int_{\pd\Om} \wt{\vn} \vdot (\vu{E} \cp \vhh) \dd{\wt{S}} \\
			&= \int_\Gt \vhh \vdot \qty(\vn\cp\vu{E}) \dd{S_t} - \int_{\pd\Om} \vu{E} \vdot \qty(\vhh \cp \wt{\vn}) \dd{\wt{S}},
		\end{split}
	\end{equation*}
	here $ \vu{E} \cp \vhh =: \vu{S} $ is the Poynting vector, which is exactly the electromagnetic energy flux.
	
	Denote by $\vb{E}$ the electric field in the plasma region. Then Faraday's Law can be written as
	\begin{equation*}
		\curl \vb{E} = -\pdv{\vh}{t}.
	\end{equation*}
	Therefore, the following jump condition on the free interface can be derived via standard arguments:
	\begin{equation*}
		\vn \cp \qty(\vb{E} - \vu{E}) = \qty(\vh \vdot \vv) \qty(\vh - \vhh).
	\end{equation*}
	On the other hand, it follows from Ohm's Law and the perfect conductor assumption on the plasma that
	\begin{equation*}
		\vb{E} + \vv \cp \vh = \vb{0},
	\end{equation*}
	which implies the following boundary condition on $\Gt$:
	\begin{equation}\label{bc EH}
		\vn \cp \vu{E} = \qty(\vv \vdot \vn) \vhh.
	\end{equation}
	One can then derive from \eqref{bc EH} and \eqref{BC} that
	\begin{equation*}
		\int_{{\Om_t^-}} \vhh \vdot \pd_t\vhh \dd{x} = \int_\Gt \theta \abs{\vhh}^2 \dd{S_t} + \int_{\pd\Om} \vu{E} \vdot \vu{J} \dd{\wt{S}}.
	\end{equation*}
	Hence,
	\begin{equation*}\label{key}
		\dv{t} \frac{1}{2}\int_{{\Om_t^-}} \abs{\vhh}^2 \dd{x} = \int_\Gt \frac{1}{2}\abs{\vhh}^2 \theta \dd{S_t} + \int_{\pd\Om} \vu{E} \vdot \vu{J} \dd{\wt{S}}.
	\end{equation*}
	In conclusion,
	\begin{equation}\label{conserv physic energy}
		\dv{t} E(t) = - \int_{\pd_\Om} \wt{\vn} \vdot \vu{S} \dd{\wt{S}} =  \int_{\pd\Om} \vu{E} \vdot \vu{J} \dd{\wt{S}},
	\end{equation}
	whose right expression is exactly the rate of energy input by the surface current.
	
	For more discussions on the physical models for MHD, one can see \cite[Chapter~4]{goedbloed_poedts_2004}. 
	
	\subsection{Review of Previous Works}
	
	Free interface motions in hydrodynamics, such as water wave problems, have garnered considerable interest within the mathematical community, resulting in numerous noteworthy advancements. We cite Wu \cite{Wu1997, Wu1999}, Lannes \cite{Lannes2005} and Alazard-Burq-Zuily \cite{Alazard-Burq-Zuily2014} for the nonlinear local well-posedness of irrotational water waves. If the fluid flow contains vorticity, one can refer to Christodoulou-Lindblad \cite{Christodoulou-Lindblad2000}, Lindblad \cite{Lindblad2005}, Coutand-Shkoller \cite{Coutand-Shkoller2007}, Cheng-Coutand-Shkoller \cite{Cheng-Coutand-Shkoller2008}, Zhang-Zhang \cite{Zhang-Zhang2008}, and Shatah-Zeng \cite{Shatah-Zeng2008-Geo, Shatah-Zeng2008-vortex, Shatah-Zeng2011} for the nonlinear local theories of water waves and vortex sheet problems. One may see the monograph by Lannes \cite{Lannes2013} for more comprehensive references on water wave problems.
	
	Unlike the free boundary problems for the Euler equations, whose local theories have been thoroughly investigated, the study of MHD cases is still in its infancy due to the strong coupling between the magnetic and velocity fields. Since two hyperbolic systems must be solved simultaneously, it is more demanding to establish the mathematical theories on free boundary problems for MHD. In particular, the mechanism of the plasma motion in the presence of magnetic fields is a critical issue that remains to be fully understood. An intriguing observation is that magnetic fields may stabilize the motion of current-vortex sheets (c.f. \cite{Trakinin2005, Coulombel-Morando-Secchi-Trebischi2012, Sun-Wang-Zhang2018, Wang-Xin2021, Liu-Xin2023}), which is quite different from the Kelvin-Helmholtz instability of vortex sheets (c.f. \cite{Ebin1988, Majda-Bertozzi2002}). 
	
	For current-vortex sheet problems (describing the motion of two plasmas separated by a free interface, along which the velocity and magnetic fields admit tangential jumps), Syrovatskij \cite{Syrovatskij} and Axford \cite{axford1962note} discovered the stability conditions for planer incompressible current-vortex sheets in the middle of last century (see Landau-Lifshitz \cite[\textsection~71]{Landau-Lifshitz-Vol8}):
	\begin{equation}\label{Syr 1"}
		\rho_+ \abs{\vh_+}^2 + \rho_- \abs{\vh_-}^2 > \frac{\rho_+ \rho_-}{\rho_+ + \rho_-}\abs{\llbracket\vv\rrbracket}^2,
	\end{equation}
	\begin{equation}\label{Syr 2"}
		\qty(\rho_+ + \rho_-)\abs{\vh_+ \cp \vh_-}^2 \ge \rho_+ \abs{\vh_+ \cp \llbracket\vv\rrbracket}^2 + \rho_-\abs{\vh_-\cp\llbracket\vv\rrbracket}^2,
	\end{equation}
	where $ \llbracket\vv\rrbracket \coloneqq \vv_+ - \vv_- $ is the velocity jump. If the original Syrovatskij condition \eqref{Syr 2"} were replaced by the following strict one: 
	\begin{equation}\label{Syro 3"}
		\qty(\rho_+ + \rho_-)\abs{\vh_+ \cp \vh_-}^2 > \rho_+ \abs{\vh_+ \cp \llbracket\vv\rrbracket}^2 + \rho_-\abs{\vh_-\cp\llbracket\vv\rrbracket}^2,
	\end{equation}
	from which \eqref{Syr 1"} follows, we mention the works by Morando-Trakhinin-Trebeschi \cite{Morando-Trakhinin-Trebeschi2008}, Sun-Wang-Zhang \cite{Sun-Wang-Zhang2018}, and Liu-Xin \cite{Liu-Xin2023} for the local well-posedness under \eqref{Syro 3"}. The above works demonstrate that the strict Syrovatskij condition \eqref{Syro 3"} indeed has a nonlinear stabilizing effect on the free interface. For the results under the surface tension, one can see Li-Li \cite{Li-Li2022} and Liu-Xin \cite{Liu-Xin2023}.
	
	For plasma-vacuum interface problems, if the magnetic field is parallel to the free boundary and $\vhh \equiv \vb{0}$, under the following Rayleigh-Taylor sign condition for the effective pressure:
	\begin{equation}\label{RT}
		-\nabla_\vn p > 0 \qq{on} \Gt,
	\end{equation}
	Hao-Luo derived a priori estimates \cite{Hao-Luo2014} and linear local well-posedness \cite{Hao-Luo2021}. Gu-Wang \cite{Gu-Wang2019} proved the nonlinear local well-posedness for a flat initial surface under \eqref{RT}. For the case with surface tension, see Gu-Luo-Zhang \cite{Gu-Luo-Zhang2021, Gu-Luo-Zhang2022} for the local well-posedness and vanishing surface tension limit for a flat initial interface. In \cite{Hao-Luo2020}, Hao-Luo gave an counterexample indicating that the correct sign in \eqref{RT} is essential for the well-posedness, as indicated by Ebin \cite{Ebin1987} for water waves. If the magnetic field in the vacuum is nontrivial and tangential to the free surface, as a reduction of \eqref{Syro 3"}, it is natural to consider a stability condition that the magnetic fields are everywhere non-collinear on the interface (which was first used by Trakhinin \cite{Trakhinin2010} in the study of  the compressible problems), i.e.
	\begin{equation}\label{eqn non collienar}
		\abs*{\vh \cp \vhh} > 0 \qq{on} \Gt.
	\end{equation}
	For the incompressible plasma-vacuum problems under \eqref{eqn non collienar}, with the graph assumption on the interface, Mordando-Trakhinin-Trebeschi \cite{Morando-Trakhinin-Trebeschi2014} showed the local well-posedness of the linearized problems, and Sun-Wang-Zhang \cite{Sun-Wang-Zhang2019} proved the nonlinear local well-posedness. Concerning the global theory, Wang-Xin \cite{Wang-Xin2021} established it for both the plasma-vacuum and the plasma-plasma cases, assuming the presence of transversal magnetic fields, the magnetic diffusivity, and surface tensions. 
	
	Despite these recent significant advancements, many issues concerning the plasma-vacuum motions require further study. For example, all existing local theories, except for the a priori estimates \cite{Hao-Luo2014} and linear results \cite{Hao-Luo2021} by Hao and Luo, are established under graph assumptions on the free interface. Meanwhile, in many realistic physical and engineering models, the free interface cannot be represented simply by a graph, which necessitates more research to understand fully these practical phenomena. Deriving the well-posedness of general problems from those results for graph interfaces is quite daunting, even for the pure fluid cases (c.f. \cite{Coutand-Shkoller2007, Shatah-Zeng2011}), because utilizing partitions of unity involves fairly complicated estimates due to simultaneous manipulations of numerous hyperbolic systems.  Furthermore, in the presence of capillary forces, the results by Gu-Luo-Zhang \cite{Gu-Luo-Zhang2021, Gu-Luo-Zhang2022} were established under the assumptions that $\vhh\equiv\vb{0}$ and the initial interface is flat. The local well-posedness for plasma-vacuum problems under non-zero $\vhh$ and non-zero surface tensions is still open. Whether the magnetic field in the vacuum can influence the stabilization effect of the surface tension remains to be studied. 
	
	Instead of using partitions of unity, we opt for a more geometric approach introduced in \cite{Shatah-Zeng2011} to characterize the problem and derive the nonlinear local theory in a more intrinsic way (see also our previous work on the current-vortex sheet problems \cite{Liu-Xin2023}). Moreover, such an approach can also help to understand various stability conditions more clearly. Indeed, it will be seen in the equation \eqref{eqn Dt^2 kappa} that the surface tension corresponds to a third order positive differential operator acting on the free interface, which serves as a stabilizer for the surface motion, while the non-collinearity condition on the magnetic fields and the Rayleigh-Taylor sign condition respectively correspond to a second order positive differential operator and a first order one. Thus, concerning the stabilization effect, it seems that the following orders are plausible: \textit{surface tensions} $ > $ \textit{the non-collinearity condition} $ > $ \textit{the Rayleigh-Taylor sign condition}.
	The results in this paper not only establish the nonlinear local well-posedness of general plasma-vacuum problems but also facilitate the study of long-term plasma dynamics, especially the creation of splash singularities within a finite time. Practically, they may also help to comprehend and compute abundant phenomena in physics and engineering that entail free boundaries, such as those related to geodynamos, solar winds, liquid metal batteries, and fusion reactors.
	
	\section{Main Results}	
	\subsection{Stabilization Effect of the Surface Tension} Assume that the surface current $ \vu{J} : \pd\Om \to \mathrm{T}\pd\Om $ satisfies 
	\begin{equation}\label{space vJh}
		\vu{J}\in C^0\qty{[0, T^*]; H^{\kk-\frac{1}{2}}(\pd\Om)} \cap C^1\qty{[0, T^*]; H^{\kk-\frac{3}{2}}(\pd\Om)}, \qand \Div_{\pd\Om} \vu{J}(t) = 0
	\end{equation}
	for a large constant $T^* \gg 1$. When capillary forces prevail, the following theorem holds:
	\begin{thm}[$ \alpha = 1 $ case]\label{thm p-v wp}
		Suppose that $ k \ge 2 $ is an integer, $ \Om \subset \R^3 $ is a bounded simply-connected domain with a $C^1 \cap H^{\kk+1} $ boundary, and $ \vu{J}(t) $ is a tangential vector field on $ \pd\Om $ satisfying \eqref{space vJh}.
		For given initial hypersurface $ \Gamma_0 \in H^{\kk+1} $ and two solenoidal vector fields $ \vv_0, \vh_0 \in H^{\kk}(\Om_0) $, if $ \Gamma_0 $ separates $ \Om $ into two disjoint simply-connected parts, then there exists a constant $ T > 0 $, so that the problem \eqref{MHD}-\eqref{BC} has a solution in the space
		\begin{equation*}
			\Gt\in C^0_t H^{\kk+1}, \quad \vv\in C^0_t H^{\kk}(\Om_t^+), \quad \vh \in C^0_t H^{\kk}(\Om_t^+), \quad \vhh \in C^0_tH^{\kk}(\Om_t^-)
		\end{equation*}
		for $ t \in [0, T] $. Furthermore, if $ k \ge 3 $, the solution is unique and  depends on the initial data continuously, i.e. the problem \eqref{MHD}-\eqref{BC} is locally well-posed.
	\end{thm}
	
	\begin{rk}
		In \cite{Gu-Luo-Zhang2021}, the authors showed the local well-posedness under the assumptions that $\vhh \equiv \vb{0}$ and the initial free interface is flat. Note that these two additional assumptions are removed here.
	\end{rk}
	
	\subsection{Vanishing Surface Tension Limit under the Rayleigh-Taylor Sign Condition}
	
	Assume that $\vhh \equiv \vb{0}$ (which follows automatically if $\vu{J} \equiv \vb{0}$), and there is a positive constant $ \lambda_0 $ such that the following Rayleigh-Taylor sign condition holds:
	\begin{equation}\label{R-T}
		\mathfrak{t} \coloneqq - \nabla_{\vn} p \ge \lambda_0 > 0 \qq{on} \Gt,
	\end{equation}
	where $ p $ is the effective pressure, then the following vanishing surface tension result holds:
	\begin{thm}[$\alpha \to 0$ limit]\label{thm vani st RT}
		Assume that $k \ge 3$ is an integer, $ 0 < \alpha \le 1 $, and $ \vhh \equiv \vb{0} $.  For given initial data $ \Gamma_0 \in H^{\kk+1}$ and $ \vv_0, \vh_0 \in H^{\kk}(\Om_0^+) $ satisfying the Rayleigh-Taylor sign condition \eqref{R-T}, there exists a constant $T > 0$, independent of $\alpha$, so that the problem \eqref{MHD}, \eqref{BC} is well-posed for $t \in [0, T]$. Furthermore, as $ \alpha \to 0 $, subject to a subsequence, the solution to the problem \eqref{MHD}, \eqref{BC} with surface tension converges weakly to a solution to \eqref{MHD}, \eqref{BC} with $ \alpha = 0 $ in the space
		\begin{equation*}\label{key}
			\Gt \in H^{\kk} \qc \vv \in H^{\kk}(\Om_t^+), \qand \vh \in H^{\kk}(\Om_t^+).
		\end{equation*}
	\end{thm}
	
	\begin{rk}
		In \cite{Gu-Luo-Zhang2022}, the authors showed the vanishing surface tension limit under \eqref{R-T} for a flat initial interface. Our result can be regarded as a generalization of that in \cite{Gu-Luo-Zhang2022}. Although it seems plausible to derive the result for general interface problems via partitions of unity, the arguments will involve challenging difficulties. For instance, one is that the estimates of the transition maps are quite demanding, due to the strong coupling between multiple plasma bulks and the interactions between the velocity and magnetic fields. The other is that the Alinhac good unknown methods applied in \cite{Gu-Luo-Zhang2022} cannot be utilized straightforwardly to the general case.
	\end{rk}	
	
	\subsection{Stabilization Effect of the Non-collinear Magnetic Fields}
	If there is no surface tension, we show the local well-posedness under \eqref{eqn non collienar}.	Due to the hairy ball theorem, \eqref{eqn non collienar} cannot hold uniformly on a surface diffeomorphic to a sphere. Thus, we consider the following region:
	\begin{center}
		\begin{tikzpicture}
			\path
			(0,0) coordinate (a0)
			(4,0) coordinate (a1)
			(4,4) coordinate (a2)
			(0,4) coordinate (a3);
			\draw[dashed, thick] (a0) -- (a3);
			\draw[dashed, thick] (a1) -- (a2);
			\draw[ultra thick] (a0) -- (a1);
			\draw[ultra thick] (a2) -- (a3);
			
			\path;
			\draw[red,use Hobby shortcut] (0,2) .. (0.25,2) .. (1,2.5) .. (2,3) .. (3,2.5) .. (2,2) .. (1.5,1) .. (2,1.2) .. (3.5,2) .. (4,2);
			
			\node[blue] at (3.5,1) {\Large $ \Omega^-_t $};
			\node[blue] at (0.7,3) {\Large $ \Omega^+_t $};
			\node[red] at (2.5, 2.5) {\Large $ \Gamma_t $};
			\node at (6.5,2) {\Large $ \Omega \coloneqq \mathbb{T}^2 \times (-1, 1) $};
			\node at (6.5,0.2) {\Large $ \Gamma_- \coloneqq \mathbb{T}^2 \times \{-1\} $};
			\node at (6.5,3.8) {\Large $ \Gamma_+ \coloneqq \mathbb{T}^2 \times \{+1\} $};
		\end{tikzpicture}
	\end{center}
	here $\Gt$ is a $C^1 \cap H^2 $ hypersurface diffeomorphic to $\mathbb{T}^2$. Note that $\Gt$ is not assumed to be a graph, so it can represent a general portion of sea waves. Accordingly, the boundary conditions are modified as:
	\begin{equation}\label{BC'}
		(\mathrm{BC'}) \begin{cases*}
			\vv \vdot \vn = \theta &on $ \Gt $, \\
			p = \alpha^2 \kappa + \txfrac{1}{2}\abs*{\vhh}^2 &on $ \Gt $, \\
			\vh \vdot \vn = 0 = \vhh \vdot \vn &on $ \Gt $, \\
			\vv \vdot \wt{\vn}_+ = \vh \vdot \wt{\vn}_+ = 0 &on $ \Gamma_+ $, \\
			\wt{\vn}_- \cp \vhh = \vu{J} \qc \wt{\vn}_- \vdot \vu{E} = 0 &on $ \Gamma_- $.
		\end{cases*}
	\end{equation}
	Furthermore, suppose that the surface current $\vu{J}(t)$ satisfies
	\begin{equation}\label{J'}
		\vu{J} \in C^0\qty{[0, T^*]; H^{\kk-\frac{1}{2}}(\Gamma_-)} \cap C^1 \qty{[0, T^*]; H^{\kk-\frac{3}{2}}(\Gamma_-)},
	\end{equation}
	and the following compatibility conditions:
	\begin{equation}\label{compatibility J_hat}
		\Div_{\mathbb{T}^2} \vu{J} = 0 \qand \vu{J} \vdot \wt{\vn}_- = 0 \qq{on} \Gamma_-.
	\end{equation}
	
	It follows from \eqref{eqn non collienar} that
	\begin{equation}\label{def Upsilon vh vhh}
		0 < \Upsilon(\vh, \vhh) \coloneqq \inf_{\substack{\vb{a} \in \mathrm{T}\Gamma_t; \\ \abs{\vb{a}}=1}} \inf_{z \in \Gamma_t} \abs{\vb{a} \vdot \vh(z)}^2 + \abs*{\vb{a} \vdot \vhh (z)}^2.
	\end{equation}
	Then, the following local well-posedness result holds:
	\begin{thm}[$ \alpha = 0 $ case]\label{thm alpha=0 case"}
		Let $ k \ge 3 $ be an integer and $ \Om \coloneqq \mathbb{T}^2 \times (-1, 1)$. Suppose that $ \Gamma_0 $ is a $ C^1 \cap H^{\kk+\frac{1}{2}} $ hypersurface diffeomorphic to $ \mathbb{T}^2 $ separating $ \Om $ into two parts. Assume that $ \vv_0, \vh_0 \in H^{\kk}(\Om_0^+) $ are two $ H^{\kk} $ solenoidal vector fields, and $ \vhh_0 \in H^\kk(\Om_0^-) $ obtained by solving \eqref{eqn pM}, \eqref{BC'} with $ \vu{J}(0) $ satisfying
		\begin{equation*}
			\begin{split}
				\Upsilon(\vh_0, \vhh_0) \ge 2\mathfrak{s}_0 > 0.
			\end{split}
		\end{equation*}
		If $ \Gamma_0 $ does not touch the top or the bottom boundary, namely, for some positive constant  $ c_0 $,
		\begin{equation*}\label{key}
			\dist\qty(\Gamma_0, \Gamma_{\pm}) \ge 2c_0 > 0,
		\end{equation*}
		then there exists a constant $ T > 0 $, so that the plasma-vacuum interface problem \eqref{MHD}-\eqref{eqn pM}, \eqref{BC'} admits a unique solution in the space
		\begin{equation*}\label{key}
			\Gt \in C^0\qty([0, T]; H^{\kk+\frac{1}{2}}); (\vv, \vh) \in C^0\qty([0, T]; H^{\kk}(\Om_t^+)); \vhh \in C^0\qty([0, T]; H^{\kk}(\Om_t^-)).
		\end{equation*}
		Furthermore, the solution $ (\Gt, \vv, \vh, \vhh) $ satisfies
		\begin{equation*}\label{key}
			\Upsilon(\vh, \vhh) \ge \mathfrak{s}_0 \qand \dist\qty(\Gt, \Gamma_\pm) \ge c_0,
		\end{equation*}
		and depends on the initial data continuously. That is, the problem \eqref{MHD}-\eqref{eqn pM}, \eqref{BC'} is locally wellposed.
	\end{thm}
	
	\begin{rk}
		Theorem~\ref{thm alpha=0 case"} is a generalization of the main results in \cite{Sun-Wang-Zhang2019}, whose authors established the local well-posedness under the graph assumption on the free boundary, while Theorem~\ref{thm alpha=0 case"} applies to general interfaces. We utilize a geometric characterization of the free interface to eliminate the graph assumption.
	\end{rk}
	
	Moreover, the following result on the vanishing surface tension limit under \eqref{eqn non collienar} holds:
	\begin{thm}[$ \alpha \to 0 $ limit]\label{thm vanishing surf limt"}
		Assume that $ 0 \le \alpha \le 1,  k \ge 3 $, $ \Om = \mathbb{T}^2 \times (-1, 1) $, $ \vu{J} $ is a given tangential vector field on $ \Gamma_- $ satisfying \eqref{J'}-\eqref{compatibility J_hat}, the initial data $ \Gamma_0 \in H^{\kk+1} $ is diffeomorphic to $ \mathbb{T}^2 $ with $ \dist(\Gamma_0, \Gamma_\pm) \ge 2 c_0 > 0 $, and $ \vv_{0}, \vh_{0} \in H^{\kk}(\Om_0^+)$ are solenoidal. Suppose further that $ \vhh_0 \in H^{\kk}(\Om_0^-) $ satisfies $ \Upsilon(\vh_0, \vhh_0) \ge 2\mathfrak{s}_0 > 0 $. Then, there is a constant $ T > 0 $, independent of $ \alpha $, so that the problem \eqref{MHD}-\eqref{eqn pM}, \eqref{BC'} is well-posed for $ t \in [0, T] $. Furthermore, as $ \alpha \to 0 $, by passing to a subsequence, the solution to \eqref{MHD}-\eqref{eqn pM}, \eqref{BC'} with surface tension converges weakly to a solution to \eqref{MHD}-\eqref{eqn pM}, \eqref{BC'} with $ \alpha = 0 $ in the space $ \Gt \in H^{\kk+\frac{1}{2}} $, $ \vv, \vh \in H^{\kk}(\Om_t^+) $, and $ \vhh \in H^{\kk}(\Om_t^-) $. 
	\end{thm}
	
	\begin{rk}
		It should be emphasized here that Theorem~\ref{prop evo kappa'} is in fact also a main result in this paper, which yields a fine description of the evolution of the mean curvature for the plasma-vacuum interface that enables one to clarify the effects of various stability conditions, such as the surface tension, the non-collinearity condition, and the Rayleigh-Taylor sign condition. Indeed, the a priori estimates under those stability conditions can be derived from \eqref{eqn Dt^2 kappa} by using the methods introduced in \cite{Shatah-Zeng2008-Geo} and \cite{Liu-Xin2023}. In addition, it is natural to ask whether the uniform-in-$\alpha$ a priori estimates hold when $ \vhh \not\equiv \vb{0} $. Although it seems feasible to derive them from \eqref{eqn Dt^2 kappa}, one needs to control the Sobolev norm of the term $ \DD_{(\Dt\vhh - \DD_\vhh \vv)}\kappa$, which cannot be dominated by the terms appeared in the standard energy. Thus, one may need new techniques to establish the local well-posedness theory under merely the Rayleigh-Taylor sign condition if $\vhh \not\equiv \vb{0}$.
	\end{rk}

	\section{Preliminaries}\label{sec prelimi}
	\subsection{Geometry of Hypersurfaces}\label{sec geo}
	In this subsection, we shall discuss briefly some geometric properties of hypersurfaces, and one can refer to \cite{Ecker2004, Shatah-Zeng2008-Geo, Liu-Xin2023} for more detailed discussions.
	Suppose that  $ \Gt \subset \R^d $ is a family of  hypersurfaces evolving with the velocity $ \vv : \Gt \to \R^d $, and $\vn$ is the unit normal vector of $\Gt$.
	
		Note that the derivatives on hypersurfaces can be defined extrinsically via the projections from $ \R^d $. In particular, for a function $ f $ and a vector field $ \vb{X} $ defined in a neighborhood of $ \Gt \subset \R^d $, the tangential gradient of $ f $ is given by
		\begin{equation}\label{key}
			\grad_\Gt f = \grad f - (\vn \vdot \grad f) \vn,
		\end{equation}
		where $ \grad f $ is the gradient in $ \R^d $; and the tangential divergence of $ \vb{X} $ is defined to be
		\begin{equation}\label{eqn div Gt div Rd}
			\Div_\Gt \vb{X} = \Div_{\R^d} \vb{X} - \vn \vdot \DD_{\vn} \vb{X},
		\end{equation}
		where $ \DD $ is the covariant derivative in $ \R^d $. Thus, the Laplace-Beltrami operator can be computed via
		\begin{equation}\label{lap gt f alt}
			\begin{split}
				\lap_\Gt f &= \Div_\Gt \grad_\Gt f = \Div_\Gt \grad f - \Div_\Gt \qty[(\vn\vdot\grad f)\vn] \\
				&= \lap_{\R^d} f - \DD^2 f(\vn, \vn) - \kappa \vn \vdot \grad f,
			\end{split}
		\end{equation}
		for any $ C^2 $ function defined in a neighborhood of $ \Gt \subset \R^d $, where $\kappa$ is the mean curvature of $\Gt$ to be given in \eqref{eqn kp div N}. The above definitions are identical to the intrinsic ones (c.f. \cite[Appendix~A]{Ecker2004}).
		
		The second fundamental form of $\Gt$ is a $(0, 2)$-tensor defined by:
		\begin{equation}\label{second fundamental form}
			\II (\vb*{\tau}) \coloneqq \DD_{\vb*{\tau}} \vn \qfor \vb*{\tau} \in \mrm{T}\Gt.
		\end{equation}
	The mean curvature is defined to be the trace of $ \II $, i.e.
	\begin{equation}\label{eqn kp div N}
		\kappa \coloneqq \tr(\II) = \Div_\Gt \vn,
	\end{equation}
	while the last equality follows from the definitions of tangential divergence and the second fundamental form (see \cite{Ecker2004} or \cite{Liu-Xin2023}). If $\Gt$ is regarded as a submanifold of $\R^d$, one can define the natural Riemannian metric and covariant derivatives. Here we state some identities, whose derivations can be found in \cite[Appendix~A]{Ecker2004}.
	The first one is Simons' identity:
	\begin{equation}\label{simons' identity}
		\lap_\Gt \II = \qty(\DD_\Gt)^2 \kappa + (\kappa \II - \abs{\II}^2 \vb{I}) \vdot \II,
	\end{equation}
	where $\DD_\Gt$ is the covariant derivative on $\Gt$.
	The second one can be derived from Codazzi's equation:
	\begin{equation}\label{lap vn}
		\lap_\Gt \vn = - \abs{\II}^2 \vn + \grad_\Gt{\kappa}.
	\end{equation}
	Moreover, the Codazzi equation also implies that $\DD_\Gt \II$ is a totally symmetric $(0, 3)$-tensor on $\Gt$.

	Next, we shall introduce some evolution equations. Denote by $\Dt$ the material derivative along the particle path given by $\vv$. For the evolution of $ \vn $, it holds that
	\begin{equation}\label{eqn dt n}
		\Dt \vn = - \qty[(\grad\vv)^* \vdot\vn]^\top,
	\end{equation}
	where ${}^\top$ is the tangential projection onto $ \mathrm{T}\Gt$.
	For the area element of $\Gt$, one has
	\begin{equation}\label{eqn Dt dS}
		\Dt \dd{S_t} = \Div_\Gt \vv \dd{S_t}.
	\end{equation}
	The evolution equation for the second fundamental form is
	\begin{equation}\label{eqn dt II tan}
		(\Dt\II)^\top = - \DD^\top \qty[(\DD\vv)^*\vn]^\top - \II \vdot \qty(\DD\vv)^\top,
	\end{equation}
	where $\DD^\top$ is also the covariant derivative on $\Gt$.	The evolution of $ \kappa $ is given by
	\begin{equation}\label{eqn dt kappa}
			\Dt \kappa = -\vn \vdot \lap_\Gt \vv - 2 \ip{\II}{\vb{A}},
	\end{equation}
	where $ \ip{\cdot}{\cdot} $ is the standard inner product of tensors, and $\vb{A}$ is a tensor on $\Gt$ given by
	\begin{equation}
		2\vb{A} \coloneqq (\DD\vv)^\top+\qty[(\DD\vv)^\top]^*.
	\end{equation}
	The second order evolution equation for $\kappa$ is:
	\begin{equation}\label{eqn Dt2 kappa alt}
		\begin{split}
			\Dt^2 \kappa =\,  &-\lap_\Gt(\vn\vdot\Dt\vv) - \abs{\II}^2 (\vn\vdot\Dt\vv) + \grad_\Gt\kappa \vdot\Dt\vv + 2\vn\vdot(\grad\vv)\vdot(\lap_\Gt\vv)^\top \\
			&+ 4 \ip{\vb{A}}{\vn\vdot\qty(\DD_\Gt)^2\vv}
			-\kappa \abs{\qty[(\grad\vv)^*\vdot\vn]^\top}^2 -2\ip{\II}{\DD_\Gt \vv \vdot \DD_\Gt \vv}  + 4 \ip{\II \vdot\vb{A} + \vb{A}\vdot\II}{\vb{A}}.
		\end{split}
	\end{equation}

	\subsection{Reference Hypersurface}
	
	Suppose that $ k \ge 2 $ is an integer, and $ \Gs \subset \Om $ is a compact reference hypersurface without boundary separating $ \Om $ into two disjoint simply-connected parts $ \Om_*^\pm $. Assume that $ \Gs $  is of Sobolev class $ H^{\kk + 1} $. Denote by $ \vn_{*} $ the outward unit normal of $ \pd \Om_*^+ = \Gs $. Let $ \II_{*} $ be the second fundamental form of $ \Gs $ with respect to $ \vn_{*} $, and $ \kappa_{*} $ the corresponding mean curvature.
	
	As in  \cite{Shatah-Zeng2011} and \cite{Liu-Xin2023}, we shall consider the evolution of hypersurfaces in a tubular neighborhood of $ \Gs $. Take a unit vector field $ \vb*{\nu} \in H^{\kk+1} (\Gs; \mathbb{R}^{2})$ for which $ \vnu \cdot \vn_{*} \ge 9/10 $ by mollifying $ \vn_* $.
	Then, one can derive from the implicit function theorem that there exists a constant $ \delta_0 > 0 $, determined by $ \Gs $ and $ \vnu $, so that each hypersurface $ \Gamma $ in a small $ C^1 $ neighborhood of $\Gs$ is associated to a unique height function $ \gamma_\Gamma : \Gs \to \R $ so that
	\begin{equation}\label{key}
		\Phi_\Gamma (p) \coloneqq p + \gamma_\Gamma(p)\vnu(p) \qfor p \in \Gs
	\end{equation}
	is a diffeomorphism from $ \Gs $ to $ \Gamma $, as long as $ \gamma_\Gamma \in [ - \delta_0, \delta_0]$. Thus, the function $ \gamma_\Gamma $ can be used to represent any hypersurface $ \Gamma $ close to $\Gs$.
	
	\begin{defi}\label{def Lambda}
		For $ \delta>0 $ and $ \frac{1+3}{2} < s \le 1+\kk $, define $ \Lambda(\Gs, s, \delta) $ to be the collection of all hypersurfaces $ \Gamma $ close to $ \Gs $, whose associated coordinate functions $ \gamma_\Gamma $ satisfy $ \abs{\gamma_\Gamma}_{H^s(\Gs)} < \delta $.
	\end{defi}
	
	As $ s > \frac{3-1}{2} + 1 $, one has $ H^s(\Gs) \hookrightarrow C^1(\Gs) $. Thus, $ \delta \ll 1 $ yields that each $ \Gamma \in \Lambda $ also separates $ \Om $ into two disjoint simply-connected domains.
	
	\subsection{Recovering a Hypersurface from Its Mean Curvature}
	Here, we represent the evolution of a hypersurface by its mean curvature $ \kappa \coloneqq \tr \II $. For an $ H^s $ hypersurface $ \Gamma \in \Lambda(\Gs, s, \delta_0) $ with $ s> 2 $, the unit normal $ \vn $ has the same regularity as $ \grad \gamma_\Gamma $. Then the mapping from $ \gamma_\Gamma \in \H{s} $ to the mean curvature $ \kappa \circ \Phi_\Gamma \in \H{s-2} $ is smooth. 
	
	In order to establish a bijection between a hypersurface and its mean curvature, one can consider the following quantity:
	\begin{equation}\label{def ka}
		\K [\gmgm](p) \equiv \ka (p) \coloneqq \kappa \circ \Phi_\Gamma (p) + a^2 \gmgm (p) \qfor p \in \Gs,
	\end{equation}
	where $ a $ is a parameter depending only on $ \Gs $ and $ \vnu $, and the leading order term of $\ka$ is the mean curvature (c.f. \cite{Shatah-Zeng2011}).
	
	For a small constant $ \delta_0 > 0 $, define
	\begin{equation}\label{def lambda*}
		\Lambda_* \coloneqq \Lambda \qty(\Gs, \kk - \frac{1}{2}, \delta_0).
	\end{equation}
	Then, one can derive the following lemma from bootstrap arguments and Simons' identity \eqref{simons' identity} (c.f. \cite[p. 719]{Shatah-Zeng2008-Geo}):
	\begin{lem}\label{est ii}
		For $ \Gamma \in \Lambda_* $ with $ \kappa \in H^s(\Gamma) $, $ \kk - \frac{5}{2} \le s \le \kk - 1 $, the following estimate holds:
		\begin{equation}\label{key}
			\abs{\vn}_{H^{s+1}(\Gamma)} + \abs{\II}_{H^s(\Gamma)} \le C_* \qty(1 + \abs{\kappa}_{H^s(\Gamma)}),
		\end{equation}
		for some constant $ C_* $ depending only on $ \Lambda_* $.
	\end{lem}
	
	If $ \Lambda_* $ is regarded as an open subset of $ \H{\kk-\frac{1}{2}} $, then $ \K $ is a $ C^3$-morphism from $ \Lambda_* \subset \H{\kk - \frac{1}{2}} $ to $ \H{\kk -\frac{5}{2}} $. Furthermore, by taking $ a \gg 1 $, one can check the positivity of $ \eval{\var{\K}}_{\Gs}^{} $ and show that $ \K $ is actually a $ C^3 $ diffeomorphism. Indeed, the following proposition holds (\cite[Lemma 2.2]{Shatah-Zeng2011}):
	
	\begin{prop}\label{prop K}
		There are positive constants $ C_*, \delta_0, \delta_1, a_0 $ depending only on $ \Gs $ and $ \vnu $ such that for $ a \ge a_0 $, $ \K $ is a $ C^3 $ diffeomorphism from $ \Lambda_* \subset \H{\kk-\frac{1}{2}} $ to $ \H{\kk-\frac{5}{2}} $. Denote by
		\begin{equation*}
			B_{\delta_1} \coloneqq \Set*{\ka \given \abs{\ka - \kappa_{*}}_{\H{\kk-\frac{5}{2}}} < \delta_1},
		\end{equation*}
		where $ \kappa_{*} $ is the mean curvature of $ \Gs $ with respect to $ \vn_{*} $, then
		\begin{equation*}
			\abs{\K^{-1}}_{C^3\qty(B_{\delta_1}; \H{\kk-\frac{1}{2}})} \le C_*.
		\end{equation*}
		Furthermore, if $ \ka \in B_{\delta_1} \cap \H{s-2}$ with $ \kk-\frac{1}{2} \le s \le \kk +1 $, then $ \gmgm,  \Phi_\Gamma \in \H{s}$, and for $ \max\qty{s'-2, -s} \le s'' \le s' \le s $, it holds that
		\begin{equation}\label{var est K^-1}
			\abs{\var{\K^{-1}}(\ka)}_{\LL\qty(\H{s''}; \H{s'})} \le C_* a^{s'-s''-2} \qty(1+ \abs{\ka}_{\H{s-2}}).
		\end{equation}		  		
	\end{prop}
	
	\subsection{Harmonic Coordinates and Dirichlet-Neumann Operators}\label{sec harmonic coord}
	
	For a given hypersurface $ \Gamma \in \Lambda(\Gs, s, \delta) $, define a map $ \mathcal{X}_\Gamma^\pm : \Om_*^\pm \to \Om_\Gamma^\pm $ by
	\begin{equation}\label{key}
		\begin{cases*}
			\lap_y \X_\Gamma^\pm = 0 &for $ y \in \Om_*^\pm $, \\
			\X_\Gamma^\pm (z) = \Phi_\Gamma(z) &for $ z \in \Gamma_* $, \\
			\X_\Gamma^- (z') = z' &for $ z' \in \pd \Omega $.
		\end{cases*}
	\end{equation}
	Then, it is clear that
	\begin{equation}\label{key}
		\norm{\nabla \X_\Gamma^\pm - \id|_{\Om_*^\pm}}_{H^{s-\frac{1}{2}}(\Om_*^\pm)} \le C \abs{\gamma_\Gamma}_{H^s(\Gs)} < C\delta,
	\end{equation}
	where $ C > 0 $ is uniform in $ \Gamma \in \Lambda(\Gs, s, \delta) $. Thus, there is a constant $ \delta_0 >0 $ determined by $ \Gs $ and $ \vnu $, so that $ \X_\Gamma^\pm $ are diffeomorphisms from $ \Om_*^\pm $ to $ \Om^\pm_\Gamma $ respectively, whenever $ \delta \le \delta_0 $.
	
	We list some basic inequalities, whose proofs are standard (c.f. \cite{Shatah-Zeng2008-Geo,Bahouri-Chemin-Danchin2011}).
	\begin{lem}\label{lem composition harm coordi}
		Suppose that $ \Gamma \in \ls $. Then there are constants $ C_1, C_2 > 0 $, depending on $ \ls $, so that
		\begin{enumerate}
			\item If $ u_\pm \in H^\sigma (\Omega_\Gamma^\pm) $ for $ \sigma \in \qty[-\kk, \kk] $, then 
			\begin{equation*}
				\dfrac{1}{C_1} \norm{u_\pm}_{H^\sigma(\Om_\Gamma^\pm)} \le \norm{u_\pm \circ \X_\Gamma^\pm}_{H^\sigma(\Om_*^\pm)} \le C_1  \norm{u_\pm}_{H^\sigma(\Om_\Gamma^\pm)}.
			\end{equation*}
			\item If $ f \in H^s (\Gamma) $ for $ s \in \qty[\frac{1}{2}-\kk, \kk-\frac{1}{2}] $, then 
			\begin{equation*}
				\dfrac{1}{C_2} \abs{f}_{H^{s}(\Gamma)} \le \abs{f\circ\Phi_\Gamma}_{\H{s}} \le C_2 \abs{f }_{H^{s}(\Gamma)}.
			\end{equation*}
		\end{enumerate}
	\end{lem}
	
	\begin{lem}\label{lem product est}
		Assume that $ \Gamma \in \Lambda_* $. Then there are constants $ C_1, C_2 > 0 $ determined by $ \Lambda_* $ such that
		\begin{enumerate}
			\item For $ u_\pm \in H^{\sigma_1}(\Om_\Gamma^\pm)  $, $ w_\pm \in H^{\sigma_2}(\Om_\Gamma^\pm)$ and $ \sigma_1 \le \sigma_2 $,
			\begin{equation*}
				\begin{split}
					\norm{u_\pm \cdot w_\pm}_{H^{\sigma_1 + \sigma_2 - \frac{3}{2}}(\Om_\Gamma^\pm)} \le C_1  \norm{u_\pm}_{H^{\sigma_1}(\Om_\Gamma^\pm)} \norm{w}_{H^{\sigma_2}(\Om_\Gamma^\pm)} \qif \sigma_2 < \dfrac{3}{2} \qc 0 <  \sigma_1 + \sigma_2 \le \kk.
				\end{split}
			\end{equation*}
			\begin{equation*}
				\begin{split}
					\norm{u_\pm \cdot w_\pm}_{H^{\sigma_1}(\Om_\Gamma^\pm)} \le C_1 \norm{u_\pm}_{H^{\sigma_1}(\Om_\Gamma^\pm)} \norm{w_\pm}_{H^{\sigma_2}(\Om_\Gamma^\pm)}  \qif \frac{3}{2} < \sigma_2 \le \kk \qc \sigma_1 + \sigma_2 > 0.
				\end{split}
			\end{equation*}
			
			\item For $ f \in H^{s_1}(\Gamma) $, $ g \in H^{s_2}(\Gamma) $ and $ s_1 \le s_2 $, 
			\begin{equation*}\label{key}
				\abs{fg}_{H^{s_1+s_2-\frac{2}{2}}(\Gamma)} \le C_2 \abs{f}_{H^{s_1}(\Gamma)} \abs{g}_{H^{s_2}(\Gamma)} \qif s_2 < 1 \qc 0 \le s_1 + s_2 \le \kk-\frac{1}{2},
			\end{equation*}
			\begin{equation*}\label{key}
				\abs{fg}_{H^{s_1}(\Gamma)} \le C_2 \abs{f}_{H^{s_1}(\Gamma)} \abs{g}_{H^{s_2}(\Gamma)}\qif 1 < s_2 \le \kk-\frac{1}{2} \qc s_1 + s_2 > 0.
			\end{equation*}
		\end{enumerate}
	\end{lem}
	
	For any smooth function $ f $ defined on $ \Gamma \in \Lambda_* $, denote by $ \h_\pm f $ the harmonic extension to $ \Om_\Gamma^\pm $, namely
	\begin{equation}\label{harmonic ext +}
		\begin{cases*}
			\lap \h_+ f = 0 \qfor x \in \Om_\Gamma^+, \\
			\h_+ f = f \qfor x \in \Gamma,
		\end{cases*}
	\end{equation}
	and
	\begin{equation}\label{harmonic ext 2}
		\begin{cases*}
			\lap \h_- f = 0 \qfor x \in \Om_\Gamma^-, \\
			\h_- f = f \qfor x \in \Gamma, \\
			\DD_{\wt{\vn}} \h_- f = 0 \qfor x \in \pd\Om.
		\end{cases*}
	\end{equation}
	The Dirichlet-Neumann operators are defined to be 
	\begin{equation}\label{def DN op}
		\n_\pm f \coloneqq \pm \vn \vdot (\grad \h_\pm f)|_{\Gamma}^{}.
	\end{equation}
	
	Assume that $ \Gamma \in \Lambda_* \subset H^{\kk-\frac{1}{2}} $ and $ \frac{3}{2}-\kk \le s \le \kk-\frac{1}{2} $. Then, the Dirichlet-Neumann operators $ \n_\pm : H^{s}(\Gamma) \to H^{s-1}(\Gamma) $ satisfy the following properties (c.f. \cite[pp. 738-741]{Shatah-Zeng2008-Geo}):
	\begin{enumerate}[1.]
		\item $ \n_\pm $ are self-adjoint on $ L^2(\Gamma) $ with compact resolvents;
		\item $\ker(\n_\pm) = \qty{\mathrm{const.}}$;
		\item There is a constant $ C_* > 0 $ uniform in $ \Gamma \in \Lambda_* $  so that
		\begin{equation}\label{property1 n}
			C_*\abs{f}_{H^{s}(\Gamma)} \ge \abs{\n_\pm(f)}_{H^{s-1}(\Gamma)} \ge \frac{1}{C_*}\abs{f}_{H^{s}(\Gamma)},
		\end{equation} 
		for any $ f $ satisfying $ \int_\Gamma f \dd{S} = 0 $;
		\item For $ \frac{1}{2}-\kk \le s_1 \le \kk-\frac{1}{2} $, there is a constant $ C_* $ determined by $ \Lambda_* $ so that
		\begin{equation}\label{equiv n lap}
			\dfrac{1}{C_*} \qty(\mathrm{I}-\lap_\Gamma)^{\frac{s_1}{2}} \le \qty(\mathrm{I}+\n_\pm)^{s_1} \le C_*\qty(\mathrm{I}-\lap_\Gamma)^{\frac{s_1}{2}},
		\end{equation}
		i.e., the norms on $ H^{s_1}(\Gamma) $ defined by interpolating $ \qty(\mathrm{I}-\lap_\Gamma)^\frac{1}{2} $ and $ \qty(\mathrm{I}+\n_\pm) $ are equivalent;
		\item For $ \frac{1}{2}-\kk \le s_2 \le \kk-\frac{3}{2} $,
		\begin{equation*}
			(\n_\pm)^{-1} : H^{s_2}_{0}(\Gamma) \to H^{s_2+1}_{0}(\Gamma),
		\end{equation*}
		\begin{equation*}
			H^{s_2}_{0}(\Gamma) = \Set*{f\in H^{s_2}(\Gamma) \given \int_\Gamma f \dd{S} = 0 }
		\end{equation*}
		are well-defined and bounded uniformly in $ \Gamma \in \Lambda_* $.
	\end{enumerate}
	
	
	At the end of this subsection, we state an important lemma (c.f. \cite[p. 863]{Shatah-Zeng2008-vortex}):
	\begin{lem}\label{lem lap-n}
		Suppose that $ \Gamma \in \Lambda_* $ with $ \kappa \in H^{\kk-\frac{3}{2}}(\Gamma) $. Then for $ \frac{1}{2} - \kk \le s \le \kk-\frac{1}{2} $, one has
		\begin{equation}\label{key}
			\abs{\qty(-\lap_\Gamma)^{\frac{1}{2}} - \n_\pm}_{\LL\qty(H^s(\Gamma))} \le C_* \qty(1+\abs{\kappa}_{H^{\kk-\frac{3}{2}}(\Gamma)}),
		\end{equation}
		where the constant $ C_*>0 $ is uniform in $ \Gamma \in \Lambda_* $.
	\end{lem}
	
	\subsection{Commutator Estimates}
	For a vector field (not necessarily being solenoidal) $ \vv(t) : \Om_t^+ \to \R^3 $, denote by
	\begin{equation*}\label{key}
		\Dt \coloneqq \pd_t + \DD_{\vv}.
	\end{equation*}
	Then the following lemma can be derived from the identities introduced in \cite[pp. 709-710]{Shatah-Zeng2008-Geo} and standard product estimates:
	\begin{lem}\label{Dt comm est lemma}
		Suppose that $ \Gt \in \Lambda_* $, and $ \vv  \in H^{\kk}(\Om_t^+) $ is the evolution velocity of $ \Om_t^+ $. Let $ f (t, x) : \Gt \to \R $ and $ h (t, x) : \Om_t^+ \to \R $ be two functions. Then the following commutator estimates hold:
		\begin{enumerate}[I.]
			\item For $ 1 \le s \le \kk $, $\norm{\comm{\Dt}{\h_+}f}_{H^s(\Om_t^+)} \lesssim_{\Lambda_*} \abs{f}_{H^{s-\frac{1}{2}}(\Gt)}\cdot \norm{\vv}_{H^{\kk}(\Om_t^+)}$;
			
			\item For $ 1 \le s \le \kk $, $\norm{\comm{\Dt}{\lap^{-1}}h}_{H^s(\Om_t^+)} \lesssim_{\Lambda_*} \norm{h}_{H^{s-2}(\Om_t^+)} \cdot \norm{\vv}_{H^\kk(\Om_t^+)} $;
			
			\item For $ -\frac{1}{2} \le s \le \kk- \frac{3}{2} $, $\abs{\comm{\Dt}{\n_\pm}f}_{H^s(\Gt)} \lesssim_{\Lambda_*} \abs{f}_{H^{s+1}(\Gt)} \cdot \abs{\vv}_{H^{\kk-\frac{1}{2}}(\Gt)} $;
			
			\item For $ \frac{1}{2} \le s \le \kk - \frac{1}{2} $, $\abs{\comm{\Dt}{\n_\pm^{-1}}f}_{H^s(\Gt)} \lesssim_{\Lambda_*} \abs{f}_{H^{s-1}(\Gt)} \cdot \abs{\vv}_{H^{\kk-\frac{1}{2}}(\Gt)}$;
			
			\item For $ -2 \le s \le \kk - \frac{5}{2} $, $\abs{\comm{\Dt}{\lap_\Gt}f}_{H^s(\Gt)} \lesssim_{\Lambda_*} \abs{f}_{H^{s+2}(\Gt)} \cdot \abs {\vv}_{H^{\kk-\frac{1}{2}}(\Gt)} $;
			
			\item For $ 0 \le s \le \kk-1 $, $ \norm{\comm{\Dt}{\DD}h}_{H^s(\Om_t^+)} \less \norm{h}_{H^{s+1}(\Om_t^+)} \cdot \norm{\vv}_{H^\kk(\Om_t^+)}$.
		\end{enumerate}
	\end{lem}
	
	\subsection{Div-Curl Systems}\label{sec div-cul system}
	In this subsection, we list some basic results on div-curl systems (c.f. \cite{Cheng-Shkoller2017} for details):
	\begin{thm}\label{thm div-curl}
		Suppose that $ U $ is a bounded domain in $ \R^3 $ for which $\pd U \in H^{\kk-\frac{1}{2}} $. Given $ \vb{f}, g \in H^{l-1}(U) $ with $ \div \vb{f} = 0 $ and $ h \in H^{l-\frac{1}{2}}(\pd U) $, consider the following boundary value problem:
		\begin{subnumcases}{}
			\curl \vbu = \vb{f} &in $ U $, \label{3.27a}\\
			\div \vbu = g &in $ U $, \label{3.27b}\\
			\vbu \vdot \vn = h &on $ \pd U $. \label{3.27c}
		\end{subnumcases}
		If on each connected component $ \Gamma $ of $ \pd U $, it holds that
		\begin{equation}\label{compatibility div-curl}
			\int_\Gamma \vf \vdot \vn \dd{S} = 0,
		\end{equation}
		and the following compatibility condition holds:
		\begin{equation}\label{3.29}
			\int_{\pd U} h \dd{S} = \int_U g \dd{x},
		\end{equation}
		then for $ 1 \le l \le \kk-1 $, there is a solution $ \vbu \in H^l(U) $  such that
		\begin{equation}\label{key}
			\norm{\vbu}_{H^l(U)} \le C\qty(\abs{\pd U}_{H^{\kk-\frac{1}{2}}}) \cdot \qty(\norm{\vf}_{H^{l-1}(U)} + \norm{g}_{H^{l-1}(U)} + \abs{h}_{H^{l-\frac{1}{2}}(\pd U)}).
		\end{equation}
		The solution is unique whenever $ U $ is simply-connected.
	\end{thm}
	\begin{rk}
		If $ \vb{f} = \curl \vbu $ for some vector field $ \vbu $, then (\ref{compatibility div-curl}) holds naturally (see \cite[Remark 1.2]{Cheng-Shkoller2017}).
	\end{rk}
	
	\begin{thm}\label{thm div-curl"}
		Consider the boundary value problem \eqref{3.27a}-\eqref{3.27b} with the boundary condition
			\begin{equation}
				\vbu \cp \vn = \vh \qq{on}  \pd U. \tag{\ref{3.27c}'}
			\end{equation}
		Then the same conclusion holds true under the same assumptions in Theorem \ref{thm div-curl} except \eqref{3.29} being replaced by that
		\begin{equation}\label{key}
			\vf \vdot \vn = \Div_{\pd\Om} \vh \qq{on} \pd U, \tag{\ref{3.29}'}
		\end{equation}
		and for $ \Sigma \subset \overline{U} $ with a piece-wise smooth boundary $ \pd\Sigma \subset \pd U $, whose unit normal $ \vb{n} $ is compatible to the orientation of $ \pd\Sigma $,
		\begin{equation}\label{key}
			\int_\Sigma \vf \vdot \vb{n} \dd{S} = \oint_{\pd\Sigma} \qty(\vn\cp\vh) \vdot \dd{\vb{r}}. \tag{\ref{3.29}''}
		\end{equation}
	\end{thm}
	
	\section{Reformulation of the Problem}
	
	\subsection{Transformation of the Velocity Field}
	As a consequence of div-curl theories, any vector field defined in a bounded simply-connected domain is determined uniquely by its divergence, curl and boundary conditions. As both the velocity and magnetic fields are solenoidal, they can be determined by the vorticities, currents and the corresponding boundary conditions.
	
	Therefore, denoting by
	\begin{equation}\label{key}
		\vom_{*} \coloneqq \qty(\curl \vv) \circ \X_\Gt^+,
	\end{equation}
	then the velocity field $ \vv $ can be uniquely determined by $ \ka $, $ \theta $ and $ \vom_* $ via solving the following div-curl problem:
	\begin{equation}\label{sys div-curl v}
		\begin{cases*}
			\div \vv  = 0 &in $\Om_t^+$, \\
			\curl \vv = \vom_{*} \circ (\X_\Gt^+)^{-1} &in $\Om_t^+$, \\
			\vv \vdot \vn = \theta &on $\Gt$.
		\end{cases*}
	\end{equation}
	
	Next, for a function $ f : \Gt \to \R $, it is natural to pull back $ \Dt f $ to $ \Gs $ via $ \Phi_\Gt $, namely, one needs to look for a vector field $ \vv_* : \Gs \to \R^3 $ so that
	\begin{equation}\label{pull back Dt to Gs}
		\Dt_* \qty(f \circ \Phi_\Gt) = \qty(\Dt f) \circ \Phi_\Gt,
	\end{equation}
	where
	\begin{equation}\label{key}
		\Dt_* \coloneqq \pd_t + \DD_{\vv_*}.
	\end{equation}
	It suffices to define
	\begin{equation}\label{def vv_*}
		\vv_* \coloneqq \qty(\DD\Phi_\Gt)^{-1}\qty(\vv \circ \Phi_\Gt - \pd_t\gt \vnu).
	\end{equation}
	Due to the fact that
	\begin{equation*}\label{key}
		\theta = \qty(\pd_t \gamma_\Gt \vnu)\circ\qty(\Phi_\Gt)^{-1} \vdot \vn,
	\end{equation*}
	it is clear that $\vv_*$ is tangential to $\Gs$.
	
	\subsubsection*{Variational Estimates}
	In order to compute the variation of $ \vv_* $, one can first assume that both $ \ka $ and $ \vom_* $ are parameterized by $ \beta $. Applying $ \pdv*{\beta} $ to the identity
	\begin{equation*}
		\qty(\DD \Phi_\Gt) \vdot \vv_{*} = \vv \circ \Phi_\Gt - (\pd_t\gamma_\Gt)\vnu,
	\end{equation*}
	one has
	\begin{equation*}
		\DD\qty(\pd_\beta \gt \vnu) \vdot \vv_{*} + (\DD\Phi_\Gt) \vdot \pd_\beta\vv_{*} = \pd_\beta\qty(\vv \circ \Phi_\Gt) - \qty(\pd^2_{t\beta}\gt)\vnu,
	\end{equation*}
	where
	\begin{equation*}
		\pd_\beta (\vv \circ \Phi_\Gt) = \qty(\pd_\beta \vv + \DD_{(\pd_\beta \gt \vnu)\circ (\Phi_\Gt)^{-1}} \vv) \circ \Phi_\Gt.
	\end{equation*}
	Denote by $ \vb*{\mu} \coloneqq (\pd_\beta \gt \vnu)\circ (\Phi_\Gt)^{-1}  $ and use the notation
	\begin{equation}\label{def Dbt}
		\Dbt \coloneqq \pd_\beta + \DD_{\h_+\vb*{\mu}}.
	\end{equation}
	Then
	\begin{equation}\label{eqn pd beta v+*}
		\pd_\beta \vv_{*} = (\DD \Phi_\Gt)^{-1} \vdot \qty[\qty(\Dbt\vv)\circ\Phi_\Gt - \qty(\pd^2_{t\beta}\gt)\vnu -\DD(\pd_\beta\gt\vnu)\vdot\vv_{*} ].
	\end{equation}
	In particular, $ \qty[(\DD\Phi_\Gt) \vdot \pd_\beta\vv_{*}] \circ \Phi_\Gt^{-1} \vdot \vn \equiv 0 $, so $ \pd_\beta \vv_{*} \in \mrm{T}\Gs $.
	
	Applying $\Dbt$ to \eqref{sys div-curl v} and utilizing the commutator estimates together with the div-curl estimates, one can derive that for $ \frac{1}{2} \le \sigma \le \kk -\frac{3}{2} $,
	\begin{equation}\label{est pd beta v+*}
		\begin{split}
			\abs{\pd_\beta \vv_{*}}_{\H{\sigma}} \lesssim_{\Lambda_*} &\abs{\pd_\beta \gt}_{\H{\sigma + 1}} \qty(\norm{\vom_{*}}_{H^{\kk-1}(\Om_*^+)} + \abs{\pd_t \gt }_{\H{\kk-\frac{1}{2}}}) \\
			&+ \norm{\pd_\beta\vom_{*}}_{H^{\sigma-\frac{1}{2}}(\Om_*^+)} + \abs{\pd^2_{t\beta}\gt}_{\H{\sigma}}.
		\end{split}
	\end{equation}
	Recalling that $ \gt = \K^{-1}(\ka) $, one has
	\begin{equation}\label{eqn pd gt}
		\pd_\beta \gt = \var{\K^{-1}}(\ka) [\pd_\beta\ka],
	\end{equation}
	and
	\begin{equation}\label{eqn pd2 gt}
		\pd^2_{t\beta}\gt = \var{\K^{-1}}(\ka)[\pd^2_{t\beta}\ka] +  \var[2]{\K^{-1}}(\ka)[\pd_t\ka, \pd_\beta\ka].
	\end{equation}
	Thereby, one can derive from the linear relations that there exist three linear operators $ \opb(\ka) $, $ \opf(\ka) $, and $ \opg(\ka, \pd_t\ka, \vom_{*}) $, whose ranges are all in $ \mrm{T}\Gs $, so that
	\begin{equation}\label{eqn pd beta v}
		\pd_\beta \vv_{*} = \opb(\ka)\pd^2_{t\beta}\ka + \opf(\ka)\pd_\beta\vom_{*} + \opg(\ka, \pd_t \ka, \vom_{*})\pd_\beta\ka.
	\end{equation}
	Moreover, the following lemma holds (c.f. \cite{Liu-Xin2023} for the proof):
	\begin{lem}\label{lem3.1}
		Suppose that $ a \ge a_0 $ and $ \ka \in B_{\delta_1} \subset \H{\kk-\frac{5}{2}} $, where $ a_0 $ and $ B_{\delta_1} $ are given in Proposition \ref{prop K}. If   
		$ s'-2 \le s'' \le s' \le \kk-\frac{3}{2} $, $ s' \ge \frac{1}{2}  $, and $ \frac{1}{2} \le s \le \kk-\frac{3}{2} $, the following estimates hold:
		\begin{equation}\label{est B}
			\abs{\opb(\ka)}_{\LL\qty(\H{s''}; H^{s'}(\Gs; \mrm{T}\Gs)) } \le C_* a^{s'-s''-2},
		\end{equation}
		\begin{equation}\label{est var B}
			\abs{\var\opb(\ka)}_{\LL\qty[\H{\kk-\frac{5}{2}}; \LL\qty(\H{s-2}; \H{s})]} \le C_*,
		\end{equation}
		\begin{equation}\label{est F}
			\abs{\opf(\ka)}_{\LL\qty(H^{s-\frac{1}{2}}(\Om_*^+); \H{s})} \le C_*,
		\end{equation}
		and
		\begin{equation}\label{est var F}
			\abs{\var{\opf}(\ka)}_{\LL\qty[\H{\kk-\frac{5}{2}};  \LL\qty(H^{s-\frac{1}{2}}(\Om_*^+); \H{s})]} \le C_*,
		\end{equation}
		where $ C_* $ is a constant depending only on $ \Lambda_{*} $. 
		
		Moreover, if $ \ka \in B_{\delta_1} \cap \H{\kk-\frac{3}{2}} $, $ \pd_t \ka \in \H{\kk-\frac{5}{2}} $, and $ \vom_* \in H^{\kk-1}(\Om\setminus\Gs) $, then for $  \sigma'-2 \le \sigma'' \le \sigma' \le \kk + \frac{1}{2} $ and $ \sigma' \ge \frac{1}{2} $, it holds that
		\begin{equation}\label{est B version2}
			\abs{\opb(\ka)}_{\LL\qty[\H{\sigma''}; \H{\sigma'}]} \le a^{\sigma'-\sigma''-2} Q\qty(\abs{\ka}_{\H{\kk-\frac{3}{2}}}),
		\end{equation}
		\begin{equation}\label{est G}
			\begin{split}
				\abs{\opg(\ka, \pd_t\ka, \vom_{*})}_{\LL\qty[\H{s-1}; \H{s}]}  \le Q\qty(\abs{\pd_t\ka}_{\H{\kk-\frac{5}{2}}}, \norm{\vom_{*}}_{H^{\kk-1}(\Om_*^+)}),
			\end{split}
		\end{equation}
		and for $ \frac{1}{2} \le \sigma \le \kk-\frac{5}{2} $,
		\begin{equation}\label{est var G}
			\begin{split}
				&\hspace{-2em}\abs{\var{\opg}(\ka, \pd_t\ka, \vom_{*})}_{\LL\qty[\H{\sigma-1}\times \H{\sigma-2}\times H^{\sigma-\frac{1}{2}}(\Om_*^+); \LL\qty(\H{\sigma-1}; \H{\sigma}) ]} \\
				\le\, &Q\qty(\abs{\pd_t\ka}_{\H{\kk-\frac{5}{2}}}, \norm{\vom_{*}}_{H^{\kk-1}(\Om_*^+)}),
			\end{split}
		\end{equation}
		where $ Q $ represents a generic polynomial depending only on $ \Lambda_* $.
	\end{lem}
	
	\subsection{Evolution Equation for the Modified Mean Curvature}\label{sec eqn kappa}
	Notice that the effective pressure $p$ can be decomposed into
	\begin{equation}\label{decom pressure'}
		p = p_{\vv, \vv} - p_{\vh, \vh} + \alpha^2 \h_+ \kappa + \h_+ \qty(\txfrac{1}{2}\abs*{\vhh}^2),
	\end{equation}
	where $ p_{\vb{a}, \vb{b}} $ is defined by:
	\begin{equation}\label{def p_ab^+}
		\begin{cases*}
			\lap p_{\vb{a}, \vb{b}} = - \tr(\DD\vb{a} \vdot \DD\vb{b}) \qfor x \in\Omega_t^+,\\
			p_{\vb{a}, \vb{b}} = 0 \qfor x \in \Gt;
		\end{cases*}
	\end{equation}
	for solenoidal vector fields $ \vb{a} $ and $ \vb{b} $.
	
	If $ (\Gt, \vv, \vh, \vhh) $ is a solution to \eqref{MHD}-\eqref{BC}, then, by substituting \eqref{MHD} and \eqref{decom pressure'} into  \eqref{eqn Dt2 kappa alt}, one obtains by direct computations that
	\begin{equation}\label{eqn dt2 kappa}
		\Dt^2 \kappa - \alpha^2 \lap_{\Gt}\n_+\kappa - \DD^\top_{\vh}\DD^\top_{\vh}\kappa - \DD^\top_{\vhh}\DD^\top_{\vhh}\kappa = \mathfrak{R},
	\end{equation}
	where
	\begin{equation}\label{def R'}
		\begin{split}
			\mathfrak{R} =\, &\lap_{\Gt}\qty[\DD_{\vn}(p_{\vv,\vv} - p_{\vh,\vh}) + \qty(\n_+ - \n_-)\hh] \\
			&-\lap_\Gt \qty[\vn \vdot \qty(\grad\h_- \hh - \grad \hh )] \\
			&+\lap_\Gt\qty[\II(\vh, \vh) + \II(\vhh, \vhh)] - \DD_{\vh}\DD_{\vh}\kappa  -\DD_{\vhh}\DD_{\vhh}\kappa \\
			&+ \abs{\II}^2 \qty[\alpha^2\n_+\kappa + \n_+\hh + \DD_{\vn}(p_{\vv,\vv}-p_{\vh,\vh}) + \II(\vh,\vh)] \\
			&+ \grad^\top\kappa \vdot \qty(\DD^\top_\vh \vh - \alpha^2 \grad^\top \kappa - \grad^\top \hh) + 2\vn\vdot(\grad\vv)\vdot(\lap_\Gt\vv)^\top \\
			&+ 4 \ip{\vb{A}}{\vn\vdot\qty(\DD^\top)^2\vv}
			-\kappa \abs{\qty[(\grad\vv)^*\vdot\vn]^\top}^2 -2\ip{\II}{\DD^\top \vv \vdot \DD^\top \vv} \\
			&+ 4 \ip{\II \vdot\vb{A} + \vb{A}\vdot\II}{\vb{A}},
		\end{split}
	\end{equation}
	and
	\begin{equation}\label{eqn vb A}
		2\vb{A} = (\DD\vv)^\top+\qty[(\DD\vv)^\top]^*.
	\end{equation}
	It follows from \eqref{simons' identity}, Lemma \ref{est ii}, Lemma \ref{lem lap-n}, and the product estimates that there is a generic polynomial $ Q $ determined by $ \Lambda_{*} $, such that
	\begin{equation}\label{est R}
		\abs{\mathfrak{R}}_{H^{\kk-\frac{5}{2}}(\Gt)} \le
		\begin{cases*}
			Q\qty(\abs{\ka}_{\H{\kk-1}}, \norm{\vv}_{H^{\kk}(\Om_t^+)}, \norm{\vh}_{H^{\kk}(\Om_t^+)}, \norm{\vhh}_{H^{\kk}(\Om_t^-)}) &if $ k = 2 $, \\
			Q\qty(\abs{\ka}_{\H{\kk-\frac{3}{2}}}, \norm{\vv}_{H^{\kk}(\Om_t^+)}, \norm{\vh}_{H^{\kk}(\Om_t^+)}, \norm{\vhh}_{H^{\kk}(\Om_t^-)}) &if $ k \ge 3 $.
		\end{cases*}
	\end{equation}
	
	To derive the evolution of $\ka$, it is natural to pull back $\Dt \kappa$ to $\Gs$ via $\Phi_\Gt$. For the simplicity of notations, we set
	\begin{equation}\label{key}
		\opA(\ka) f \coloneqq \qty[-\lap_{\Gamma}\n_+ (f \circ \Phi_\Gamma^{-1})] \circ \Phi_\Gamma,
	\end{equation}
	and
	\begin{equation}\label{key}
		\opR(\ka, \vb{U}_*) f \coloneqq \qty[\DD^\top_{\vb{U}} \DD^\top_{\vb{U}} (f \circ \Phi_{\Gamma}^{-1})]\circ\Phi_\Gamma,
	\end{equation}
	where $ \Gamma \in \Lambda_{*} $, $ f : \Gs \to \R $, $ \vb{U}_* \in \mathrm{T}\Gs $, and $ \vb{U} \coloneqq \mathrm{T}\Phi_\Gamma (\vb{U}_*) \in \mathrm{T}\Gamma $.
	Moreover, the following lemma holds (c.f. \cite{Liu-Xin2023} for a proof):
	\begin{lem}\label{lem 3.3}
		There are positive constants $ C_*, \delta_1 $ depending only on $ \Lambda_* $, so that for $ \ka \in B_{\delta_1} \subset \H{\kk-\frac{5}{2}} $, $ \vb{U}_* \in \H{\kk-\frac{1}{2}} $, $ 2 \le s \le \kk - \frac{1}{2} $, $ 1 \le \sigma \le \kk-\frac{1}{2} $ and $ 2 \le s_1 \le \kk - \frac{1}{2} $, the following estimates hold:
		\begin{equation}\label{est opA}
			\abs{\opA(\ka)}_{\LL\qty[\H{s}; \H{s-3}]} \le C_*,
		\end{equation}
		\begin{equation}\label{est opR}
			\abs{\opR(\ka, \vb{U}_*)}_{\LL\qty[\H{\sigma}; \H{\sigma-2}]} \le C_* \abs{{\vb{U}}_*}_{\H{\kk-\frac{1}{2}}}^2,
		\end{equation}
		and
		\begin{equation}\label{est var opA}
			\abs{\var{\opA}(\ka)}_{\LL\qty[\H{\kk-\frac{5}{2}}; \LL\qty(\H{s_1}; \H{s_1-3})]} \le C_*.
		\end{equation}
		Furthermore, if $ k \ge 3 $, it holds for $ 2 \le s_2 \le \kk-1 $ that
		\begin{equation}\label{est var opR}
			\begin{split}
				\abs{\var{\opR}(\ka, {\vb{U}}_*)}_{\LL\qty[\H{\kk-\frac{5}{2}}\times\H{\kk-2}; \LL\qty(\H{s_2}; \H{s_2-2})]} 
				\le C_*\qty(1 +\abs{{\vb{U}}_*}^2_{\H{\kk-\frac{1}{2}}}).
			\end{split}
		\end{equation}
	\end{lem}
	
	We now define a vector field $ \vW : \Gt \to \R^3 $ by
	\begin{equation}\label{def vW"}
		\begin{split}
			\vW &\coloneqq \Dt \vv - \DD_\vh \vh + \grad\qty(p_{\vv,\vv}-p_{\vh,\vh}+\alpha^2\h_+\kappa + \h_+\hh) \\
			&\equiv \Dt \vv + \va{\mathfrak{b}},
		\end{split}
	\end{equation}
	then $ \vW \equiv 0 $ if $ (\Gt, \vv, \vh, \vhh) $ is a solution to \eqref{MHD}-\eqref{BC}.
	Substituting 
	\begin{equation*}
		\Dt \vv = \vW - \va{\mathfrak{b}}
	\end{equation*}
	into \eqref{eqn Dt2 kappa alt} and pulling back the equation to $\Gs$ via \eqref{pull back Dt to Gs}, one can calculate that
	\begin{equation}\label{eqn dt2 ka}
		\begin{split}
			&\Dt_*^2\qty(\kappa \circ \Phi_\Gt) + \alpha^2 \opA(\ka)(\kappa \circ \Phi_\Gt) - \opR(\ka, \vh_*)(\kappa\circ\Phi_\Gt) - \opR(\ka,\vhh_*)(\kappa\circ\Phi_\Gt) \\
			&\quad = \qty[\mathfrak{R} - \lap_\Gt\qty(\vW\vdot\vn) + \vW\vdot\lap_\Gt\vn]\circ\Phi_\Gt,
		\end{split}
	\end{equation}
	where $\vv_*$ and $\mathfrak{R}$ are given by \eqref{def vv_*} and \eqref{def R'} respectively,
	\begin{equation*}
		\vh_* \coloneqq (\DD\Phi_\Gt)^{-1}\vdot\qty(\vh\circ\Phi_\Gt), \qand \vhh_* \coloneqq (\DD\Phi_\Gt)^{-1}\vdot(\vhh\circ\Phi_\Gt).
	\end{equation*}
	
	Due to the relation
	\begin{equation*}
		\Dts^{\!2} = \pd^2_{tt} + \DD_{\vv_*}\DD_{\vv_*} + 2 \DD_{\vv_*}\pd_t + \DD_{\pd_t \vv_*}
	\end{equation*}
	and (\ref{eqn pd beta v}), the term $ \pd_t \vv_* $ involves $ \pd^2_{tt} \ka $, so (\ref{eqn dt2 ka}) is a nonlinear equation for $ \pd^2_{tt}\ka $. In order to get one which is linear for $ \pd^2_{tt}\ka $, one may drive from (\ref{eqn pd beta v}) that
	\begin{equation}\label{eqn pdt2 kappa+}
		\begin{split}
			&\pd^2_{tt}(\kappa \circ \Phi_\Gt) + \opC_\alpha(\ka, \pd_t\ka, \vv_*, \vh_*, \vhh_*)(\kappa \circ \Phi_\Gt) \\
			&+ \grad^\top(\kappa \circ \Phi_\Gt) \vdot \qty[\opb(\ka)\pd^2_{tt}\ka + \opf(\ka)\pd_t\vom_*+\opg(\ka, \pd_t\ka, \vom_*)\pd_t\ka] \\
			&\quad =\qty{\mathfrak{R}  - \lap_\Gt \qty(\vW \vdot \vn) + \vW \vdot \lap_\Gt \vn} \circ\Phi_\Gt,
		\end{split}
	\end{equation}
	where the following notation has been used:
	\begin{equation}\label{def opC"}
		\opC_\alpha \coloneqq 2\DD_{\vv_*}\pd_t + \DD_{\vv_*}\DD_{\vv_*} + \alpha^2 \opA(\ka) - \opR(\ka, \vh_*) - \opR(\ka, \vhh_*).
	\end{equation}
	
	Since $ \ka = \kappa \circ \Phi_\Gt + a^2 \gt $, one also needs to calculate $ \pd^2_{tt}\gt $. Note that the following relation holds for the evolution velocity $ \vv : \Gt \to \R^3 $ of $ \Gt $:
	\begin{equation*}\label{key}
		\pd_t \gt \vnu \vdot (\vn \circ \Phi_\Gt) = (\vv \vdot \vn) \circ \Phi_\Gt,
	\end{equation*}
	which implies (c.f. \cite{Liu-Xin2023} for details)
	\begin{equation}\label{eqn pd2 tt gt}		
		\begin{split}
			\pd^2_{tt}\gt = \dfrac{(\vn \circ \Phi_\Gt)}{\vnu \vdot (\vn \circ \Phi_\Gt)} \vdot \qty[\qty(\vW - \va{\mathfrak{b}})\circ\Phi_\Gt - \DD_{\vv_*}\qty(\vv\circ\Phi_\Gt + \pd_t\gt\vnu) ].
		\end{split}
	\end{equation}
	In particular, the term $ \pd^2_{tt}\gt $ does not involve $ \pd^2_{tt}\ka $.
	
	The combination of (\ref{def ka}) and (\ref{eqn pdt2 kappa+}) yields
	\begin{equation}\label{eqn pdt2 ka premitive}
		\begin{split}
			&\qty[\mrm{Id} + \grad^\top(\kappa \circ \Phi_\Gt) \vdot \opb(\ka) ]\pd^2_{tt}\ka + \opC_\alpha(\ka, \pd_t\ka, \vv_*, \vh_*, \vhh_*) \ka \\ &+ \grad^\top(\kappa \circ \Phi_\Gt) \vdot \qty[\opf(\ka)\pd_t\vom_* + \opg(\ka, \pd_t\ka, \vom_*)\pd_t\ka] \\
			&+ a^2 \dfrac{\vn \circ \Phi_\Gt}{\vnu \vdot (\vn \circ \Phi_\Gt)} \vdot \qty[\va{\mathfrak{b}} \circ \Phi_\Gt + \DD_{\vv_*}\qty(\vv\circ\Phi_\Gt + \pd_t\gt\vnu)]-a^2 \opC_\alpha(\ka, \pd_t\ka, \vv_*, \vh_*, \vhh_*)\gt \\
			&\quad = \mathfrak{R} \circ \Phi_\Gt + \qty[-\lap_\Gt \qty(\vW \vdot \vn) + \vW \vdot \lap_\Gt \vn + a^2 \dfrac{\vW \vdot \vn}{\vn \vdot (\vnu \circ \Phi_\Gt^{-1})}] \circ \Phi_\Gt.
		\end{split}
	\end{equation}
	Define a new operator:
	\begin{equation}\label{def opB}
		\opB(\ka) \coloneqq \grad^\top(\kappa \circ \Phi_\Gt) \vdot \opb(\ka).
	\end{equation}
	Then according to Lemma \ref{lem3.1}, the following estimate holds:
	\begin{equation}\label{est opB}
		\abs{\opB(\ka)}_{\LL\qty[\H{s''}; \H{s'}]} \less a^{s'-s''-2+\epsilon} \abs{\ka}_{\H{\kk-1}},
	\end{equation}
	for $ s'-2 \le s'' \le s' \le \kk - 2 $, $ s' \ge \frac{1}{2} $, and $ 0 < \epsilon \le s'' -s' +2 $. If $ k \ge 3 $, one may take $ \epsilon = 0 $ and for $ \sigma'-2\le \sigma'' \le \sigma' \le \kk-\frac{5}{2} $, $ \sigma' \ge \frac{1}{2} $,
	\begin{equation}\label{est opB'}
		\abs{\opB(\ka)}_{\LL\qty[\H{\sigma''}; \H{\sigma'}]} \less a^{\sigma'-\sigma''-2} \abs{\ka}_{\H{\kk-\frac{3}{2}}}. \tag{\ref{est opB}'}
	\end{equation}
	
	Letting $ s'=s'' $, $ 0 < \epsilon < \frac{1}{2} $ ($ \epsilon = 0 $ if $ k \ge 3 $) and $ a_0 $ large enough compared to $ \abs{\ka}_{\H{\kk-1}} $ (or $ a_0 $ large compared to $ \abs{\ka}_{\H{\kk-\frac{3}{2}}} $ if $ k \ge 3 $), one has
	\begin{equation}\label{est opcal b}
		\abs{\opB(\ka)}_{\LL(\H{s'})} \le \frac{1}{2} < 1,
	\end{equation}
	for $ \frac{1}{2} \le s' \le \kk-2 $ (or $ \frac{1}{2} \le s' \le \kk - \frac{5}{2} $ if $ k \ge 3 $). Namely, $ \qty[\mathrm{I}+\opB(\ka)] $ is an isomorphism on $ \H{s'} $. 
	Define
	\begin{equation}\label{key}
		\vj_* \coloneqq (\curl \vh) \circ \X_\Gt^+.
	\end{equation}
	Then $ \vh $ can be obtained from $ (\ka, \vj_*) $ via solving the corresponding div-curl problems. Similarly, $ \vhh $ can be uniquely determined by $ \ka $ and $ \vu{J} $.
	Applying the operator $ \qty[\mathrm{I}+\opB(\ka)]^{-1} $ to (\ref{eqn pdt2 ka premitive}), one can get the evolution equation for $ \ka $ as (which is, in particular, irrelevant to the original plasma-vacuum problems):
	\begin{equation}\label{eqn pdt2 ka "}
		\begin{split}
			&\pd^2_{tt} \ka + \opC_\alpha (\ka, \pd_t\ka, \vv_*, \vh_*, \vhh_*)\ka - \opF(\ka)\pd_t\vom_* - \opG(\ka, \pd_t\ka, \vom_*, \vj_*, \vu{J}) \\
			&\quad = \qty[\mathrm{I} + \opB(\ka)]^{-1}\qty{\qty[-\lap_\Gt \qty(\vW \vdot \vn) + \vW \vdot \lap_\Gt \vn + a^2 \dfrac{\vW \vdot \vn}{\vn \vdot (\vnu \circ \Phi_\Gt^{-1})}] \circ \Phi_\Gt},
		\end{split}
	\end{equation}
	The operators $ \opF $ and $ \opG $ defined above satisfy the following lemma, whose proof follows form the same arguments as those in \cite{Liu-Xin2023}:
	\begin{lem}\label{lem 3.4}
		Assume that $ a \ge a_0 $ and $ \ka \in B_{\delta_1} $ as in Proposition \ref{prop K}. For $ \frac{1}{2} \le s \le \kk-2 $ and $ \epsilon > 0 $, there are positive constants $ C_* $ and generic polynomials $ Q $ determined by $ \Lambda_* $, so that
		\begin{equation}\label{est opF"}
			\abs{\opF(\ka)}_{\LL\qty[H^{s+\epsilon-\frac{1}{2}}(\OGs); \H{s}]} \le C_* \abs{\ka}_{\H{\kk-1}},
		\end{equation}
		and
		\begin{equation}\label{est opG"}
			\begin{split}
				&\hspace{-1em}\abs{\opG(\ka, \pd_t\ka, \vom_*, \vj_*, \vu{J})}_{\H{\kk-\frac{5}{2}}} \\ \le\, &a^2 Q\qty(\abs{\ka}_{\H{\kk-1}}, \abs{\pd_t\ka}_{\H{\kk-\frac{5}{2}}}, \norm{\vom_*}_{H^{\kk-1}(\Om_*^+)}, \norm{\vj_*}_{H^{\kk-1}(\Om_*^+)}, \abs*{\vu{J}}_{H^{\kk-\frac{1}{2}}(\pd\Om)}).
			\end{split}
		\end{equation}
		Furthermore, if $ k \ge 3 $, for $ \frac{1}{2} \le \sigma \le \kk-\frac{5}{2} $, there hold
		\begin{equation}\label{est opF""}
			\abs{\opF(\ka)}_{\LL\qty[H^{\sigma-\frac{1}{2}}(\Om_*^+); \H{\sigma}]} \le C_* \abs{\ka}_{\H{\kk-\frac{3}{2}}},
		\end{equation}
		\begin{equation}\label{est var opF"}
			\abs{\var\opF(\ka)}_{\LL\qty[\H{\kk-\frac{5}{2}}; \LL\qty(H^{\kk-4}(\Om_*^+); \H{\kk-4})]} \le C_*,
		\end{equation}
		\begin{equation}\label{est opG""}
			\begin{split}
				&\hspace{-1em}\abs{\opG(\ka, \pd_t\ka, \vom_*, \vj_*, \vu{J})}_{\H{\kk-\frac{5}{2}}} \\ &\le a^2 Q\left(\alpha\abs{\ka}_{\H{\kk-1}}, \abs{\ka}_{\H{\kk-\frac{3}{2}}}, \abs{\pd_t\ka}_{\H{\kk-\frac{5}{2}}}, \norm{\vom_*}_{H^{\kk-1}(\Om_*^+)}, \right. \\
				&\hspace{20em} \left. \norm{\vj_*}_{H^{\kk-1}(\Om_*^+)}, \abs*{\vu{J}}_{H^{\kk-\frac{1}{2}}(\pd\Om)} \right),
			\end{split}
		\end{equation}
		and
		\begin{equation}\label{est var opG"}
			\begin{split}
				&\hspace{-1em}\abs{\var'\opG}_{\LL\qty[\H{\kk-\frac{5}{2}}\times\H{\kk-4}\times H^{\kk-\frac{5}{2}}(\Om_*^+) \times H^{\kk-\frac{5}{2}}(\Om_*^+); \H{\kk-4}]} \\
				\le\, &a^2 Q\qty(\abs{\ka}_{\H{\kk-\frac{3}{2}}}, \abs{\pd_t\ka}_{\H{\kk-\frac{5}{2}}}, \norm{\vom_*}_{H^{\kk-1}(\Om_*^+)}, \norm{\vj_*}_{H^{\kk-1}(\Om_*^+)}, \abs*{\vu{J}}_{H^{\kk-\frac{1}{2}}(\pd\Om)}),
			\end{split}
		\end{equation}
		where $ \var'\opG $ is the variational derivative with respect to the first four variables.
	\end{lem}
	
	\subsection{Evolution Equations for the Current and Vorticity}\label{sec curr-vorticity eqn}
	Since the vorticity, current and the corresponding boundary conditions can be used to recover the vector fields $ \vv $ and $ \vh $ by solving the corresponding div-curl problems, it is natural to investigate the evolution of the current and vorticity.
	
	Setting
	\begin{equation*}
		\vom \coloneqq \curl \vv \qc
		\vj \coloneqq \curl \vh.
	\end{equation*}
	Then taking curl of the equations \eqref{MHD} yields
	\begin{equation}\label{eqn pdt vom}
		\pd_t \vom + (\vv \vdot \grad)\vom - (\vh \vdot \grad)\vj = (\vom \vdot \grad)\vv - (\vj \vdot \grad) \vh,
	\end{equation}
	and
	\begin{equation}\label{eqn pdt vj}
		\pd_t\vj + (\vv \vdot \grad)\vj - (\vh \vdot \grad) \vom = (\vj \vdot \grad)\vv - (\vom \vdot \grad) \vh - 2\tr(\grad \vv \cp \grad \vh),
	\end{equation}
	where in the Cartesian coordinate
	\begin{equation*}
		\tr(\grad \vv \cp \grad \vh ) = \sum_{l=1}^3 \grad v^l \cp \grad h^l.
	\end{equation*}

	\section{Linear Estimates}
	Suppose that $ \Gs \in H^{\kk+1} $ ($ k \ge 2 $) is a reference hypersurface, and $ \Lambda_* $ defined by (\ref{def lambda*}) satisfies all the properties given in the preliminary.
	\subsection{Linear Problems for the Modified Mean Curvature}\label{sec linear ka}
	Now, assume that there are a family of hypersurfaces $ \Gt \in \Lambda_* $ and three tangential vector fields $ \vv_*, \vh_{*}, \vhh_* : \Gs \to \mathrm{T}\Gs $ satisfying:
	\begin{equation}\label{key}
		\ka \in C^0\qty{[0, T]; \H{\kk-1}} \cap C^1\qty{[0, T]; B_{\delta_1} \subset \H{\kk-\frac{5}{2}}}, 
	\end{equation}
	and
	\begin{equation}\label{key}
		\vv_{*}, \vh_{*}, \vhh_* \in C^0\qty{[0, T]; \H{\kk-\frac{1}{2}}} \cap C^1\qty{[0, T]; \H{\kk-2}}. 
	\end{equation}
	
	Define three positive constants $ L_0, L_1$, and $ L_2 $ to be
	\begin{equation}
		L_0 \coloneqq \abs{\vv_*(0)}_{\H{\kk-2}},
	\end{equation}
	\begin{equation}\label{key}
		L_1 \coloneqq \sup_{t\in[0, T]} \qty{\abs{\ka (t)}_{\H{\kk-1}},  \abs{\pd_t \ka(t)}_{\H{\kk-\frac{5}{2}}}, \abs{\qty(\vv_{*}(t), \vh_{*}(t), \vhh_*(t))}_{\H{\kk-\frac{1}{2}}}},
	\end{equation}
	and
	\begin{equation}\label{key}
		L_2 \coloneqq \sup_{t\in [0, T]} \abs{\qty(\pd_t\vv_{*}(t), \pd_t\vh_{*}(t), \pd_t\vhh_*(t))}_{\H{\kk-2}}.
	\end{equation}
	Consider the following linear initial value problem:
	\begin{equation}\label{eqn linear 1"}
		\begin{cases}
			\pd^2_{tt} \f + \opC(\ka, \pd_t\ka, \vv_*, \vh_*, \vhh_*) \f = \g, \\
			\f(0) = \f_0, \quad \pd_t \f(0) = \f_1,
		\end{cases}
	\end{equation}
	where $ \f_0, \f_1, \g(t) : \Gs \to \R $ are three given functions, and $ \opC $ is given by:
	\begin{equation*}
		\opC(\ka, \pd_t\ka, \vv_*, \vh_*, \vhh_*) \coloneqq 2\DD_{\vv_*}\pd_t + \DD_{\vv_*}\DD_{\vv_*} + \opA(\ka) - \opR(\ka, \vh_*) - \opR(\ka, \vhh_*),
	\end{equation*}
	which is exactly \eqref{def opC"} with $ \alpha = 1 $.
	For any integer $0 \le l \le k-2$, the energy functional is defined to be:
	\begin{equation}\label{key}
		\begin{split}
			E_l (t, \f, \pd_t\f) \coloneqq \int_\Gt &\abs{\qty(-\n_+^\frac{1}{2}\lap_\Gt\n_+^\frac{1}{2})^\frac{l}{2}\n_+^\frac{1}{2}\qty[(\pd_t\f + \DD_{\vv_*}\f) \circ \Phi_\Gt^{-1}]}^2 \\
			&+ \abs{\qty(-\n_+^\frac{1}{2}\lap_\Gt\n_+^\frac{1}{2})^{\frac{1+l}{2}}\n_+^\frac{1}{2}(\f \circ \Phi_\Gt^{-1})}^2 \\
			&+ \abs{\qty(-\n_+^\frac{1}{2}\lap_\Gt\n_+^\frac{1}{2})^\frac{l}{2}\n_+^\frac{1}{2}\qty[(\DD_{\vh_*}\f) \circ \Phi_\Gt^{-1}]}^2 \\
			&+ \abs{\qty(-\n_+^\frac{1}{2}\lap_\Gt\n_+^\frac{1}{2})^\frac{l}{2}\n_+^\frac{1}{2}\qty[(\DD_{\vhh_*}\f) \circ \Phi_\Gt^{-1}]}^2 \dd{S_t}.
		\end{split}
	\end{equation}
	
	Then, the following lemma holds:
	\begin{lem}\label{lem est E_l}
		For any integer $ 0 \le  l \le k-2 $, and $ 0 \le t \le T $, it holds that
		\begin{equation}\label{est E_l}
			\begin{split}
				&\hspace{-2em}E_l (t, \f, \pd_t\f) - E_l (0, \f_0, \f_1) \\
				\le&Q(L_1, L_2)\int_0^t \qty(\abs{\f(s)}_{\H{\frac{3}{2}l + 2}} + \abs{\pd_t\f(s)}_{\H{\frac{3}{2}l + \frac{1}{2}}} + \abs{\g(s)}_{\H{\frac{3}{2}l + \frac{1}{2}}}) \times \\
				&\hspace{8em} \times \qty(\abs{\f(s)}_{\H{\frac{3}{2}l + 2}} + \abs{\pd_t\f(s)}_{\H{\frac{3}{2}l + \frac{1}{2}}}) \dd{s},
			\end{split}
		\end{equation}
		where $ Q $ is a generic polynomial determined by $ \Lambda_* $.
	\end{lem}
	\begin{proof}
		We first introduce some notations:
		\begin{equation*}
			\Dt \coloneqq \pd_t + \DD_\vmu, \qq{with} \vmu \coloneqq (\pd_t\gt \vnu) \circ \Phi_\Gt^{-1} : \Gt \to \R^3 ,
		\end{equation*} 
		and for any function $ f : \Gs \to \R $ and vector field $ \vb{a}_*: \Gs \to \mathrm{T}\Gs $, set
		\begin{equation*}
			\bar{f} \coloneqq f \circ \Phi_\Gt^{-1} : \Gt \to \R, \quad \vb{a} \coloneqq \qty(\DD\Phi_\Gt \cdot \vb{a}_*) \circ \Phi_\Gt^{-1} : \Gt \to \mathrm{T}\Gt.
		\end{equation*} 
		Thus,
		\begin{equation*}
			(\pd_t\f) \circ \Phi_\Gt^{-1} = \Dt \bar{\f}, \quad \qty(\DD_{\vb{a}_*}\f) \circ\Phi_\Gt^{-1} = \DD_{\vb{a}} \bar{\f},
		\end{equation*}
		and the equation in \eqref{eqn linear 1"} can be rewritten as:
		\begin{equation}\label{eqn lin'}
			\Dt^2\baf + 2\DD_\vv\Dt\baf + \DD_\vv\DD_\vv \baf + (-\lap_\Gt\n_+)\baf - \DD_\vh\DD_\vh \baf - \DD_\vhh\DD_\vhh \baf = \bar{\g}.
		\end{equation}
		
		Furthermore, we will use the conventions that
		\begin{equation*}\label{key}
			\abs{f}_{s} \coloneqq \abs{f}_{H^s(\Gt)},
		\end{equation*}
		and
		\begin{equation*}
			\mathfrak{u} \lesq \mathfrak{v}
		\end{equation*}
		if there exists a generic polynomial $ Q = Q(L_1, L_2) $ determined by $\Lambda_{*}$, such that $ \mathfrak{u} \le Q(L_1, L_2) \mathfrak{v} $.
		For simplicity, set
		\begin{equation}
			\mathcal{D} \coloneqq \qty(-\n_+^{\frac{1}{2}}\lap_\Gt\n_+^{\frac{1}{2}})^{\frac{1}{2}},
		\end{equation}
		which is a $\frac{3}{2}$-th order self-adjoint differential operator on $\Gt$.
		
		It follows from Lemma \ref{Dt comm est lemma} that the following estimate holds (see \cite[Lemma 5.1]{Liu-Xin2023} for more detailed calculations):
		\begin{equation}\label{lin est 1}
			\abs{\frac{1}{2}\dv{t}\int_\Gt \abs{\opd^l\n_+^\frac{1}{2}\qty(\Dt\baf+\DD_\vv\baf)}^2 \dd{S_t} - I } \lesq \abs{\vmu}_{\kk-\frac{1}{2}} \times \qty(\abs{\f}_{\H{\frac{3}{2}l+\frac{3}{2}}} + \abs{\pd_t\f}_{\H{\frac{3}{2}l+\frac{1}{2}}}),
		\end{equation}
		where
		\begin{equation}\label{eqn I}
			I = \int_\Gt \opd^l \n_+^\frac{1}{2}\qty(\Dt^2\baf + \Dt\DD_\vv\baf) \times \opd^l\n_+^{\frac{1}{2}}\qty(\Dt\baf + \DD_\vv\baf) \dd{S_t}.
		\end{equation}
		Plugging \eqref{eqn lin'} into \eqref{eqn I} yields:
		\begin{equation}\label{eqn lin I}
			\begin{split}
				I =\, &\int_\Gt \opd^l \n_+^\frac{1}{2}\qty(\Dt\baf + \DD_\vv \baf) \times \opd^l\n_+^\frac{1}{2}\g \dd{S_t} \\
				&+ \int_\Gt \opd^l \n_+^\frac{1}{2}\qty(\Dt\DD_\vv\baf - 2 \DD_\vv\Dt\baf - \DD_\vv\DD_\vv\baf) \times \opd^l\n_+^\frac{1}{2}\qty(\Dt\baf + \DD_\vv\baf) \dd{S_t} \\
				&-\int_\Gt \opd^l\n_+^\frac{1}{2}\qty(-\lap_\Gt\n_+\baf) \times \opd^l \n_+^\frac{1}{2}\qty(\Dt\baf+\DD_\vv\baf)\dd{S_t} \\
				&+\int_\Gt \opd^l\n_+^\frac{1}{2}\qty(\DD_\vh\DD_\vh\baf) \times \opd^l\n_+^\frac{1}{2}\qty(\Dt\baf + \DD_\vv\baf) \dd{S_t} \\
				&+\int_\Gt \opd^l\n_+^\frac{1}{2}\qty(\DD_\vhh\DD_\vhh\baf) \times \opd^l\n_+^\frac{1}{2}\qty(\Dt\baf +\DD_\vv\baf) \dd{S_t} \\
				\eqqcolon &\, I_1 +\cdots + I_5
			\end{split}
		\end{equation}
		It is clear that
		\begin{equation}\label{est I_1}
			\abs{I_1} \lesq \abs{\g}_{\H{\frac{3}{2}l + \frac{1}{2}}} \times \qty(\abs{\pd_t\f}_{\H{\frac{3}{2}l + \frac{1}{2}}} + \abs{\f}_{\H{\frac{3}{2}l + \frac{3}{2}}}).
		\end{equation}
		Note that for two functions $ \phi, \psi : \Gt \to \R $, the integration-by-parts formula on $\Gt$ is:
		\begin{equation}\label{formula int by parts Gt}
			\int_\Gt - (\DD_\vv \phi) \cdot \psi \dd{S_t} = \int_\Gt (\DD_\vv\psi)\cdot\phi + \phi\psi(\Div_\Gt\vv) \dd{S_t}.
		\end{equation}
		$ I_2 $ can be computed via:
		\begin{equation}\label{int DuDuf * Dtf}
			\begin{split}
				\int_\Gt \opd^l \n_+^\frac{1}{2}(\DD_\vv\DD_\vv \baf) \cdot \opd^l \n_+^\frac{1}{2}(\Dt\baf) \dd{S_t}  = \int_\Gt \n_+(\DD_\vv\DD_\vv\baf) \cdot (-\lap_\Gt\n_+)^l (\Dt\baf) \dd{S_t}.
			\end{split}
		\end{equation}
		Commuting $ \DD_\vv $ with $ \n_+ $ and $ \lap_\Gt $ yields
		\begin{equation}\label{est DuDu f * Dt f}
			\begin{split}
				&\hspace{-1em}\abs{ \int_\Gt \opd^l \n_+^\frac{1}{2} (\DD_\vv \DD_\vv \baf) \cdot \opd^l \n_+^\frac{1}{2} (\Dt\baf) + \opd^l\n_+^\frac{1}{2}(\DD_\vv\baf) \cdot \opd^l\n_+^\frac{1}{2} (\DD_\vv\Dt \baf)  \dd{S_t}} \\
				&\lesq  \abs{\f}_{\H{\frac{3}{2}l +\frac{3}{2}}} \cdot \abs{\pd_t \f}_{\H{\frac{3}{2}l+\frac{1}{2}}}.
			\end{split}
		\end{equation}		
		It follows from \eqref{formula int by parts Gt} that
		\begin{equation}\label{key}
			\begin{split}
				\abs{\int_\Gt \opd^l \n_+^\frac{1}{2} (\DD_\vv \DD_\vv \baf) \cdot \opd^l \n_+^\frac{1}{2} (\DD_\vv\baf) \dd{S_t}} \lesq\abs{\f}_{\H{\frac{3}{2}l+\frac{3}{2}}}^2,
			\end{split}
		\end{equation}
		and
		\begin{equation}
			\begin{split}
				\abs{\int_\Gt \opd^l \n_+^\frac{1}{2} (\DD_\vv \Dt \baf) \cdot \opd^l \n_+^\frac{1}{2} (\Dt \baf) \dd{S_t} } \lesssim_{L_1} \abs{\pd_t \f}_{\H{\frac{3}{2}l + \frac{1}{2}}}^2.
			\end{split}
		\end{equation}
		Furthermore, since
		\begin{equation}\label{est I2 last}
			\abs{\comm{\Dt}{\DD_\vv}\baf}_{\frac{3}{2}l+\frac{1}{2}} = \abs{\DD_{\qty(\Dt\vv-\DD_\vv \vmu)} \baf}_{\frac{3}{2}l+\frac{1}{2}} \lesq \abs{\f}_{\H{\frac{3}{2}l + \frac{7}{4}}},
		\end{equation}
		(\ref{int DuDuf * Dtf})-(\ref{est I2 last}) yield
		\begin{equation}\label{est I_2}
			\abs{I_2} \lesq \abs{\pd_t\f}_{\H{\frac{3}{2}l+\frac{1}{2}}}^2 + \abs{\f}_{\H{\frac{3}{2}l+2}}^2.
		\end{equation}
		
		$ I_3 $ can be calculated as:
		\begin{equation*}\label{key}
			\begin{split}
				I_3 = &-\int_\Gt \opd^l \n_+^\frac{1}{2} \qty(-\lap_\Gt\n_+)\baf \cdot \opd^l\n_+^\frac{1}{2} \qty(\Dt\baf + \DD_\vv \baf) \dd{S_t} \\
				= &-\int_\Gt \n_+\baf \cdot (-\lap_\Gt\n_+)^{l+1}(\Dt\baf + \DD_\vv\baf) \dd{S_t},
			\end{split}
		\end{equation*}
		which, together with the commutator estimates and integration by parts, yields
		\begin{equation}\label{est I_3}
			\begin{split}
				\abs{I_3 + \dfrac{1}{2}\dv{t}\int_\Gt \abs{\opd^{l+1}\n_+^\frac{1}{2}\baf}^2 \dd{S_t}} 
				\lesq \abs{\f}_{\H{\frac{3}{2}l + 2}}^2.
			\end{split}
		\end{equation}
		
		As for $ I_4 $, the relation that
		\begin{equation}
			\comm{\DD_\vv}{\DD_\vh} \baf = \DD_{\comm{\vv}{\vh}} \baf,
		\end{equation}
		and the integration-by-parts lead to
		\begin{equation}\label{key}
			\begin{split}
				\abs{\int_\Gt \opd^l\n_+^\frac{1}{2}(\DD_\vh\DD_\vh\baf) \cdot \opd^l\n_+^\frac{1}{2}(\DD_\vv\baf)\dd{S_t}} 
				\lesq \abs{\f}_{\H{\frac{3}{2}l+\frac{3}{2}}}^2.
			\end{split}
		\end{equation}		
		It is also not hard to show that
		\begin{equation}\label{key}
			\begin{split}
				&\hspace{-1em}\abs{ \int_\Gt \opd^l \n_+^\frac{1}{2} (\DD_\vh \DD_\vh \baf) \cdot \opd^l \n_+^\frac{1}{2} (\Dt\baf) + \opd^l\n_+^\frac{1}{2}(\DD_\vh\baf) \cdot \opd^l\n_+^\frac{1}{2} (\DD_\vh\Dt \baf)  \dd{S_t}} \\
				&\lesq  \abs{\f}_{\H{\frac{3}{2}l +\frac{3}{2}}} \times \abs{\pd_t \f}_{\H{\frac{3}{2}l+\frac{1}{2}}},
			\end{split}
		\end{equation}
		and
		\begin{equation}\label{I_4 est end}
			\begin{split}
				&\hspace{-1em}\abs{\int_\Gt \opd^l \n_+^\frac{1}{2} (\DD_\vh \Dt \baf) \cdot \opd^l \n_+^\frac{1}{2} (\DD_\vh\baf) \dd{S_t } - \dfrac{1}{2}\dv{t} \int_\Gt \abs{\opd^l\n_+^\frac{1}{2}(\DD_\vh\baf)}^2 \dd{S_t}} \\
				&\lesq\abs{\f}_{\H{\frac{3}{2}l+\frac{3}{2}}} \times \qty(\abs{\f}_{\H{\frac{3}{2}l+2}} + \abs{\pd_t\f}_{\H{\frac{3}{2}l+\frac{1}{2}}}).
			\end{split}
		\end{equation}
		Thus,
		\begin{equation}\label{est I_4}
			\begin{split}
				&\abs{I_4 + \frac{1}{2} \dv{t}\int_\Gt \abs{\opd^l\n_+^\frac{1}{2}(\DD_\vh\baf)}^2 \dd{S_t}} \\
				&\quad\lesq \abs{\f}_{\H{\frac{3}{2}l+\frac{3}{2}}} \times \qty(\abs{\f}_{\H{\frac{3}{2}l+2}} + \abs{\pd_t\f}_{\H{\frac{3}{2}l+\frac{1}{2}}}).
			\end{split}
		\end{equation}
		Similar arguments yield
		\begin{equation}\label{est I_5}
			\begin{split}
				&\abs{I_5 + \frac{1}{2} \dv{t}\int_\Gt \abs{\opd^l\n_+^\frac{1}{2}(\DD_\vhh\baf)}^2 \dd{S_t}} \\
				&\quad\lesq \abs{\f}_{\H{\frac{3}{2}l+\frac{3}{2}}} \times \qty(\abs{\f}_{\H{\frac{3}{2}l+2}} + \abs{\pd_t\f}_{\H{\frac{3}{2}l+\frac{1}{2}}}).
			\end{split}
		\end{equation}
		
		In conclusion, it follows from \eqref{lin est 1}-\eqref{est I_1}, \eqref{est I_2}-\eqref{est I_3}, and \eqref{est I_4}-\eqref{est I_5} that
		\begin{equation}\label{key}
			\begin{split}
				\abs{\dv{t} E_l} \lesq &\qty(\abs{\f}_{\H{\frac{3}{2}l+2}} + \abs{\pd_t\f}_{\H{\frac{3}{2}l+\frac{1}{2}}}) \times \\
				&\quad\times \qty(\abs{\f}_{\H{\frac{3}{2}l+2}} + \abs{\pd_t\f}_{\H{\frac{3}{2}l+\frac{1}{2}}} + \abs{\g}_{\H{\frac{3}{2}l+\frac{1}{2}}}),
			\end{split}
		\end{equation}
		which completes the proof of Lemma \ref{lem est E_l}.
	\end{proof}
	
	With the help of Lemma \ref{lem est E_l}, it is routine prove that (see \cite{Liu-Xin2023} for more details):
	\begin{prop}
		There exists a positive constant $ C_* $, depending only on $ \Lambda_* $, so that for each integer $ 0 \le l \le k-2 $, and $ 0 \le t \le T $ ($ T \le C $ for some constant $ C = C(L_1, L_2) $), the following estimate holds:
		\begin{equation}\label{est linear eqn ka}
			\begin{split}
				&\hspace{-1em}\abs{\f(t)}_{\H{\frac{3}{2}l+2}}^2 + \abs{\pd_t\f(t)}_{\H{\frac{3}{2}l+\frac{1}{2}}}^2 \\
				\le\, &C_* \exp\qty\big[Q(L_1, L_2)t]\times \\
				& \times \qty{ \abs{\f_0}_{\H{\frac{3}{2}l+2}}^2 + \abs{\f_1}_{\H{\frac{3}{2}l+\frac{1}{2}}}^2 + (L_0)^{12} \abs{\f_0}_{\H{\frac{3}{2}l+\frac{1}{2}}}^2 + \int_0^t \abs{\g(t')}^2_{\H{\frac{3}{2}l+\frac{1}{2}}} \dd{t'} },
			\end{split}
		\end{equation}
		where $ Q $ is a generic polynomial determined by $ \Lambda_* $.
	\end{prop}

	The existence of solutions to the linear initial value problem \eqref{eqn linear 1"} can be derived from \eqref{est linear eqn ka} by using Galerkin's method (or standard semi-group method in \cite{Shatah-Zeng2011}). In other words, the following linear well-posedness result holds:
	\begin{prop}
		For $ 0 \le l \le k-2 $ and $ \g \in C^0 \qty([0, T]; H^{\frac{3}{2}l + \frac{1}{2}}(\Gs)) $, there exists a positive constant $ T \le C(L_1, L_2) $, so that the problem \eqref{eqn linear 1"} is well-posed in the space $ C^0\qty{[0, T]; \H{\frac{3}{2}l+2}} \cap C^1\qty{[0, T]; \H{\frac{3}{2}l+\frac{1}{2}}} $, and the estimate \eqref{est linear eqn ka} holds.
	\end{prop}
	
	\subsection{Linear Problems for the Current and Vorticity}\label{sec linear current vortex}
	Assume that $ \{\Gt\}_{0\le t \le T} \subset \Lambda_* $ is  a family of hypersurfaces, and $ \vv(t), \vh(t) : \Om_t^+ \to \R^3 $ are given vector fields satisfying
	\begin{equation}\label{key}
		\begin{cases*}
			\div\vv = 0 = \div\vh &in $ \Om_t^+ $, \\
			\vh \vdot \vn = 0 &on $ \Gt $, \\
			\vv \vdot \vn = \vn \vdot \qty(\pd_t\gt\vnu)\circ\Phi_\Gt^{-1} &on $ \Gt $.
		\end{cases*}
	\end{equation}
	For the simplicity of notations, we denote by $ \bar{L}_1 $ a positive constant such that
	\begin{equation}\label{key}
		\sup_{t \in [0, T]} \qty(\abs{\ka(t)}_{\H{\kk-2}}, \abs{\pd_t\ka(t)}_{\H{\kk-\frac{5}{2}}}, \norm{\qty(\vv(t), \vh(t))}_{H^{\kk}(\Om_t^+)}) \le \bar{L}_1.
	\end{equation}
	
	Consider the following linearized current-vorticity systems in $ \Om_t^+ $:
	\begin{equation}\label{eqn linear vom}
		\pd_t \wt{\vom} + (\vv\vdot\grad)\wt{\vom} - (\vh\vdot\grad)\wt{\vj} = (\wt{\vom}\vdot\grad)\vv - (\wt{\vj}\vdot\grad)\vh,
	\end{equation}
	and
	\begin{equation}\label{eqn linear vj}
		\pd_t \wt{\vj} + (\vv\vdot\grad)\wt{\vj} - (\vh\vdot\grad)\wt{\vom} = (\wt{\vj}\vdot\grad)\vv - (\wt{\vom}\vdot\grad)\vh -2\tr(\grad\vv\cp \grad\vh).
	\end{equation}
	Set 
	\begin{equation}\label{def vxi veta}
		\vb*{\xi} \coloneqq \wt{\vom} - \wt{\vj}, \quad \vb*{\eta} \coloneqq \wt{\vom} + \wt{\vj}.
	\end{equation}
	Then
	\begin{equation}\label{eqn xi}
		\pd_t\vxi  + \qty[(\vv+\vh)\vdot\grad]\vb*{\xi} = (\vb*{\xi}\vdot\grad)(\vv+\vh) + 2 \tr(\grad\vv\cp\grad\vh),
	\end{equation}
	\begin{equation}\label{eqn eta}
		\pd_t \veta + \qty[(\vv-\vh)\vdot\grad]\veta = (\veta\vdot\grad)(\vv-\vh) - 2\tr(\grad\vv\cp\grad\vh).
	\end{equation}
	
	As shown in \cite{Sun-Wang-Zhang2018}, due to the fact that $ \vh \vdot \vn \equiv 0 $, $\vv \pm \vh$ are also evolution velocities of $\Om_t^+$. In particular, the local well-posedness of the above systems on $(\vxi, \veta)$ can be shown via using characteristic method (c.f. \cite{Caflisch-Klapper_Steele97}). Furthermore, the following energy estimates hold (see \cite{Sun-Wang-Zhang2018} or \cite{Liu-Xin2023} for a proof):
	\begin{prop}\label{prop linear vom vj}
		For $ 0 \le t \le T $, it follows that
		\begin{equation}\label{est linear vom vj}
			\begin{split}
				&\hspace{-2em}\norm{\wt{\vom}(t)}_{H^{\kk-1}(\Om_t^+)}^2 + \norm{\wt{\vj}(t)}_{H^{\kk-1}(\Om_t^+)}^2 \\
				\le\,  &C_*\exp{Q\qty(\bar{L}_1)t} \qty(1+ \norm{\wt{\vom}(0)}_{H^{\kk-1}(\Om_0^+)}^2 + \norm{\wt{\vj}(0)}_{H^{\kk-1}(\Om_0^+)}^2),
			\end{split}
		\end{equation}
		where $ Q $ is a generic polynomial determined by $ \Lambda_* $.
	\end{prop}
	
	\section{Nonlinear Problems}\label{sec nonlinear}
	Consider a reference hypersurface $ \Gs \in H^{\kk+1} $ separating $ \Om $ into two disjoint simply-connected parts, a transversal vector field $ \vnu : \Gt \to \R^3 $, and a small constant $ \delta_0 > 0 $ so that
	\begin{equation}\label{key}
		\Lambda_{*} \coloneqq \Lambda\qty(\Gs, \kk-\frac{1}{2}, \delta_0)
	\end{equation}
	satisfies all the properties discussed in the preliminary. Assume that for some large $ T^* > 0 $
	\begin{equation}\label{key}
		\vu{J} \in C^0\qty([0, T^*]; H^{\kk-\frac{1}{2}}(\pd\Om)) \cap C^1\qty([0, T^*]; H^{\kk-\frac{3}{2}}(\pd\Om)),
	\end{equation} 
	and
	\begin{equation}\label{key}
		\sup_{[0, T^*]} \qty(\abs*{\vu{J}}_{H^{\kk-\frac{1}{2}}(\pd\Om)} + \abs*{\vu{J}}_{H^{\kk-\frac{3}{2}}(\pd\Om)}) \le M_*.
	\end{equation}
	We shall solve the nonlinear problem via iterations on the linearized problems in the following space:
	\begin{equation}\label{key}
		\begin{split}
			&\ka \in C^0\qty([0, T]; \H{\kk-1}) \cap C^1\qty([0, T]; B_{\delta_1}\subset\H{\kk-\frac{5}{2}}) \cap C^2\qty([0, T]; \H{\kk-4}); \\
			&\vom_{*}, \vj_{*} \in  C^0\qty([0, T]; H^{\kk-1}(\Om_*^+)) \cap C^1\qty([0, T]; H^{\kk-2}(\Om_*^+)).
		\end{split}
	\end{equation}
	Now, we introduce the function spaces more precisely:
	\begin{defi}
		For given positive constants $ T, M_0, M_1, M_2$, and $ M_3 $, define $ \mathfrak{X} $ to be the collection of $ \qty(\ka, \vom_*, \vj_*) $ satisfying:
		\begin{equation*}
			\abs{\ka(0) - \kappa_{*}}_{\H{\kk-\frac{5}{2}}} \le \delta_1,
		\end{equation*}
		\begin{equation*}
			\abs{\qty(\pd_t \ka)(0)}_{\H{\kk-4}}, \norm{\vom_*(0)}_{H^{\kk-\frac{5}{2}}(\Om_*^+)}, \norm{\vj_*(0)}_{H^{\kk-\frac{5}{2}}(\Om_*^+)} \le M_0,
		\end{equation*}
		\begin{equation*}
			\sup_{t \in [0, T]} \qty(\abs{\ka}_{\H{\kk-1}}, \abs{\pd_t \ka}_{\H{\kk-\frac{5}{2}}}, \norm{\vom_*}_{H^{\kk-1}(\Om_*^+)}, \norm{\vj_*}_{H^{\kk-1}(\Om_*^+)}) \le M_1,
		\end{equation*}
		\begin{equation*}
			\sup_{t \in [0, T]} \qty(\norm{\pd_t\vom_*}_{H^{\kk-2}(\Om_*^+)}, \norm{\pd_t\vj_*}_{H^{\kk-2}(\Om_*^+)}) \le M_2,
		\end{equation*}
		and
		\begin{equation*}
			\sup_{t \in [0, T]} \abs{\pd^2_{tt}\ka}_{\H{\kk-4}} \le a^2 M_3 \ (\text{here $ a $ is the constant in the definition of $ \ka $}).
		\end{equation*}
	\end{defi}
	
	For $ 0 < \epsilon \ll \delta_1 $ and $ A > 0 $, consider a collection of initial data:
	\begin{equation*}
		\mathfrak{I}(\epsilon, A) \coloneqq\qty{\qty\big(\ini{\ka}, \ini{\pd_t\ka}, \ini{\vom_*}, \ini{\vj_*})}
	\end{equation*}
	for which
	\begin{gather*}
		\abs{\ini{\ka}-\kappa_{*}}_{\H{\kk-1}}<\epsilon; \\ \abs{\ini{\pd_t\ka}}_{\H{\kk-\frac{5}{2}}},\
		\norm{\ini{\vom_*}}_{H^{\kk-1}(\Om_*^+)},\ \norm{\ini{\vj_*}}_{H^{\kk-1}(\Om_*^+)} < A.
	\end{gather*}

	\subsection{Fluid Region, Velocity and Magnetic Fields}\label{section recovery}
	For $ \qty\big(\ka, \vom_*, \vj_*)  \in \mathfrak{X}$, $ \ka(t) $ induces a family of hypersurfaces $ \Gt \in \Lambda_* $, provided that $ M_1 T $ is not too large. Thus $ \Phi_\Gt $ and $ \X_\Gt $ can be defined by $ \ka(t) $.
	
	Set
	\begin{equation}\label{def bar vom vj}
		\bar{\vom} \coloneqq \Pb \qty(\vom_*(t) \circ \qty(\X_\Gt^+)^{-1}), \quad \bar{\vj} \coloneqq \Pb\qty(\vj_*(t) \circ \qty(\X_\Gt^+)^{-1}),
	\end{equation}
	where $\Pb$ is the Leray projection of a vector fields into its divergence-free part.
	
	Now, by solving the following div-curl problems:
	\begin{equation}\label{div-curl nonlinear v}
		\begin{cases*}
			\div \vv = 0 &in $ \Om_t^+ $, \\
			\curl \vv = \bar{\vom} &in $ \Om_t^+ $, \\
			\vv \vdot \vn = \vn \vdot \qty(\pd_t \gt \vnu)\circ\Phi_\Gt^{-1} &on $ \Gt $;
		\end{cases*}
	\end{equation}
	and
	\begin{equation}\label{div-curl nonlinear h}
		\begin{cases*}
			\div \vh = 0 &in $ \Om_t^+ $, \\
			\curl \vh = \bar{\vj} &in $ \Om_t^+ $, \\
			\vh \vdot \vn = 0 &on $ \Gt $.
		\end{cases*}
	\end{equation}
	one can obtain the corresponding velocity and magnetic fields $ \vv, \vh: \Om_t^+ \to \R^3 $. Furthermore, the following estimate holds thanks to Theorem \ref{thm div-curl}:
	\begin{equation}\label{key}
		\sup_{t \in [0, T]}\qty(\norm{\vv}_{H^{\kk}(\Om_t^+)}, \norm{\vh}_{H^{\kk}(\Om_t^+)}) \le Q(M_1),
	\end{equation}
	where $ Q $ is a generic polynomial determined by $ \Lambda_* $. For the magnetic field in the vacuum, consider the following div-curl problem:
	\begin{equation}\label{eqn div-curl vhh}
		\begin{cases*}
			\div\vhh = 0 &in $ \Om_t^- $, \\
			\curl \vhh  = \vb{0} &in $ \Om_t^- $, \\
			\vhh \vdot \vn = 0 &on $ \Gt $, \\
			\wt{\vn} \cp \vhh = \vu{J} &on $ \pd\Om $.
		\end{cases*}
	\end{equation}
	Then, it follows from the div-curl estimate that
	\begin{equation}\label{key}
		\norm{\vhh}_{H^{\kk}(\Om_t^-)} \le C\qty(\abs{\ka}_{\H{\kk-\frac{3}{2}}}) \cdot \abs{\vu{J}}_{H^{\kk-\frac{1}{2}}(\pd\Om)} \le Q\qty(M_1, M_*).
	\end{equation}
	Similarly, the time derivative of $ \vhh $ satisfies:
	\begin{equation}\label{eqn dt vhh}
		\begin{cases*}
			\div \pd_t\vhh = 0 \qc \curl \pd_t\vhh = \vb{0} &in $ \Om_t^- $, \\
			\pd_t\vhh \vdot \vn = \vn \vdot \qty(\DD_\vhh\vv - \DD_\vv\vhh) &on $ \Gt $, \\
			\wt{\vn} \cp \pd_t \vhh = \pd_t \vu{J} &on $ \pd\Om $.
		\end{cases*}
	\end{equation}
	Hence,
	\begin{equation}\label{key}
		\norm{\pd_t\vhh}_{H^{\kk-1}(\Om_t^-)} \le Q(M_1, M_*)
	\end{equation}
	for some generic polynomial $ Q $ depending on $ \Lambda_{*} $.
	
	\subsection{Iteration Mapping}\label{sec itetarion map}
	For $ \qty(\itn{\ka}, \itn{\vom_*}, \itn{\vj_*}) \in \mathfrak{X} $ and $ \qty\Big(\ini{\ka}, \ini{\pd_t\ka}, \ini{\vom_*}, \ini{\vj_*}) \in \mathfrak{I}(\epsilon, A) $,
	define the $ (n+1)$-th step by solving:
	\begin{equation}\label{eqn (n+1)ka"}
		\begin{cases*}
			\pd^2_{tt}\itm{\ka} + \opC\qty(\itn{\ka}, \itn{\pd_t\ka}, \itn{\vv_*}, \itn{\vh_*}, \itn{\vhh_*})\itm{\ka} \\ \qquad = \opF\qty(\itn{\ka})\pd_t \itn{\vom_*} + \opG\qty(\itn{\ka}, \itn{\pd_t\ka}, \itn{\vom_*}, \itn{\vj_*}, \vu{J}) \\
			\itm{\ka}(0) = \ini{\ka}, \quad \itm{\pd_t\ka}(0) = \ini{\pd_t\ka};
		\end{cases*}
	\end{equation}
	and
	\begin{equation}\label{eqn linear (n+1) vom vj"}
		\begin{cases*}
			\pd_t\itm{\vom} + \DD_{\itn{\vv}}\itm{\vom} - \DD_{\itn{\vh}}\itm{\vj} = \DD_{\itm{\vom}}\itn{\vv} - \DD_{\itm{\vj}}\itn{\vh}, \\
			\pd_t\itm{\vj} + \DD_{\itn{\vv}}\itm{\vj} - \DD_{\itn{\vh}}\itm{\vom} \\
			\hspace{6em}  =  \DD_{\itm{\vj}}\itn{\vv} - \DD_{\itm{\vom}}\itn{\vh}
			- 2\tr(\grad\itn{\vv} \cp \grad \itn{\vh}), \\
			\itm{\vom}(0) = \Pb \qty(\ini{\vom_*} \circ \qty(\X_{\itn{\Gamma_0}}^+)^{-1}), \quad \itm{\vj}(0) = \Pb\qty(\ini{\vj_*}\circ \qty(\X_{\itn{\Gamma_0}}^+)^{-1}),
		\end{cases*}
	\end{equation}
	where $ \qty(\itn{\vv}, \itn{\vh}, \itn{\vhh}) $ is induced by $ \qty(\itn{\ka}, \itn{\vom_*}, \itn{\vj_*}, \vu{J}) $ via solving \eqref{def bar vom vj}-\eqref{div-curl nonlinear h} and \eqref{eqn div-curl vhh}; the tangential vector fields $ \itn{\vv_*} $, $ \itn{\vh_*} $, and $ \itn{\vhh_*} $ on $ \Gs $ are defined by
	\begin{equation}\label{def itn vv_*}
		\begin{split}
			\itn{\vv_{*}} &\coloneqq \qty(\DD\Phi_{\itn{\Gt}})^{-1}\qty[\itn{\vv}\circ\Phi_{\itn{\Gt}} - \qty(\pd_t\gamma_{\itn{\Gt}})\vnu], \\
			\itn{\vh_{*}} &\coloneqq \qty(\DD\Phi_{\itn{\Gt}})^{-1} \qty(\itn{\vh} \circ \Phi_{\itn{\Gt}}), \\
			\itn{\vhh_{*}} &\coloneqq \qty(\DD\Phi_{\itn{\Gt}})^{-1} \qty(\itn{\vhh} \circ \Phi_{\itn{\Gt}});
		\end{split}
	\end{equation}
	and the current-vorticity equations are considered in the domain $ \itn{\Om_t^+} $.
	
	Denoting by
	\begin{equation}\label{key}
		\itm{\vom_*}\coloneqq \itm{\vom}\circ\X_{\itn{\Gt}}^+, \qand \itm{\vj_*}\coloneqq \itm{\vj}\circ\X_{\itn{\Gt}}^+,
	\end{equation}
	we will show that $ \qty(\itm{\ka}, \itm{\vom_*}, \itm{\vj_*}) \in \mathfrak{X} $.
	Indeed, Lemma \ref{lem 3.4} implies
	\begin{equation*}\label{key}
		\begin{split}
			&\hspace{-2em}\abs{\opF\qty(\itn{\ka})\itn{\pd_t\vom_*}}_{\H{\kk-\frac{5}{2}}} \\
			\le\, &Q\qty(\abs{\itn{\ka}}_{\H{\kk-1}})\norm{\itn{\pd_t\vom_*}}_{H^{\kk-\frac{5}{2}}(\Om_*^+)} \\
			\le\, &Q(M_1) \cdot M_2,
		\end{split} 
	\end{equation*}
	and
	\begin{equation*}\label{key}
		\begin{split}
			&\hspace{-2em}\abs{\opG\qty(\itn{\ka}, \itn{\pd_t\ka}, \itn{\vom_*}, \itn{\vj_*}, \vu{J})}_{\H{\kk-\frac{5}{2}}} \\
			\le\, &a^2 Q\qty(\abs{\itn{\ka}}_{\H{\kk-1}}, \abs{\itn{\pd_t\ka}}_{\H{\kk-\frac{5}{2}}}, \norm{\qty(\itn{\vom_*}, \itn{\vj_*})}_{H^{\kk-1}(\Om_*^+)}, \abs{\vu{J}}_{H^{\kk-\frac{1}{2}}(\pd\Om)}) \\
			\le\, &a^2 Q(M_1, M_*).
		\end{split}
	\end{equation*}
	Moreover, one can derive from \eqref{est linear eqn ka} that
	\begin{equation*}\label{key}
		\begin{split}
			&\sup_{t \in [0, T]}\qty(\abs{\itm{\ka}}_{\H{\kk-1}} + \abs{\itm{\pd_t\ka}}_{\H{\kk-\frac{5}{2}}}) \\
			&\le C_* \exp\qty{Q\qty(M_*, M_1, M_2, a^2M_3)T}   \times \qty[C_* + \epsilon + A + (M_0)^{12} + T\cdot \qty(a^2+M_2)Q(M_*, M_1)].
		\end{split}
	\end{equation*}
	If $ T \ll 1 $, $ M_1 \gg M_0$, and $M_1 \gg A$, one can get
	\begin{equation}\label{est ka n+1}
		\sup_{t \in [0, T]}\qty(\abs{\itm{\ka}}_{\H{\kk-1}}+ \abs{\itm{\pd_t\ka}}_{\H{\kk-\frac{5}{2}}}) \le M_1.
	\end{equation}		
	Moreover, choosing $ M_3 $ large enough compared to $M_*, M_1$, and $M_2 $, one may obtain from (\ref{eqn (n+1)ka"}) that 
	\begin{equation}\label{key}
		\sup_{t \in [0, T]} \abs{\itm{\pd^2_{tt}\ka}}_{\H{\kk-4}} \le a^2M_3.
	\end{equation}
	
	Similarly, combining Proposition \ref{prop linear vom vj} and (\ref{eqn linear (n+1) vom vj"}) yields for $T \ll 1$ and $M_1 \gg \abs{\kappa_{*}}_{\H{\kk-\frac{5}{2}}}$ that
	\begin{equation}\label{key}
		\begin{split}
			&\hspace{-2em}\sup_{t \in [0, T]}\qty(\norm{\itm{\vom_*}}_{H^{\kk-1}(\Om_*^+)}, \norm{\itm{\vj_*}}_{H^{\kk-1}(\Om_*^+)}) \\
			\le\, &Q\qty(\abs{\itn{\ka}}_{C^0_t\H{\kk-\frac{5}{2}}}) e^{Q(M_1)T}(1+2A) \\
			\le\, &M_1.
		\end{split}
	\end{equation}
	Next, choosing $ M_2 \gg M_1 $, one can get that
	\begin{equation*}\label{key}
		\begin{split}
			&\hspace{-2em}\norm{\itm{\pd_t\vom_*}(t)}_{H^{\kk-2}(\Om_*^+)} \\
			\le\, &C_* \qty(\norm{\itm{\pd_t\vom}}_{H^{\kk-2}(\itn{\Om_t^+})} + \norm{\itm{\vom}}_{H^{\kk-1}(\itn{\Om_t^+})}\abs{\itn{\pd_t\ka}}_{\H{\kk-\frac{5}{2}}} ) \\
			\le\,  &Q(M_1) \le M_2.
		\end{split}
	\end{equation*}
	Namely,
	\begin{equation}\label{key}
		\sup_{t \in [0, T]} \qty(\norm{\itm{\pd_t\vom_*}}_{H^{\kk-2}(\Om_*^+)}, \norm{\itm{\pd_t\vj_*}}_{H^{\kk-2}(\Om_*^+)}) \le M_2.
	\end{equation}
	
	Define
	\begin{equation}\label{key}
		\begin{split}
			\mathfrak{T}\qty{\qty[\ini{\ka}, \ini{\pd_t\ka}, \ini{\vom_*}, \ini{\vj_*}], \qty[\itn{\ka}, \itn{\vom_*}, \itn{\vj_*}]}
			\coloneqq \qty[\itm{\ka}, \itm{\vom_*}, \itm{\vj_*}].
		\end{split}
	\end{equation}
	The following proposition follows directly from the previous arguments:
	\begin{prop}
		Suppose that $ k \ge 2 $. For any $ 0 < \epsilon \ll \delta_0 $, $ A > 0 $, and $M_* \ge 0$, there are positive constants $ M_0, M_1, M_2$, and $ M_3 $, so that for small $ T > 0 $,
		\begin{equation*}
			\mathfrak{T}\qty\Big{\qty[\ini{\ka}, \ini{\pd_t\ka}, \ini{\vom_*}, \ini{\vj_*}], \qty[\ka, \vom_*, \vj_*]} \in \mathfrak{X},
		\end{equation*}
		holds for any $ \qty[\ini{\ka}, \ini{\pd_t\ka}, \ini{\vom_*}, \ini{\vj_*}] \in \mathfrak{I}(\epsilon, A) $ and $ \qty[\ka, \vom_*, \vj_*] \in \mathfrak{X} $.
	\end{prop}
	
	\subsection{Contraction of the Iteration Mapping}\label{sec contra ite}
	Suppose that  $ \qty(\itn{\ka}(\beta), \itn{\vom_*}(\beta), \itn{\vj_*}(\beta)) \subset \mathfrak{X} $ and  $ \qty\Big(\ini{\ka}(\beta), \ini{\pd_t\ka}(\beta), \ini{\vom_*}(\beta), \ini{\vj_*}(\beta)) \subset \mathfrak{I}(\epsilon, A)$  are two families parameterized by $ \beta $.
	Define 
	\begin{equation*}
		\begin{split}
			&\hspace{-2em}\qty(\itm{\ka}(\beta), \itm{\vom_*}(\beta), \itm{\vj_*}(\beta)) \\
			\coloneqq\, &\mathfrak{T}\qty{\qty\Big(\ini{\ka}(\beta), \ini{\pd_t\ka}(\beta), \ini{\vom_*}(\beta), \ini{\vj_*}(\beta)), \qty(\itn{\ka}(\beta), \itn{\vom_*}(\beta), \itn{\vj_*}(\beta))}.
		\end{split}
	\end{equation*}
	Then applying $ \pdv*{\beta} $ to (\ref{eqn (n+1)ka"}) and (\ref{eqn linear (n+1) vom vj"}) respectively yields
	\begin{equation}\label{eqn contra 1}
		\begin{cases*}
			\pd^2_{tt} \pd_\beta \itm{\ka} + \itn{\opC}\pd_\beta\itm{\ka}  = - \qty(\pd_\beta\itn{\opC})\itm{\ka} + \pd_\beta \qty(\itn{\opF}\itn{\pd_t\vom_*} + \itn{\opG}), \\
			\pd_\beta\itm{\ka}(0) = \pd_\beta\ini{\ka}(\beta), \quad \pd_t\qty(\pd_\beta\itm{\ka})(0) = \pd_\beta\ini{\pd_t\ka}(\beta),
		\end{cases*}
	\end{equation}
	and
	\begin{equation}\label{eqn beta n+1}
		\begin{cases*}
			\pd_t\Dbt\itm{\vom} + \DD_{\itn{\vv}}\Dbt\itm{\vom} - \DD_{\itn{\vh}}\Dbt\itm{\vj} = \DD_{\Dbt\itm{\vom}}\itn{\vv} - \DD_{\Dbt\itm{\vj}}\itn{\vh} + \va{\g}_1, \\
			\pd_t \Dbt\itm{\vj} + \DD_{\itn{\vv}}\Dbt\itm{\vj} - \DD_{\itn{\vh}}\Dbt\itm{\vom} = \DD_{\Dbt\itm{\vj}}\itn{\vv} - \DD_{\Dbt\itm{\vom}}\itn{\vh} + \va{\g}_2, \\
			\Dbt\itm{\vom}(0) = \Pb\qty{\qty[\pd_\beta\ini{\vom_*}] \circ \qty(\X_{\itn{\Gamma_0}(\beta)}^+)^{-1}},  \quad \Dbt\itm{\vj}(0) = \Pb\qty{\qty[\pd_\beta\ini{\vj_*}]\circ\qty(\X_{\itn{\Gamma_0}(\beta)}^+)^{-1}},
		\end{cases*}
	\end{equation}
	where
	\begin{equation}\label{key}
		\Dbt \coloneqq \pdv{\beta} + \DD_{\vb*{\mu}} \qc \vb*{\mu} \coloneqq \h_+\qty[\qty(\pd_\beta\gamma_{\itn{\Gt}}\vnu)\circ\Phi_{\itn{\Gt}(\beta)}^{-1}],
	\end{equation}
	\begin{equation}\label{key}
		\begin{split}
			\va{\g}_1 \coloneqq\, &\comm{\pd_t}{\Dbt}\itm{\vom} + \comm{\DD_{\itn{\vv}}}{\Dbt}\itm{\vom} - \comm{\DD_{\itn{\vh}}}{\Dbt}\itm{\vj} \\
			&+\itm{\vom} \vdot \Dbt\DD\itn{\vv} - \itm{\vj} \vdot \Dbt\DD\itn{\vh},
		\end{split}
	\end{equation}
	and
	\begin{equation}\label{eqn g2}
		\begin{split}
			\va{\g}_2 \coloneqq\, &-2\Dbt\tr(\grad\itn{\vv} \cp \grad\itn{\vh}) + \comm{\pd_t}{\Dbt}\itm{\vj} + \comm{\DD_{\itn{\vv}}}{\Dbt}\itm{\vj} \\
			&-\comm{\DD_{\itn{\vh}}}{\Dbt}\itm{\vom}+ \itm{\vj}\vdot \Dbt\DD\itn{\vh} - \itm{\vom}\vdot\Dbt\DD\itn{\vh}.
		\end{split}
	\end{equation}
	
	Consider the following energy functionals:
	\begin{equation}\label{key}
		\begin{split}
			\itn{\E}(\beta)\coloneqq &\sup_{t \in [0, T]}\left(\abs{\pd_\beta\itn{\ka}}_{\H{\kk-\frac{5}{2}}} + \abs{\pd_\beta\itn{\pd_t\ka}}_{\H{\kk-4}} + \norm{\pd_\beta\itn{\vom_*}}_{H^{\kk-\frac{5}{2}}(\Om_*^+)} +  \right. \\
			&\qquad\qquad \left. + \norm{\pd_\beta\itn{\vj_*}}_{H^{\kk-\frac{5}{2}}(\Om_*^+)} + \norm{\pd_\beta\pd_t\itn{\vom_*}}_{H^{\kk-4}(\Om_*^+)}   \right),
		\end{split}
	\end{equation}
	and
	\begin{equation}\label{key}
		\begin{split}
			\ini{\E}(\beta)\coloneqq &\sup_{t \in [0, T]}\left(\abs{\pd_\beta\ini{\ka}}_{\H{\kk-\frac{5}{2}}} + \abs{\pd_\beta\ini{\pd_t\ka}}_{\H{\kk-4}}  +  \right. \\
			&\qquad\qquad \left. + \norm{\pd_\beta\ini{\vom_*}}_{H^{\kk-\frac{5}{2}}(\Om_*^+)} + \norm{\pd_\beta\ini{\vj_*}}_{H^{\kk-\frac{5}{2}}(\Om_*^+)} \right).
		\end{split}
	\end{equation}
	
	Note that Lemmas \ref{lem3.1}-\ref{lem 3.4} imply
	\begin{equation*}\label{key}
		\begin{split}
			\abs{\qty(\pd_\beta\itn{\opC})\itm{\ka}}_{\H{\kk-4}}
			\le Q(M_1, M_*) \itn{\E}\abs{\itm{\ka}}_{\H{\kk-2}} \le Q(M_1, M_*) \itn{\E},
		\end{split}
	\end{equation*}
	\begin{equation*}\label{key}
		\abs{\opF\qty(\itn{\ka})\pd_\beta\pd_t\itn{\vom_*}}_{\H{\kk-4}} \le Q(M_1) \norm{\pd_\beta\pd_t \itn{\vom_*}}_{H^{\kk-4}(\Om_*^+)},
	\end{equation*}
	and
	\begin{equation*}\label{key}
		\begin{split}
			&\hspace{-2em}\abs{\qty(\pd_\beta\itn{\opF})\pd_t\itn{\vom_*} + \pd_\beta\itn{\opG}}_{\H{\kk-4}} \\
			\le\, &C_* \abs{\pd_\beta\itn{\ka}}_{\H{\kk-\frac{5}{2}}}\norm{\pd_t\itn{\vom_*}}_{H^{\kk-\frac{7}{2}}(\Om_*^+)} + a^2 Q(M_1, M_*) \itn{\E} \\
			\le\, &\qty(C_*M_2 + a^2 Q(M_1, M_*)) \itn{\E}.
		\end{split}
	\end{equation*}
	Taking $ l=k-3 $ in (\ref{est linear eqn ka}) yields
	\begin{equation}\label{est beta ka(n+1)}
		\begin{split}
			&\sup_{t \in [0, T]} \qty(\abs{\pd_\beta\itm{\ka}}_{\H{\kk-\frac{5}{2}}} + \abs{\pd_\beta\itm{\pd_t\ka}}_{\H{\kk-4}}) \\
			&\quad\le C_* \exp{Q\qty(M_1, M_2, a^2M_3, M_*)T}  \times\qty(\ini{\E} + T \cdot \qty[C_*M_2 + \qty(1+a^2)Q(M_1, M_*)]\itn{\E}).
		\end{split}
	\end{equation}
	To estimate (\ref{eqn beta n+1}), one first observes that
	\begin{equation*}\label{key}
		\begin{split}
			&\norm{\va{\g}_1}_{H^{\kk-\frac{5}{2}}(\itn{\Om_t^+})} \\
			&\quad  \le Q(M_1) \qty(\abs{\pd_t\pd_\beta\itn{\ka}}_{\H{\kk-4}} + \norm{\qty(\Dbt \itn{\vv}, \Dbt\itn{\vh})}_{H^{\kk-\frac{3}{2}}(\itn{\Om_t^+})} + \abs{\pd_\beta\ka}_{\H{\kk-\frac{5}{2}}})
		\end{split}
	\end{equation*}
	and \eqref{est pd beta v+*} imply
	\begin{equation*}\label{key}
		\norm{\va{\g}_1}_{H^{\kk-\frac{5}{2}}(\itn{\Om_t^+})} \le Q(M_1) \itn{\E}.
	\end{equation*}
	Similarly, it holds that
	\begin{equation*}\label{key}
		\norm{\va{\g}_2}_{H^{\kk-\frac{5}{2}}(\itn{\Om_t^+})} \le Q(M_1) \itn{\E}.
	\end{equation*}
	It follows from Proposition \ref{prop linear vom vj} that
	\begin{equation}\label{est dbt vom vj (n+1)}
		\begin{split}
			&\hspace{-2em}\sup_{t \in [0, T]}\qty(\norm{\Dbt\itm{\vom}}_{H^{\kk-\frac{5}{2}}(\itn{\Om_t^+})} + \norm{\Dbt\itm{\vj}}_{H^{\kk-\frac{5}{2}}(\itn{\Om_t^+})}) \\
			\le\, &e^{Q(M_1)T} \qty{\ini{\E} + TQ(M_1)\itn{\E}},
		\end{split}
	\end{equation}
	and thus
	\begin{equation*}\label{key}
		\norm{\pd_t\Dbt\itm{\vom}}_{H^{\kk-4}(\itn{\Om_t^+})} \le e^{Q(M_1)T} Q(M_1)\qty{\ini{\E} + TQ(M_1)\itn{\E}}.
	\end{equation*}
	Since
	\begin{equation*}\label{key}
		\qty[\pd_t\pd_\beta\itm{\vom_*}]\circ\qty(\X_{\itn{\Gt}}^+)^{-1} = \pd_t\Dbt\itm{\vom} + \DD_{\h\qty[\qty(\pd_t\itn{\gt}\vnu)\circ\Phi_{\itn{\Gt}}^{-1}]}\Dbt\itm{\vom},
	\end{equation*}
	one has
	\begin{equation}\label{est pd beta t vom (n+1)}
		\norm{\pd_\beta\pd_t\itm{\vom_*}}_{H^{\kk-4}(\Om_*^+)} \le e^{Q(M_1)T} Q(M_1)\qty{\ini{\E} + TQ(M_1)\itn{\E}}.
	\end{equation}
	
	Then (\ref{est beta ka(n+1)})-(\ref{est pd beta t vom (n+1)}) imply that
	\begin{equation}\label{key}
		\begin{split}
			\itm{\E} \le\, &C_* \exp\qty{Q\qty(M_1, M_2, a^2M_3, M_*)T} \\
			&\quad\times \qty{\ini{\E} + T\itn{\E} \cdot \qty[M_2 + \qty(1+a^2)Q(M_1, M_*)]} \\
			&+ e^{Q(M_1)T} Q(M_1)\qty{\ini{\E} + TQ(M_1)\itn{\E}},
		\end{split}
	\end{equation}
	Thus, if $ T $ is small enough with respect to $ M_1, M_2, M_3, M_* $ and $ a $, it holds that
	\begin{equation}\label{key}
		\itm{\E} \le \frac{1}{2}\itn{\E} + Q(M_1)\ini{\E},
	\end{equation}
	In other words, one has obtained the following proposition:
	\begin{prop}\label{fixed point"}
		Assume that $ k \ge 3 $. For any $ 0 < \epsilon \ll \delta_0 $, $ A > 0 $ and $ M_* \ge 0 $, there are positive constants $ M_0, M_1, M_2$, and $ M_3 $, so that if $ T $ is small enough, there exists a map $ \mathfrak{S} : \mathfrak{I}(\epsilon, A) \to  \mathfrak{X} $ such that
		\begin{equation}\label{key}
			\mathfrak{T}\qty{\mathfrak{x}, \mathfrak{S(x)}} = \mathfrak{S(x)},
		\end{equation}
		for each $ \mathfrak{x} = \qty[\ini{\ka}, \ini{\pd_t\ka}, \ini{\vom_*}, \ini{\vj_*}] \in \mathfrak{I}(\epsilon, A) $.
	\end{prop}
	
	\subsection{The Original Plasma-Vacuum Problems}\label{sec Back to ori p-v problem}
	It still remains to show that the unique fixed point given in Proposition \ref{fixed point"} will induce a solution to the original plasma-vacuum problem \eqref{MHD}-\eqref{BC}.
	
	Following the arguments in \cite{Liu-Xin2023}, one first observes that any initial data $\qty{\Gamma_0, \vv_0, \vh_0}$ can induce the quantity $\qty{\ka(0), \pd_t\ka(0), \vom_*(0), \vj_*(0)}$, as long as $\Gs$ and $\vnu$ is taken so that $\Gamma_0 \in \Lambda_{*}$. (Indeed, one can directly set $\Gs \coloneqq \Gamma_0$.)
	
	Set $ \qty{\ini{\ka}, \ini{\pd_t\ka}, \ini{\vom_*}, \ini{\vj_*}} \coloneqq \qty{\ka(0), \pd_t\ka(0), \vom_*(0), \vj_*(0)} $, and take the unique fixed point $ \qty{\ka(t), \vom_*(t), \vj_*(t)} \in \mathfrak{X} $ given in  Proposition~\ref{fixed point"}. Thus, $ \{\Gt, \vv, \vh, \vhh\} $ can be obtained via solving \eqref{def bar vom vj}-\eqref{div-curl nonlinear h} and \eqref{eqn div-curl vhh}. We will show that	the induced quantity $ \{\Gt, \vv, \vh, \vhh\} $ is a solution to the problem \eqref{MHD}-\eqref{BC} with initial data $ \qty{\Gamma_0, \vv_0, \vh_0} $.
	
	First, we show that
	\begin{equation}\label{key}
		\Pb \qty(\vom_*(t) \circ \qty(\X_\Gt^+)^{-1}) = \vom_*(t) \circ \qty(\X_\Gt^+)^{-1}, \quad \Pb \qty(\vj_*(t) \circ \qty(\X_\Gt^+)^{-1}) = \vj_*(t) \circ \qty(\X_\Gt^+)^{-1}.
	\end{equation}
	Indeed, (\ref{eqn linear (n+1) vom vj"}) and the fact that $ \div \vv \equiv 0 \equiv \div\vh $ yield
	\begin{equation}\label{key}
		\begin{cases*}
			\pd_t (\div\vom) + \DD_{\vv}(\div \vom) = \DD_{\vh} (\div\vj), \\
			\pd_t (\div\vj) + \DD_{\vv}(\div\vj) = \DD_{\vh} (\div\vom),
		\end{cases*}
	\end{equation}
	where
	\begin{equation*}\label{key}
		\vom(t) \coloneqq \vom_*(t) \circ \qty(\X_\Gt^+)^{-1}, \quad \vj(t)\coloneqq \vj_*(t) \circ \qty(\X_\Gt^+)^{-1}.
	\end{equation*}
	Since $ \div \vom(0) = 0 = \div \vj(0) $, it follows from the characteristic methods that $ \div\vom \equiv 0 \equiv \div \vj $ for all $ t $, which proves the claim (c.f. \cite{Liu-Xin2023} for more details).
	Thus,
	\begin{equation}\label{eqn curl vv vh}
		\curl\vv = \vom \qand \curl \vh = \vj.
	\end{equation}
	
	Secondly, define the effective pressure as in \eqref{decom pressure'}:
	\begin{equation*}\label{def p'}
		p = p_{\vv, \vv} - p_{\vh, \vh} + \alpha^2 \h_+ \kappa + \h_+ \qty(\txfrac{1}{2}\abs*{\vhh}^2).
	\end{equation*}
	As in  \cite{Sun-Wang-Zhang2018} and \cite{Liu-Xin2023}, one can introduce the following two vector fields:
	\begin{equation}
		\vV \coloneqq \pd_t \vv + \DD_\vv\vv + \grad p - \DD_\vh\vh, \label{def vV} 
	\end{equation}
	and
	\begin{equation}
		\vH \coloneqq \pd_t \vh + \DD_\vv \vh - \DD_\vh \vv. \label{def vH}
	\end{equation}
	To show that $\{\Gt, \vv, \vh, \vhh\}$ is a solution, one only needs to verify that $\vV \equiv \vb{0}$ and $\vH \equiv \vb{0}$ (since $\vhh$ satisfies \eqref{eqn pM} automatically). The div-curl estimates enable one to check only the following relations:
	\begin{equation}
		\begin{cases*}
			\div \vV = 0\qc\div \vH = 0 &in $\Om_t^+$, \\
			\curl \vV = \vb{0}\qc\curl \vH = \vb{0} &in $\Om_t^+$, \\
			\vV \vdot \vn = 0\qc\vH \vdot \vn = 0 &on $\Gt$.
		\end{cases*}
	\end{equation}
	
	It follows from \eqref{decom pressure'} that $\div \vV \equiv 0$. The curl of $ \vV $ can be computed via (\ref{eqn curl vv vh}) and \eqref{eqn linear (n+1) vom vj"} as:
	\begin{equation}\label{key}
		\curl \vV = \pd_t \vom + \DD_{\vv}\vom - \DD_{\vom}\vv - \DD_{\vh}\vj + \DD_{\vj}\vh = \vb{0}.
	\end{equation}
	To verify the boundary condition, one notes first that $\vV = \vW$ on $\Gt$ for $\vW$ defined by \eqref{def vW"}. Furthermore, as $\qty{\ka, \vom_*, \vj_*}$ is a fixed point of $\mathfrak{T}$, it follows from \eqref{eqn pdt2 ka "} and \eqref{eqn (n+1)ka"} that
	\begin{equation}\label{eqn vW"}
		-\lap_\Gt \qty(\vW \vdot \vn) + \vW \vdot \lap_\Gt \vn + a^2 \dfrac{\vW \vdot \vn}{\vn \vdot (\vnu \circ \Phi_\Gt^{-1})} = 0.
	\end{equation}
	Since $\div \vV = 0$, $\curl \vV = \vb{0}$ and $\Om_t^+$ is simply-connected, one can find a mean-zero function $r(t)\colon \Gt \to \R$ such that
	\begin{equation}
		\vV = \grad\h_+r.
	\end{equation}
	Set
	\begin{equation}
		\Theta \coloneqq \vV \vdot \vn = \vW \vdot \vn \qq{on} \Gt.
	\end{equation}
	Then it is clear that
	\begin{equation}
		\Theta = \n_+ r \qand r = \qty(\n_+)^{-1} \Theta.
	\end{equation}
	One can deduce from \eqref{lap vn} and \eqref{eqn vW"} that
	\begin{equation}\label{key}
		-\lap_\Gt \Theta + \qty(\dfrac{a^2}{\vn \vdot (\vnu\circ \Phi_\Gt^{-1}) } - \abs{\II}^2)\Theta +  \grad^\top r \vdot \grad^\top \kappa = 0.
	\end{equation}
	Standard elliptic estimates and \eqref{property1 n} yield that $\Theta \equiv 0$, as long as $a$ is taken large enough (see \cite{Liu-Xin2023} for more details), which implies $\vV \equiv \vb{0}$. The verification of $\vH \equiv \vb{0}$ follows from a similar and simpler way. 
	
	In conclusion, the solution to the plasma-vacuum problem \eqref{MHD}-\eqref{BC} can be obtained by the fixed point of the iteration mapping, i.e., Theorem \ref{thm p-v wp} holds.
	
	\section{A Priori Estimates under the Rayleigh-Taylor Sign Condition}
	
	\subsection{Evolution of the Mean Curvature}
	
	In order to investigate the effect of the Rayleigh-Taylor (RT) sign condition \eqref{RT}, we shall derive the evolution equation for the mean curvature again to obtain finer estimates. Indeed, the equations \eqref{eqn dt2 kappa}-\eqref{eqn vb A} are sufficient to indicate the stabilization effects of surface tensions and non-collinear magnetic fields, but not enough to reveal the influence of the RT condition. When comparing \eqref{eqn dt2 kappa}-\eqref{est R} with \eqref{eqn Dt^2 kappa} and \eqref{est mathfrak[R]}, one can see that the term related to the RT condition corresponds to a first order derivative of $\kappa$ and is included in the remainder term $\mathfrak{R}$ in \eqref{eqn dt2 kappa}, since all of the first order derivatives are treated as lower order terms there. Therefore, to clarify the effect of the RT condition, one needs to derive the evolution equation for $\kappa$ with finer derivative estimates.
	
	From now on, it is always assumed that $k \ge 3$.  One first notes that $ \vh \vdot \vn \equiv 0 $ on $ \Gt $, so $ \vv \pm \vh $ are both evolution velocities of $ \Gt $. Denote by
	\begin{equation}\label{key}
		\vu{u}_\pm \coloneqq \vv \pm \vh \qand \wh{\Dt}_\pm \coloneqq \pd_t + \DD_{\vu{u}_\pm}.
	\end{equation}
	Then, it follows from \eqref{eqn dt II tan}-\eqref{eqn dt kappa} that
	\begin{equation}\label{key}
		\wh{\Dt}_{+}\kappa = - \vn \vdot \lap_{\Gt}(\vuu_+) - 2 \ip{\II}{\DD^\top\vuu_+},
	\end{equation}
	and
	\begin{equation}\label{key}
		\begin{split}
			\wh{\Dt}_{-}\wh{\Dt}_{+}\kappa =\, &- \vn\vdot\wh{\Dt}_{-}\qty(\lap_{\Gt}\vuu_{+}) - \wh{\Dt}_{-}\vn \vdot \lap_{\Gt}\vuu_{+} - 2 \ip{\wh{\Dt}_{-}\II}{(\DD\vuu_{+})^\top} - 2 \ip{\II}{\wh{\Dt}_{-}\qty(\DD\vuu_{+})^\top} \\
			=\, &- \vn \vdot \lap_\Gt \qty(\wh{\Dt}_{-}\vuu_{+}) - \vn \vdot \comm{\wh{\Dt}_{-}}{\lap_{\Gt}}\vuu_{+} + \vn \vdot \DD\vuu_{-} \vdot \qty(\lap_\Gt\vuu_{+})^\top \\
			&- 2 \ip{\DD^\top\qty[\qty(\DD\vuu_{-})^*\vn]^\top}{(\DD\vuu_{+})^\top} -  2\ip{\II}{\qty(\DD\wh{\Dt}_{-}\vuu_{+})^\top} - 2 \ip{\II}{\qty(\comm{\wh{\Dt}_{-}}{\DD}\vuu_{+})^\top}.
		\end{split}
	\end{equation}
	Thanks to the commutator formulas (c.f. \cite[pp. 709-710]{Shatah-Zeng2008-Geo}):
	\begin{equation}\label{key}
		\comm{\wh{\Dt}_{-}}{\grad}f = - \qty(\DD\vuu_-)^*(\grad f),
	\end{equation}
	and
	\begin{equation}\label{key}
		\comm{\wh{\Dt}_{-}}{\lap_\Gt}f = -2 \ip{\qty(\DD^\top)^2 f}{\DD^\top\vuu_{-}} - \grad^\top f \vdot \lap_{\Gt}\vuu_- + \kappa \vn \vdot \DD_{\grad^\top f}\vuu_{-},
	\end{equation}
	one can derive that
	\begin{equation}\label{eqn Dt_- Dt_+ kappa}
		\wh{\Dt}_- \wh{\Dt}_+ \kappa = - \vn \vdot \lap_\Gt \qty(\wh{\Dt}_- \vuu_+) - 2 \ip{\II}{\qty(\DD\wh{\Dt}_-\vuu_+)^\top} + R,
	\end{equation}
	with
	\begin{equation}\label{key}
		\abs{R}_{H^{\kk-\frac{5}{2}}(\Gt)} \le Q\qty(\abs{\vv}_{H^{\kk-\frac{1}{2}}(\Gt)}, \abs{\vh}_{H^{\kk-\frac{1}{2}}(\Gt)}),
	\end{equation}
	here $ Q $ is a generic polynomial determined by $ \Lambda_{*} $.
	
	If $ (\Gt, \vv, \vh, \vhh) $ is a solution to the problem \eqref{MHD}-\eqref{BC}, then
	\begin{equation}\label{Dt_- vbu_+}
		\begin{split}
			\wh{\Dt}_- \vuu_+ &= \pd_t (\vv + \vh) + \DD_{\vv-\vh}(\vv+ \vh) \\
			&= \pd_t \vv + \DD_\vv\vv - \DD_\vh\vh + \pd_t\vh + \DD_\vv \vh - \DD_\vh \vv \\
			&= - \grad p,
		\end{split}
	\end{equation}
	and the pressure function can be decomposed via \eqref{decom pressure'}.
	For the simplicity of notations, set
	\begin{equation}\label{def p^hat, q}
		\hat{p} \coloneqq \h_+ \qty(\txfrac{1}{2}\abs*{\vhh}^2), \qand q \coloneqq p_{\vv, \vv} - p_{\vh, \vh}.
	\end{equation}
	It follows from the definitions and standard elliptic estimates that
	\begin{equation*}\label{key}
		\norm{\grad p_{\vv, \vv}}_{H^{\kk}(\Om_t^+)} \less  \norm{\vv}_{H^{\kk}(\Om_t^+)} \qc
		\norm{\grad p_{\vh, \vh}}_{H^{\kk}(\Om_t^+)} \less \norm{\vh}_{H^{\kk}(\Om_t^+)},
	\end{equation*}
	and
	\begin{equation*}\label{key}
		\norm{\hat{p}}_{H^{\kk}(\Om_t^+)} \less \abs{\vhh}_{H^{\kk-\frac{1}{2}}(\Gt)}^2.
	\end{equation*}
	
	Substituting \eqref{Dt_- vbu_+} into \eqref{eqn Dt_- Dt_+ kappa} leads to
	\begin{equation}\label{eqn 1 dt-dt+ kappa}
		\begin{split}
			\wh{\Dt}_- \wh{\Dt}_+ \kappa =\, &\vn \vdot \lap_\Gt\qty(\grad p) + 2 \ip{\II}{\qty(\DD\grad p)^\top} + R \\
			=\, &\alpha^2\lap_\Gt\n_+\kappa - \alpha^2\abs{\grad^\top\kappa}^2 + \alpha^2\abs{\II}^2\n_+\kappa \\
			&+\vn \vdot \lap_{\Gt}\qty(\grad q + \grad \hat{p}) + 2 \ip{\II}{\qty[\DD\grad(q+\hat{p})]^\top} + R,
		\end{split}
	\end{equation}
	where the last equality is due to \eqref{lap vn}.
	It can be derived from \eqref{lap gt f alt} that
	\begin{equation}\label{lap grad q}
		\begin{split}
			\vn\vdot\lap_{\Gt}\grad q =\, &\vn \vdot \lap \grad q - \kappa \vn \vdot \DD_{\vn}\grad q - \vn \vdot \DD^2(\grad q)\qty(\vn, \vn) \\
			=\, &\vn \vdot \grad \lap q - \DD_{\vn}\qty[\kappa_{\h_+}\DD_{\vn_{\h_+}}q + \qty(\DD^2 q)\qty(\vn_{\h_+}, \vn_{\h_+})] \\
			&+\qty(\DD_{\vn}q)\n_+ \kappa + \kappa \grad q \vdot \n_+(\vn) + 2 \DD^2 q \qty(\vn, \n_+(\vn)),
		\end{split}
	\end{equation}
	where $ \kappa_{\h_+} \equiv \h_+(\kappa)  $ and $ \vn_{\h_+} \equiv \h_+(\vn) $.
	Define
	\begin{equation}\label{key}
		f \coloneqq \kappa_{\h_+}\DD_{\vn_{\h_+}}q + \qty(\DD^2 q)\qty(\vn_{\h_+}, \vn_{\h_+}).
	\end{equation}
	The relations that
	\begin{equation}\label{eqn q}
		\begin{cases*}
			\lap q = -\tr(\DD\vv)^2 + \tr(\DD\vh)^2 &in $ \Om_t^+ $, \\
			q = 0 &on $ \Gt $,
		\end{cases*}
	\end{equation}
	and \eqref{lap gt f alt} yield
	\begin{equation*}\label{key}
		f|_\Gt = \lap q - \lap_{\Gt} q = \lap q.
	\end{equation*}
	Thus,
	\begin{equation*}\label{key}
		\begin{split}
			\abs{f}_{H^{\kk-\frac{3}{2}}(\Gt)} &\le C_* \norm{\lap q}_{H^{\kk-1}(\Om_t^+)} \\
			&\le Q\qty(\norm{\vv}_{H^{\kk}(\Om_t^+)}, \norm{\vh}_{H^{\kk}(\Om_t^+)}).
		\end{split}
	\end{equation*}
	Furthermore, it follows from the standard product estimates that
	\begin{equation}\label{est lap f Omt}
		\begin{split}
			\norm{\lap f}_{H^{\kk-3}(\Om_t^+)} &\lesssim \norm{\kappa_{\h_+} \vn_{\h_+} \vdot \grad \qty[\tr(\DD\vv)^2 - \tr(\DD\vh)^2]}_{H^{\kk-3}(\Om_t^+)}  \\
			&\qquad + \norm{\DD\kappa_{\h_+} \DD\vn_{\h_+} \vdot \grad q + \DD\kappa_{\h_+} \vn_{\h_+} \DD^2 q}_{H^{\kk-3}(\Om_t^+)} \\ 
			&\qquad + \norm{\kappa_{\h_+} \DD\vn_{\h_+} \DD^2 q}_{H^{\kk-3}(\Om_t^+)} \\
			&\qquad + \norm{\DD^2\qty[\tr(\DD\vv)^2 - \tr(\DD\vh)^2] \vn_{\h_+} \vn_{\h_+}}_{H^{\kk-3}(\Om_t^+)} \\
			&\qquad + \norm{\DD^2 q \DD\vn_{\h_+}\DD\vn_{\h_+} + \DD^3 q \vn_{\h_+} \DD\vn_{\h_+}}_{H^{\kk-3}(\Om_t^+)} \\
			&\le Q\qty(\norm{\vv}_{H^{\kk}(\Om_t^+)}, \norm{\vh}_{H^{\kk}(\Om_t^+)}),
		\end{split}
	\end{equation}
	which implies
	\begin{equation}\label{key}
		\abs{\DD_{\vn}f}_{H^{\kk-\frac{5}{2}}(\Gt)} \le Q\qty(\norm{\vv}_{H^{\kk}(\Om_t^+)}, \norm{\vh}_{H^{\kk}(\Om_t^+)}).
	\end{equation}
	Thus, the product estimates yield that
	\begin{equation}\label{est n vdot lapGt grad q'}
		\abs{\vn\vdot\lap_\Gt\grad q - \qty(\DD_\vn q)\n_+\kappa}_{H^{\kk-\frac{5}{2}}(\Gt)} \le Q\qty(\norm{\vv}_{H^{\kk}(\Om_t^+)}, \norm{\vh}_{H^{\kk}(\Om_t^+)}).
	\end{equation}
	To deals with the terms involving $ \hat{p} $, one defines that
	\begin{equation}\label{key}
		\hat{f}\coloneqq \kappa_{\h_+} \DD_{\vn_{\h_+}}\hat{p} + \qty(\DD^2\hat{p})\qty(\vn_{\h_+}, \vn_{\h_+}).
	\end{equation}
	Hence,
	\begin{equation}\label{key}
		\vn \vdot \lap_\Gt \grad\hat{p} = - \DD_{\vn}\hat{f} + \qty(\DD_\vn\hat{p})\n_+\kappa + \kappa \grad\hat{p}\vdot\n_+(\vn) + 2 \DD^2\hat{p}\qty(\vn, \n_+(\vn)).
	\end{equation}
	It follows from the same calculations as in \eqref{est lap f Omt} that
	\begin{equation*}\label{key}
		\norm{\lap\hat{f}}_{H^{\kk-3}(\Om_t^+)} \le Q\qty(\abs*{\vhh}_{H^{\kk-\frac{1}{2}}(\Gt)}).
	\end{equation*} 
	Since
	\begin{equation*}\label{key}
		\hat{f}|_\Gt = \lap \hat{p} - \lap_{\Gt} \hat{p} = - \lap_{\Gt}\hh,
	\end{equation*}
	one may derive that
	\begin{equation*}\label{key}
		\abs{\DD_\vn \hat{f} - \n_+ \qty[-\lap_\Gt\hh]}_{H^{\kk-\frac{5}{2}}(\Gt)} \less  \norm{\lap\hat{f}}_{H^{\kk-3}(\Om_t^+)},
	\end{equation*}
	which yields
	\begin{equation}\label{vn dot lapGt grad hat p'}
		\begin{split}
			\abs{\vn\vdot\lap_\Gt\grad\hat{p}-\n_+\lap_{\Gt}\hh - \qty(\DD_{\vn}\hat{p})\n_+\kappa}_{H^{\kk-\frac{5}{2}}(\Gt)}  \le Q\qty(\abs*{\vhh}_{H^{\kk-\frac{1}{2}}(\Gt)}).
		\end{split}
	\end{equation}
	It follows from Lemma \ref{lem lap-n} that
	\begin{equation}\label{n - n hat lapGt hh}
		\begin{split}
			\abs{\qty(\n_+ - \n_-)\lap_\Gt\hh}_{H^{\kk-\frac{5}{2}}(\Gt)} \less \abs{\lap_\Gt \hh}_{H^{\kk-\frac{5}{2}}(\Gt)} \less  \abs*{\vhh}_{H^{\kk-\frac{1}{2}}(\Gt)}^2.
		\end{split}
	\end{equation}
	Thus, it is natural to consider $ \n_-\lap_\Gt\hh $ as a replacement of $ \n_+\lap_\Gt\hh $.
	Indeed, by applying the computations \eqref{lap grad q} to $ \Om_t^- $, one can obtain
	\begin{equation}\label{key}
		\begin{split}
			&\vn_- \vdot \lap_\Gt \grad\hh \\
			&\quad=\vn_- \vdot \grad\lap\hh + \qty[\DD_{\vn_-}\hh]{\n}_-\kappa_- \\
			&\qquad- \DD_{\vn_-}\qty[{\h}_-(\kappa_-)\DD_{{\h}_-\vn_{-}}\hh + \DD^2\hh \qty({\h}_-\vn_-, {\h}_-\vn_-)] \\
			&\qquad + \kappa_- \grad \hh \vdot {\n}_-(\vn_-) + 2\DD^2\hh\qty(\vn_-, {\n}_-(\vn_-)),
		\end{split}
	\end{equation}
	where
	\begin{equation}\label{key}
		\vn_- \coloneqq - \vn , \qand \kappa_- \coloneqq -\kappa.
	\end{equation}
	It follows from $ \div \vhh = 0 $ and $ \curl \vhh = \vb{0} $ that
	\begin{equation*}\label{key}
		\grad\hh = \DD_{\vhh}\vhh, \qand \lap\hh = \tr(\DD\vhh)^2 \qin \Om_t^-.
	\end{equation*}
	Thus, by setting
	\begin{equation}\label{key}
		g \coloneqq {\h}_-(\kappa_-)\DD_{{\h}_-\vn_{-}}\hh + \DD^2\hh \qty({\h}_-\vn_-, {\h}_-\vn_-),
	\end{equation}
	one can derive that
	\begin{equation}\label{key}
		g|_\Gt = \lap \hh - \lap_\Gt \hh = \tr(\DD\vhh)^2 - \lap_\Gt \hh,
	\end{equation}
	and
	\begin{equation}\label{key}
		\begin{split}
			\norm{\lap g}_{H^{\kk-3}(\Om_t^-)} \le\, & \norm{\kappa_{\h_-} \vn_{\h_-} \vdot \grad \tr(\DD\vhh)^2}_{H^{\kk-3}(\Om_t^-)} \\
			&+ \norm{\DD\kappa_{\h_-} \DD\vn_{\h_-} \DD\hh}_{H^{\kk-3}(\Om_t^-)} \\
			&+ \norm{\DD^2\hh \DD\kappa_{\h_-}\vn_{\h_-} + \DD^2\hh \kappa_{\h_-} \DD\vn_{\h_-}}_{H^{\kk-3}(\Om_t^-)} \\
			& + \norm{\DD^2\tr(\DD\vhh)^2 (\vn_{\h_-}, \vn_{\h_-})}_{H^{\kk-3}(\Om_t^-)} \\
			& + \norm{\DD^3\hh \DD\vn_{\h_-} \vn_{\h_-}}_{H^{\kk-3}(\Om_t^-)} \\
			\le\, &Q\qty(\abs{\kappa}_{H^{\kk-2}(\Gt)}, \norm{\vhh}_{H^{\kk}(\Om_t^-)}).
		\end{split}
	\end{equation}
	According to the estimates that
	\begin{equation}\label{key}
		\abs{\DD_{\vn_-}g - \n_- (g|_\Gt)}_{H^{\kk-\frac{5}{2}}(\Gt)} \less \norm{\lap g}_{H^{\kk-3}(\Om_t^-)}
	\end{equation}
	and
	\begin{equation}\label{key}
		\abs{\n_- \tr(\DD\vhh)^2}_{H^{\kk-\frac{5}{2}}(\Gt)} \less  \norm{\vhh}_{H^{\kk}(\Om_t^-)}^2,
	\end{equation}
	it can be deduced that
	\begin{equation}\label{n- vdot lap grad hh'}
		\begin{split}
			&\abs{\vn_- \vdot \lap_\Gt \grad\hh + \qty[\DD_{\vn_-}\hh]{\n}_-\kappa -{\n}_-\lap_\Gt\hh}_{H^{\kk-\frac{5}{2}}(\Gt)} \\
			&\quad\le Q\qty(\abs{\kappa}_{H^{\kk-2}(\Gt)}, \norm*{\vhh}_{H^{\kk}(\Om_t^-)}).
		\end{split}
	\end{equation}
	
	The combination of \eqref{vn dot lapGt grad hat p'}-\eqref{n - n hat lapGt hh} and \eqref{n- vdot lap grad hh'} yields:
	\begin{equation}\label{vn vdot lapGt grad p hat}
		\vn \vdot \lap_\Gt \grad\hat{p} = \vn_- \vdot \lap_\Gt\grad\hh + \qty[\DD_{\vn}\hat{p} + \DD_{\vn_-}\hh]\n_+\kappa + r,
	\end{equation}
	where $ r $ satisfies the estimate
	\begin{equation}\label{est reaminder p hat}
		\abs{r}_{H^{\kk-\frac{5}{2}}(\Gt)} \le 
		Q\qty(\abs{\kappa}_{H^{\kk-2}(\Gt)}, \norm*{\vhh}_{H^{\kk}({\Om_t^-})})
	\end{equation}
	
	Next, observe that
	\begin{equation}\label{n- vdot lapGt grad hh}
		\begin{split}
			&\vn_{-} \vdot \lap_{\Gt}\grad\hh = -\vn \vdot \lap_{\Gt}\qty(\DD_{\vhh}\vhh) \\ 
			&\quad= -\lap_\Gt \qty(\vn \vdot \DD_{\vhh}\vhh) + 2 \ip{\II}{\DD^\top\qty(\DD_{\vhh}\vhh)} + \DD_{\vhh}\vhh \vdot \lap_{\Gt} \vn \\
			&\quad= \lap_{\Gt}\qty[\II\qty(\vhh, \vhh)] + 2 \ip{\II}{\DD^\top\qty(\DD_{\vhh}\vhh)} + \grad^\top \kappa \vdot \DD^\top_{\vhh}\vhh - \abs{\II}^2\vn\vdot\DD_{\vhh}\vhh \\
			&\quad= \qty(\DD^\top)^2_{(\vhh, \vhh)}\kappa + \DD^\top_{\DD^\top_\vhh \vhh}\kappa + 2\ip{\DD^\top\II}{\DD^\top(\vhh\otimes\vhh)} + 2 \ip{\II}{\DD^\top\qty(\DD_{\vhh}\vhh)}  - \kappa \qty(\II \vdot \II)(\vhh, \vhh) \\
			&\quad=\DD_{\vhh}\DD_{\vhh}\kappa + 4\ip{\DD^\top_{\vhh}\II}{\DD^\top\vhh}  + 2 \ip{\II}{\DD^\top\qty(\DD_{\vhh}\vhh)} - \kappa \qty(\II \vdot \II)(\vhh, \vhh),
		\end{split}
	\end{equation}
	where the fact that $ \DD^\top \II $ is totally symmetric has been used.
	
	It follows from \eqref{eqn 1 dt-dt+ kappa}, \eqref{est n vdot lapGt grad q'}, \eqref{vn vdot lapGt grad p hat}-\eqref{n- vdot lapGt grad hh}, Lemma \ref{est ii} and the standard product estimates that
	\begin{equation}\label{key}
		\begin{split}
			\wh{\Dt}_{-} \wh{\Dt}_{+} \kappa =\, &\alpha^2 \lap_\Gt\n_+\kappa - \alpha^2\abs{\grad^\top\kappa}^2 + \alpha^2\abs{\II}^2\n_+\kappa + \DD_{\vhh}\DD_{\vhh}\kappa \\
			&+ \qty[\DD_{\vn}\qty(q+\hat{p}) + \DD_{\vn_{-}}\hh]\n_+\kappa + 4\ip{\DD^\top_{\vhh}\II}{\DD^\top\vhh} + \mathfrak{R}',
		\end{split}
	\end{equation}
	for which
	\begin{equation}\label{est mathfrak[R]}
		\abs{\mathfrak{R}'}_{H^{\kk-\frac{5}{2}}(\Gt)} \le
		Q\qty(\abs{\kappa}_{H^{\kk-2}(\Gt)}, \norm{\vv}_{H^{\kk}(\Om_t^+)}, \norm{\vh}_{H^{\kk}(\Om_t^+)}, \norm*{\vhh}_{H^{\kk}({\Om_t^-})}).
	\end{equation}
	Moreover, observe that
	\begin{equation*}\label{key}
		\wh{\Dt}_- \wh{\Dt}_+ = \qty(\Dt - \DD_{\vh}) \qty(\Dt + \DD_{\vh}) = \Dt^2 - \DD_{\vh}\DD_{\vh} + \comm{\Dt}{\DD_{\vh}},
	\end{equation*}
	and
	\begin{equation}\label{comm dt dh}
		\comm{\Dt}{\DD_{\vh}} = \DD_{(\Dt\vh - \DD_{\vh}\vv)} = 0,
	\end{equation}
	so one can deduce from the evolution equation for $ \vh $ that
	\begin{equation*}\label{key}
		\wh{\Dt}_- \wh{\Dt}_+ = \Dt^2 - \DD_{\vh}\DD_\vh.
	\end{equation*}
	
	In conclusion, the following theorem holds:
	\begin{thm}\label{prop evo kappa'}
		There is a generic polynomial $ Q $ determined by $ \Lambda_* $, such that for any solution $ (\Gt, \vv, \vh, \vhh) $ to the problem \eqref{MHD}-\eqref{BC}, with $ \Gt \in \Lambda_* \cap H^{\kk}, (\vv, \vh) \in H^{\kk}(\Om_t^+) $ and $ \vhh \in H^{\kk}(\Om_t^-) $, the evolution equation for the mean curvature $ \kappa $ of the free interface $ \Gt $ is:
		\begin{equation}\label{eqn Dt^2 kappa}
			\begin{split}
				\Dt^2\kappa =\, &\alpha^2 \lap_\Gt\n_+\kappa - \alpha^2\abs{\grad^\top\kappa}^2 + \alpha^2\abs{\II}^2\n_+\kappa + \DD_\vh\DD_\vh\kappa + \DD_{\vhh}\DD_{\vhh}\kappa \\
				&+ \qty{\DD_{\vn}\qty[p_{\vv, \vv} - p_{\vh, \vh} + \h_+ \qty(\txfrac{1}{2}\abs*{\vhh}^2)] - \DD_{\vn}\hh}\n_+\kappa + 4\ip{\DD^\top_{\vhh}\II}{\DD^\top\vhh} + \mathfrak{R}'
			\end{split}
		\end{equation}
		with $\mathfrak{R}'$ satisfying \eqref{est mathfrak[R]}.
	\end{thm}
	
	\subsection{Uniform Estimates}
	Assume that $k \ge 3$ is an integer, $ \vhh \equiv 0 $, and $ (\Gt, \vv, \vh) $ is a solution to the problem \eqref{MHD},\eqref{BC} for time $ t \in [0, T] $ with $ \Gt \in \Lambda_* \cap H^{\kk} $ and $ \vv, \vh \in H^{\kk}(\Om_t^+) $. 
	Suppose further that there exists a constant $\lambda_0 > 0$ so that the following Rayleigh-Taylor sign condition holds for all $ t \in [0, T] $:
	\begin{equation}\label{RT consition}
		\mathfrak{t} \coloneqq -\DD_{\vn}\qty(p_{\vv, \vv} - p_{\vh, \vh}) \ge \lambda_0 \qq{on} \Gt.
	\end{equation}
	
	Consider the energy functional:
	\begin{equation}\label{E_alpha}
		\begin{split}
			\bar{E}_\alpha (t) \coloneqq \int_\Gt &\abs{\qty(-\n_+^\frac{1}{2}\lap_\Gt\n_+^\frac{1}{2})^{\frac{k-2}{2}}\n_+^\frac{1}{2}(\Dt\kappa)}^2 + \alpha^2 \abs{\qty(-\n_+^\frac{1}{2}\lap_\Gt\n_+^\frac{1}{2})^{\frac{k-1}{2}}\n_+^\frac{1}{2}\kappa}^2 \\
			&+\abs{\qty(-\n_+^\frac{1}{2}\lap_\Gt\n_+^\frac{1}{2})^{\frac{k-2}{2}}\n_+^\frac{1}{2}(\DD_{\vh}\kappa)}^2 + \mathfrak{t}\abs{\qty(-\n_+^\frac{1}{2}\lap_\Gt\n_+^\frac{1}{2})^\frac{k-2}{2}\n_+ \kappa}^2 \dd{S_t}.
		\end{split}
	\end{equation}
	Then, the following lemma holds:
	\begin{lem}
		Suppose that $k \ge 3$ is an integer, $0 \le \alpha \le 1$, and $\qty(\Gt, \vv, \vh)$ is a solution to the problem \eqref{MHD}, \eqref{BC} for $t \in [0, T]$ with $\Gt \in \Lambda_{*} \cap H^{\kk+1}, (\vv, \vhh) \in H^{\kk}(\Om_t^+)$. Then there exists a generic polynomial $Q$ depending only on $\Lambda_{*}$, so that for all $t \in [0, T]$
		\begin{equation}\label{est E_alpha}
			\abs{\dv{t}\bar{E}_\alpha (t)} \le Q\qty(\abs{\qty(\Dt\kappa, \DD_\vh\kappa)}_{H^{\kk-\frac{5}{2}}(\Gt)}, \alpha\abs{\kappa}_{H^{\kk-1}(\Gt)}, \abs{\kappa}_{H^{\kk-2}(\Gt)}, \norm{\qty(\vv, \vh)}_{H^{\kk}(\Om_t^+)} ).
		\end{equation}
	\end{lem}
	\begin{proof}
		For the simplicity of notations, set
		\begin{equation*}
			\opd \coloneqq \qty(-\n_+^\frac{1}{2}\lap_\Gt\n_+^{\frac{1}{2}})^{\frac{1}{2}}.
		\end{equation*}
		
		Commuting $\Dt$ with $\opd^{k-2}\n_+^\frac{1}{2}$, one can derive from Lemma \ref{Dt comm est lemma} that (see \cite[Lemma 5.1]{Liu-Xin2023} for more details)
		\begin{equation} \label{energy est step 1}
			\begin{split}
				&\abs{\int_\Gt \qty(\opd^{k-2}\n_+^\frac{1}{2}\Dt\kappa) \times \qty(\opd^{k-2}\n_+^\frac{1}{2}\Dt^2\kappa) \dd{S_t} -\frac{1}{2}\dv{t}\int_\Gt \abs{\opd^{k-2}\n_+^{\frac{1}{2}}\Dt\kappa}^2 \dd{S_t} } \\
				&\quad \le Q\qty(\abs{\Dt\kappa}_{H^{\kk-\frac{5}{2}}(\Gt)}, \abs{\vv}_{H^{\kk-\frac{1}{2}}(\Gt)}).
			\end{split}
		\end{equation}
		Plugging \eqref{eqn Dt^2 kappa} into \eqref{energy est step 1} yields
		\begin{equation}\label{eqn I_0}
			\begin{split}
				\int_\Gt &\qty[\opd^{k-2}\n_+^{\frac{1}{2}}\qty(\Dt^2\kappa)] \cdot\qty[\opd^{k-2}\n_+^{\frac{1}{2}}(\Dt\kappa)] \dd{S_t} \\
				=\, &\int_\Gt \alpha^2\qty[\opd^{k}\n_+^{\frac{1}{2}}\kappa] \cdot\qty[\opd^{k-2}\n_+^{\frac{1}{2}}(\Dt\kappa)] \dd{S_t} \\
				&- \int_{\Gt} \alpha^2 \qty[\opd^{k-2}\n_+^{\frac{1}{2}}\abs{\grad^\top\kappa}^2] \cdot\qty[\opd^{k-2}\n_+^{\frac{1}{2}}(\Dt\kappa)] \dd{S_t} \\
				&+ \int_\Gt \alpha^2\qty[\opd^{k-2}\n_+^{\frac{1}{2}}\qty(\abs{\II}^2\n_+\kappa)] \cdot\qty[\opd^{k-2}\n_+^{\frac{1}{2}}(\Dt\kappa)] \dd{S_t}  \\
				&+ \int_\Gt \qty[\opd^{k-2}\n_+^{\frac{1}{2}}\qty(\DD_\vh\DD_\vh\kappa)] \cdot\qty[\opd^{k-2}\n_+^{\frac{1}{2}}(\Dt\kappa)] \dd{S_t} \\
				&- \int_\Gt \qty[\opd^{k-2}\n_+^{\frac{1}{2}}\qty(\mathfrak{t}\n_+\kappa)] \cdot\qty[\opd^{k-2}\n_+^{\frac{1}{2}}(\Dt\kappa)] \dd{S_t}  \\
				&+ \int_\Gt \qty[\opd^{k-2}\n_+^{\frac{1}{2}}\mathfrak{R}'] \cdot\qty[\opd^{k-2}\n_+^{\frac{1}{2}}(\Dt\kappa)] \dd{S_t} \\
				=:\, &I_1 + \cdots + I_6.
			\end{split}
		\end{equation}
		It is not hard to derive from Lemmas~\ref{est ii}-\ref{lem product est} and \eqref{est mathfrak[R]} that
		\begin{equation}
			\abs{I_2} + \abs{I_3} + \abs{I_6} \le Q \qty(\abs{\Dt\kappa}_{H^{\kk-\frac{5}{2}(\Gt)}}, \alpha\abs{\kappa}_{H^{\kk-\frac{3}{2}}(\Gt)}, \norm{\qty(\vv, \vh)}_{H^{\kk}(\Om_t^+)}).
		\end{equation}
		For $I_1$, it follows from the self-adjointness of $\opd$ and the commutator estimates that
		\begin{equation}
			\abs{I_1 - \frac{1}{2}\dv{t}\int_\Gt\alpha^2\abs{\opd^{k-1}\n_+^\frac{1}{2}\kappa}^2\dd{S_t}} \le Q\qty(\alpha\abs{\kappa}_{H^{\kk-1}(\Gt)}, \abs{\vv}_{H^{\kk-\frac{1}{2}}(\Gt)}).
		\end{equation}
		To estimate $I_4$, one can derive from \eqref{formula int by parts Gt} and \eqref{comm dt dh} that
		\begin{equation}
			\abs{I_4 - \frac{1}{2}\dv{t}\int_\Gt \abs{\opd^{k-2}\n_+^\frac{1}{2}(\DD_{\vh}\kappa)}^2 \dd{S_t}} \le Q\qty(\abs{\qty(\Dt\kappa, \DD_\vh\kappa)}_{H^{\kk-\frac{5}{2}}(\Gt)}, \abs{\qty(\vv, \vh)}_{H^{\kk-\frac{1}{2}}(\Gt)}).
		\end{equation}
		To estimate  $I_5$, one can derive from \eqref{RT consition}, \eqref{def p_ab^+} and Lemma~\ref{Dt comm est lemma} that
		\begin{equation*}
			\abs{\Dt\mathfrak{t}}_{H^{\kk-\frac{5}{2}}(\Gt)} + \abs{\mathfrak{t}}_{H^{\kk-\frac{3}{2}}(\Gt)} \le Q\qty(\norm{\qty(\vv, \vh)}_{H^{\kk}(\Om_t^+)}).
		\end{equation*}
		Therefore, standard commutator and product estimates imply that
		\begin{equation}\label{est I5}
			\abs{I_5 - \frac{1}{2}\dv{t}\int_\Gt\mathfrak{t}\abs{\opd^{k-2}\n_+\kappa}^2\dd{S_t}} \le Q\qty(\abs{\kappa}_{H^{\kk-2}(\Gt)}, \norm{\qty(\vv, \vh)}_{H^{\kk}(\Om_t^+)}).
		\end{equation}
		In conclusion, \eqref{est E_alpha} follows from  \eqref{energy est step 1}-\eqref{est I5}
	\end{proof}
	
	Now, observe that \eqref{RT consition}-\eqref{E_alpha} yield
	\begin{equation}
		\begin{split}
			&\abs{\Dt\kappa}_{H^{\kk-\frac{5}{2}}(\Gt)}^2 + \alpha^2 \abs{\kappa}_{H^{\kk-1}(\Gt)}^2 + \abs{\DD_\vh\kappa}_{H^{\kk-\frac{5}{2}}(\Gt)} + \lambda_0 \abs{\kappa}_{H^{\kk-2}(\Gt)}^2 \\
			&\quad\less \bar{E}_\alpha (t) + \abs{\int_\Gt \Dt\kappa\dd{S_t}}^2 + \abs{\int_\Gt \DD_\vh \kappa \dd{S_t}}^2 + \abs{\int_\Gt\kappa\dd{S_t}}^2. 
		\end{split}
	\end{equation}
	On the other hand, \eqref{eqn dt kappa} and the fact that $\Gt \in \Lambda_{*}$ yield
	\begin{equation}
		\abs{\int_\Gt \Dt\kappa\dd{S_t}}^2 + \abs{\int_\Gt \DD_\vh \kappa \dd{S_t}}^2 + \abs{\int_\Gt\kappa\dd{S_t}}^2 \le Q\qty(\norm{\vv}_{H^{\kk}(\Om_t^+)}, \norm{\vh}_{H^{\kk}(\Om_t^+)}),
	\end{equation}
	where $Q$ is a generic polynomial determined by $\Lambda_*$.
	Thereby, it follows from the Gr\"{o}nwall inequality that
	\begin{equation}\label{est E_alpha'}
		\bar{E}_\alpha (t) \le \bar{E}_\alpha(0) + \int_0^t Q\qty(\bar{E}_\alpha(t'), \norm{(\vv, \vh)}_{H^{\kk}(\Om_{t'}^+)}) \dd{t'},
	\end{equation}
	where $Q$ is a generic polynomial determined by $\Lambda_{*}$ and $\lambda_0$.
	
	Now, consider the energy functionals:
	\begin{equation}
		\bar{E}_0 (t) \coloneqq \int_{\Om_t^+} \abs{\vv}^2 + \abs{\vh}^2 \dd{x} + \int_\Gt \alpha^2 \dd{S_t},
	\end{equation}
	\begin{equation}
		\bar{E}_1 (t) \coloneqq \norm{\vom}_{H^{\kk-1}(\Om_t^+)}^2 + \norm{\vj}_{H^{\kk-1}(\Om_t^+)}^2,
	\end{equation}
	and
	\begin{equation}
		\bar{E}(t) \coloneqq \bar{E}_\alpha (t) + \bar{E}_0(t) + \bar{E}_1(t).
	\end{equation}
	
	Since it is assumed that $\vhh \equiv \vb{0}$, \eqref{energy conserv 1}-\eqref{energy conserv 2} yield that
	\begin{equation}
		\dv{t} \bar{E}_0 (t) \equiv 0.
	\end{equation}
	Moreover, one can derive from \eqref{eqn pdt vom}-\eqref{eqn pdt vj}, \eqref{eqn xi}-\eqref{eqn eta}, and the standard estimates for transport equations (c.f. \cite{Sun-Wang-Zhang2018} or \cite{Liu-Xin2023}) that
	\begin{equation}\label{est E_1}
		\abs{\dv{t}\bar{E}_1(t)} \le Q\qty(\norm{\vv}_{H^{\kk}(\Om_t^+)}, \norm{\vh}_{H^{\kk}(\Om_t^+)}),
	\end{equation}
	where $Q$ is a generic constant determined by $\Lambda_{*}$.
	
	Standard div-curl estimates (see \cite{Cheng-Shkoller2017}) imply that
	\begin{equation}\label{vani est v'}
		\norm{\vv}_{H^{\kk}(\Om_t^+)} \less \abs{\vn \vdot \lap_\Gt \vv}_{H^{\kk-\frac{5}{2}}(\Gt)} + \norm{\vom}_{H^{\kk-1}(\Om_t^+)} + \norm{\vv}_{L^2(\Om_t^+)},
	\end{equation}
	and
	\begin{equation}\label{vani est h'}
		\norm{\vh}_{H^{\kk}(\Om_t^+)} \less \abs{\vn \vdot \lap_\Gt \vh}_{H^{\kk-\frac{5}{2}}(\Gt)} + \norm{\vj}_{H^{\kk-1}(\Om_t^+)} + \norm{\vh}_{L^2(\Om_t^+)}.
	\end{equation}
	
	On the other hand, \eqref{eqn dt kappa} yields
	\begin{equation*}
		-\vn\vdot\lap_\Gt\vv = \Dt\kappa + 2\ip{\II}{\qty(\DD\vv)^\top},
	\end{equation*}
	which implies that
	\begin{equation}\label{est lap_gt vv}
		\abs{\vn\vdot\lap_\Gt\vv}_{H^{\kk-\frac{5}{2}}(\Gt)} \less \abs{\Dt\kappa}_{H^{\kk-\frac{5}{2}}(\Gt)} + \norm{\vv}_{H^{\kk-1}(\Om_t^+)}.
	\end{equation}
	Then, it follows from the interpolation inequality that
	\begin{equation}\label{est vv}
		\norm{\vv}_{H^{\kk}(\Om_t^+)}^2 \le Q\qty(\bar{E}).
	\end{equation}
	Furthermore, \eqref{lap vn} and the fact that $\vh \vdot \vn \equiv 0$ lead to
	\begin{equation*}
		-\vn\vdot\lap_\Gt\vh = \DD_\vh \kappa + 2 \ip{\qty(\DD\vh)^\top}{\II},
	\end{equation*}
	which yields that
	\begin{equation}\label{est lap_gt vh}
		\abs{\vn\vdot\lap_\Gt\vh}_{H^{\kk-\frac{5}{2}}(\Gt)} \less \abs{\DD_\vh\kappa}_{H^{\kk-\frac{5}{2}}(\Gt)} + \norm{\vh}_{H^{\kk-1}(\Om_t^+)}.
	\end{equation}
	Hence,
	\begin{equation}\label{est vh}
		\norm{\vh}_{H^{\kk}(\Om_t^+)}^2 \le Q\qty(\bar{E}).
	\end{equation}
	
	In conclusion, \eqref{est E_alpha'}-\eqref{est E_1}, \eqref{est vv}, and \eqref{est vh} yield that
	\begin{equation}\label{RT energy est}
		\bar{E}(t) \le \bar{E}(0) + \int_0^t Q\qty(\bar{E}(t')) \dd{t'},
	\end{equation}
	where $Q$ is a generic polynomial depending only on $\Lambda_{*}$ and $\lambda_0$.
	
	Therefore, Theorem~\ref{thm vani st RT} follows from \eqref{RT energy est} and standard weak convergence arguments.
	
	\section{Stability Effect of Non-Collinear Magnetic Fields}
	
	\subsection{Interfaces, Coordinates and Div-Curl Systems}\label{sec prelimi'}
	From now on, it is assumed that $ \Om \coloneqq \mathbb{T}^2\times(-1, 1) $ and $ \Om_t^+ $ has a solid boundary $ \Gamma_+ \coloneqq \mathbb{T}^2 \times \{+1\} $. Hence, some statements in \textsection~\ref{sec harmonic coord} and \textsection~\ref{sec div-cul system} need slight changes in order to be compatible to the new topology of $ \Om_t^\pm $. More precisely, the harmonic coordinate maps introduced in \textsection~\ref{sec harmonic coord} are now replaced by
	\begin{equation}\label{key}
		\begin{cases*}
			\lap_y \X_\Gamma^\pm = 0 &for $ y \in \Om_*^\pm $, \\
			\X_\Gamma^\pm (z) = \Phi_\Gamma(z) &for $ z \in \Gamma_* $, \\
			\X_\Gamma^\pm (z) = z &for $ z \in \Gamma_\pm$.
		\end{cases*}
	\end{equation}
	Similarly, the definitions of harmonic extensions of a function $ f $ defined on $ \Gamma $ are modified to
	\begin{equation}\label{eqn harm ext'}
		\begin{cases*}
			\lap \h_\pm f = 0 \qfor x \in \Om_\Gamma^\pm, \\
			\h_\pm f = f \qfor x \in \Gamma, \\
			\DD_{\wt{\vn}_\pm} \h_\pm f = 0 \qfor x \in \Gamma_\pm.
		\end{cases*}
	\end{equation}
	The Dirichlet-Neumann operators are also defined by \eqref{def DN op} for $ \h_\pm $ given by \eqref{eqn harm ext'}.
	
	Therefore, Lemmas \ref{lem composition harm coordi} - \ref{lem lap-n}, and those properties of the Dirichlet-Neumann operators introduced in \textsection~\ref{sec harmonic coord} still hold.
	
	As for the div-curl systems, due to the different topology, we introduce the following modifications of Theorems \ref{thm div-curl} and \ref{thm div-curl"} (see \cite{Cheng-Shkoller2017, Sun-Wang-Zhang2018, Sun-Wang-Zhang2019} for detailed proofs).
	
		Assume that $ \Gamma \in \Lambda_{*} $ is diffeomorphic to $ \mathbb{T}^2 $, with
		\begin{equation}\label{surface away condi}
			\dist(\Gamma, \Gamma_\pm) \ge c_0 > 0
		\end{equation}
		for some positive constant $ c_0 $.
		Let $ \vb{f}, g \in H^{l-1}(\Om^+) $ and $ h \in H^{l-\frac{1}{2}}(\Gamma) $ be given, satisfying the compatibility condition
		\begin{equation*}\label{key}
			\int_{\Om^+} g \dd{x} = \int_{\Gamma} h \dd{S},
		\end{equation*}
		and
		\begin{equation}\label{key}
			\div \vb{f} = 0 \text{ in } \Om^+ \qc  \int_{\Gamma_+} \vb{f} \vdot \wt{\vn}_+ \dd{S} = 0.
		\end{equation}
		Then, for $ 1 \le l \le \kk-1 $, the following div-curl problem:
		\begin{equation}\label{key}
			\begin{cases*}
				\curl \vbu = \vb{f} &in $ \Om^+ $, \\
				\div \vbu = g &in $ \Om^+ $, \\
				\vbu \vdot \vn = h &on $ \Gamma $, \\
				\vbu \vdot \wt{\vn}_+ = 0 \qc \int_{\Gamma_+} \vbu \dd{S} = \va{\mathfrak{u}} &on $ \Gamma_+ $
			\end{cases*}
		\end{equation}
		has a unique solution $ \vbu \in H^{l}(\Om^+) $ with
		\begin{equation}\label{key}
			\norm{\vbu}_{H^l(\Om^+)} \le C\qty(\abs{\ka}_{\H{\kk-\frac{3}{2}}}, c_0) \times \qty( \norm{\vf}_{H^{l-1}(\Om^+)} + \norm{g}_{H^{l-1}(\Om^+)} + \abs{h}_{H^{l-\frac{1}{2}}(\Gamma)} + \abs{\va{\mathfrak{u}}}).
		\end{equation}
	
	By solving div-curl problems with the notations:
	\begin{equation}\label{key}
		\va{\mathfrak{v}} \coloneqq \int_{\Gamma_+} \vv \dd{S_+}\qc \va{\mathfrak{h}} \coloneqq \int_{\Gamma_+} \vh \dd{S_+},
	\end{equation}
	the velocity $ \vv $ and magnetic field $ \vh $ are determined by $ (\Gt, \vom, \vj, \va{\mathfrak{v}}, \va{\mathfrak{h}}) $ uniquely.
	
	For the magnetic field in the vacuum, if $ \vu{J} \in \mathrm{T}\Gamma_- $, $ \Div_{\Gamma_-} \vu{J} = 0 $, and $ \Gt \in \Lambda_{*} $ satisfies
	\begin{equation}\label{key}
		\dist\qty(\Gt, \Gamma_\pm) \ge c_0 > 0,
	\end{equation} 
	then the following div-curl problem
	\begin{equation}\label{key}
		\begin{cases*}
			\div \vhh = 0\qc \curl \vhh = \vb{0} &in $ \Om_t^- $, \\
			\vhh \vdot \vn = \hat{\theta} &on $ \Gt $, \\
			\wt{\vn}_- \cp \vhh= \vu{J} &on $ \Gamma_- $
		\end{cases*}
	\end{equation}
	has a unique solution $ \vhh $ with the following estimate (c.f. \cite[pp. 93-94]{Sun-Wang-Zhang2019}):
	\begin{equation}\label{key}
		\begin{split}
			\norm*{\vhh}_{H^{s}(\Om_t^-)} \le C(c_0, \abs{\ka}_{H^{\kk-\frac{3}{2}}}) \times \qty(\abs*{\hat{\theta}}_{H^{s-\frac{1}{2}}(\Gt)} + \abs*{\vu{J}}_{H^{s-\frac{1}{2}}(\Gamma_-)}),
		\end{split}
	\end{equation}
	for $ \Gt \in \Lambda_{*} $ and $ 1 \le s \le \kk $.
	
	Thus, $\vhh$ can also be determined uniquely by $\ka$ and $\vu{J}$.
	
	\subsection{Reformulation of the Problem}
	
	Here, one replaces \eqref{eqn pd beta v} with
	\begin{equation}\label{key}
		\pd_\beta \vv_* = \opb(\ka)\pd^2_{t\beta}\ka + \opf(\ka)\pd_\beta\vom_* + \opg(\ka, \pd_t\ka, \vom_*)\pd_\beta\ka + \vb{S}(\ka)\pd_\beta\va{\mathfrak{v}},
	\end{equation}
	where the new term $ \vb{S} (\ka)\pd_\beta \va{\mathfrak{v}} $ is the pull-back to $ \Gs $ of the solution to the following div-curl problem:
	\begin{equation}\label{def opS}
		\begin{cases*}
			\div \vb{y} = 0 \qc \curl \vb{y}= \vb{0} &in $ \Om^+_t $, \\
			\vb{y} \vdot \vn = 0 &on $ \Gt $, \\
			\int_{\Gamma_+} \vb{y} \dd{\wt{S}} = \pd_\beta \va{\mathfrak{v}} &on $ \Gamma_+ $.
		\end{cases*}
	\end{equation}
	Therefore, since $ \pd_\beta \va{\mathfrak{v}} $ is a constant tangential vector on $ \Gamma_+ $ for each fixed $ \beta $, the following estimates can be derived from the div-curl ones:
	\begin{equation}\label{est S}
		\abs{\vb{S} (\ka)}_{\LL\qty(\R^2; \H{s})} \le C_* \qfor \frac{3}{2} \le s \le \kk-\frac{3}{2},
	\end{equation}
	and
	\begin{equation}\label{est var S}
		\abs{\var \vb{S}(\ka)}_{\LL\qty(\H{\kk-\frac{5}{2}}; \LL\qty(\R^2; \H{s'}))} \le C_* \qfor \frac{3}{2} \le s' \le \kk-\frac{3}{2}.
	\end{equation}
	
	Due to the change of topology of $ \Om^+_t $, the definition of $ p_{\vb{a}, \vb{b}} $ is replaced by:
	\begin{equation}\label{eqn p_ab '}
		\begin{cases*}
			-\lap p_{\vb{a}, \vb{b}} = \tr(\DD\vb{a} \vdot \DD\vb{b}) &in $ \Om_t^+ $, \\
			p_{\vb{a}, \vb{b}} = 0 &on $ \Gt $, \\
			\DD_{\wt{\vn}_+} p_{\vb{a}, \vb{b}} = \wt{\II}_+(\vb{a}, \vb{b}) &on $ \Gamma_+ $,
		\end{cases*}
	\end{equation}
	where the solenoidal vector fields $ \vb{a}, \vb{b} $ satisfy $ \vb{a} \vdot \wt{\vn}_+ = 0 = \vb{b} \vdot \wt{\vn}_+  $ on $ \Gamma_+ $.
	
	Therefore, the evolution equation for $ \ka $ can be obtained via a slight modification of \eqref{eqn pdt2 ka "}:
	\begin{equation}\label{key}
		\begin{split}
			&\pd^2_{tt} \ka + \opC_\alpha (\ka, \pd_t\ka, \vv_*, \vh_*, \vhh_*)\ka - \opF(\ka)\pd_t\vom_*  - \opG(\ka, \pd_t\ka, \vom_*, \vj_*, \va{\mathfrak{v}}, \va{\mathfrak{h}}, \vu{J}) - \mathscr{S}(\ka)\pd_t\va{\mathfrak{v}} \\
			&\quad = \qty[\mathrm{I} + \opB(\ka)]^{-1}\qty{\qty[-\lap_\Gt \qty(\vW \vdot \vn) + \vW \vdot \lap_\Gt \vn + a^2 \dfrac{\vW \vdot \vn}{\vn \vdot (\vnu \circ \Phi_\Gt^{-1})}] \circ \Phi_\Gt}.
		\end{split}
	\end{equation}
	Furthermore, Lemma \ref{lem 3.4} holds for $ k \ge 3, \alpha = 0 $, with a slight change of \eqref{est var opF"} and \eqref{est var opG"}:
	\begin{equation}\label{est var opF""}
		\abs{\var\opF(\ka)}_{\LL\qty[\H{\kk-\frac{5}{2}}; \LL\qty(H^{\kk-\frac{7}{2}}(\Om_*^+); \H{\kk-\frac{7}{2}})]} \le C_*, \tag{\ref{est var opF"}'}
	\end{equation}
	\begin{equation}\label{est var opG""}
		\begin{split}
			&\hspace{-2em}\abs{\var'\opG}_{\LL\qty[\H{\kk-\frac{5}{2}}\times\H{\kk-\frac{7}{2}}\times H^{\kk-2}(\Om_*^+) \times H^{\kk-2}(\Om_*^+) \times \R^2 \times \R^2; \H{\kk-\frac{7}{2}}]} \\
			\le\, &a^2 Q\qty(\abs{\pd_t\ka}_{\H{\kk-\frac{5}{2}}}, \norm{\vom_*}_{H^{\kk-1}(\Om_*^+)}, \norm{\vj_*}_{H^{\kk-1}(\Om_*^+)}),
		\end{split} \tag{\ref{est var opG"}'}
	\end{equation}
	where $ \var'\opG $ is the variation of $ \opG $ with respect to the first six variables.
	In addition, the operator $ \mathscr{S}(\ka) $ satisfies the estimates
	\begin{equation}\label{est opS""}
		\abs{\mathscr{S}(\ka)}_{\LL\qty(\R^2; \H{\kk-\frac{5}{2}})} \le Q\qty(\abs{\ka}_{\H{\kk-\frac{3}{2}}}),
	\end{equation}
	and
	\begin{equation}\label{est var opS""}
		\abs{\var\mathscr{S}(\ka)}_{\LL\qty[\H{\kk-\frac{7}{2}}; \LL\qty(\R^2; \H{\kk-\frac{7}{2}})]} \le Q\qty(\abs{\ka}_{H^{\kk-\frac{3}{2}}}).
	\end{equation}
	
	The discussions on the evolution equations for the current and vorticity remain the same as those in \textsection~\ref{sec curr-vorticity eqn}.
	
	\subsection{Linear Systems}\label{sec noncol linear}
	Similar to the arguments in \textsection~\ref{sec linear ka}, assume that $ \Gs \in H^{\kk+\frac{1}{2}} (k\ge 3) $ is a reference hypersurface, and $ \Lambda_* $  defined by \eqref{def lambda*} satisfies all the properties discussed in the preliminary. Suppose that there are a family of hypersurfaces $ \Gt \in \Lambda_* $ and three tangential vector fields $ \vv_*, \vh_*, \vhh_* : \Gs \to \mathrm{T}\Gs $ satisfying:
	\begin{equation}\label{key}
		\ka \in C^0\qty{[0, T]; \H{\kk-\frac{3}{2}}} \cap C^1\qty{[0, T]; B_{\delta_1} \subset \H{\kk-\frac{5}{2}}}, 
	\end{equation}
	and
	\begin{equation}\label{key}
		\vv_{*}, \vh_{*}, \vhh_* \in C^0\qty{[0, T]; \H{\kk-\frac{1}{2}}} \cap C^1\qty{[0, T]; \H{\kk-\frac{3}{2}}}. 
	\end{equation}
	Furthermore, assume that there are positive constants $ c_0$ and $ \mathfrak{s}_0 $, so that the following two relations hold uniformly on $ [0, T] $:
	\begin{equation}\label{key}
		\Upsilon(\vh_*, \vhh_*) \ge \mathfrak{s}_0 > 0,
		\qand
		\dist\qty(\Gt, \Gamma_\pm) \ge c_0 > 0.
	\end{equation}
	
	The positive constants $ L_1$ and $ L_2 $ are defined to be:
	\begin{equation}\label{key}
		\sup_{t\in[0, T]} \qty{\abs{\ka (t)}_{\H{\kk-\frac{3}{2}}},  \abs{\pd_t \ka(t)}_{\H{\kk-\frac{5}{2}}}, \abs*{(\vv_{*}(t), \vh_{*}(t), \vhh_*(t))}_{\H{\kk-\frac{1}{2}}}} \le L_1,
	\end{equation}
	\begin{equation}\label{key}
		\sup_{t\in [0, T]} \abs*{(\pd_t\vv_{*}(t), \pd_t\vh_{*}(t), \pd_t\vhh_*(t))}_{\H{\kk-\frac{3}{2}}} \le L_2.
	\end{equation}
	
	Consider the linear initial value problem:
	\begin{equation}\label{eqn linear 1""}
		\begin{cases}
			\pd^2_{tt}\f + \opC_0(\ka, \pd_t\ka, \vv_*, \vh_*, \vhh_*)\f = \g, \\
			\f(0) = \f_0, \quad \pd_t \f(0) = \f_1,
		\end{cases}
	\end{equation}
	where $ \f_0, \f_1, \g(t) : \Gs \to \R $ are three given functions, and $ \opC_0 $ is given by \eqref{def opC"} with $ \alpha = 0 $.
	
	Define the following energy functional:
	\begin{equation}\label{eqn linear 2'}
		\begin{split}
			\wt{E}_l (t, \f, \pd_t\f) \coloneqq \int_\Gt &\abs{\qty(-\n_+^\frac{1}{2}\lap_\Gt\n_+^\frac{1}{2})^\frac{l}{2}\n_+^\frac{1}{2}\qty[(\pd_t\f + \DD_{\vv_*}\f) \circ \Phi_\Gt^{-1}]}^2 \\
			&+ \abs{\qty(-\n_+^\frac{1}{2}\lap_\Gt\n_+^\frac{1}{2})^\frac{l}{2}\n_+^\frac{1}{2}\qty[(\DD_{\vh_*}\f) \circ \Phi_\Gt^{-1}]}^2 \\
			&+ \abs{\qty(-\n_+^\frac{1}{2}\lap_\Gt\n_+^\frac{1}{2})^\frac{l}{2}\n_+^\frac{1}{2}\qty[(\DD_{\vhh_*}\f) \circ \Phi_\Gt^{-1}]}^2 \dd{S_t}.
		\end{split}
	\end{equation}
	Then it follows from the arguments in \S~\ref{sec linear ka} that for any integer $ 0 \le  l \le k-2 $ and $ 0 \le t \le T $, it holds that
	\begin{equation}\label{est E_l'}
		\begin{split}
			&\hspace{-2em}\wt{E}_l (t, \f, \pd_t\f) - \wt{E}_l (0, \f_0, \f_1) \\
			\le&Q(L_1, L_2)\int_0^t \qty(\abs{\f(s)}_{\H{\frac{3}{2}l + \frac{3}{2}}} + \abs{\pd_t\f(s)}_{\H{\frac{3}{2}l + \frac{1}{2}}} + \abs{\g(s)}_{\H{\frac{3}{2}l + \frac{1}{2}}}) \times \\
			&\hspace{8em} \times \qty(\abs{\f(s)}_{\H{\frac{3}{2}l + \frac{3}{2}}} + \abs{\pd_t\f(s)}_{\H{\frac{3}{2}l + \frac{1}{2}}}) \dd{s},
		\end{split}
	\end{equation}
	where $ Q $ is a generic polynomial determined by $ \Lambda_* $.
	
	Since $\Upsilon(\vh_*, \vhh_*) \ge \mathfrak{s}_0 > 0$, it is clear that $\vh_*$ and $\vhh_*$ form a global frame of $\Gt$, and the following estimate holds for each function $f \colon \Gt \to \R$ and $1 \le s \le \kk-\frac{1}{2}$:
	\begin{equation}
		\abs{f}_{H^{s}(\Gt)} \le Q(L_1, \mathfrak{s}_0^{-1}) \times \qty( \abs{f}_{L^2(\Gt)} + \abs{\DD_{\vh} f}_{H^{s-1}(\Gt)} + \abs{\DD_\vhh f}_{H^{s-1}(\Gt)}),
	\end{equation}
	where $Q$ is a generic polynomial determined by $\Lambda_*$. Therefore, it is routine to prove the following proposition:
	\begin{prop}
		For $ 0 \le l \le k-2 $, $ T \le C(L_1, L_2, \mathfrak{s}_0) $ and $ \g \in C^0 \qty([0, T]; H^{\frac{3}{2}l + \frac{1}{2}}(\Gs)) $, the problem \eqref{eqn linear 2'} is well-posed in $ C^0\qty([0, T]; \H{\frac{3}{2}l+\frac{3}{2}}) \cap C^1\qty([0, T]; \H{\frac{3}{2}l+\frac{1}{2}}) $, and the following energy estimate holds:
		\begin{equation}\label{est linear eqn ka""}
			\begin{split}
				&\hspace{-1em}\abs{\f(t)}_{\H{\frac{3}{2}l+\frac{3}{2}}}^2 + \abs{\pd_t\f(t)}_{\H{\frac{3}{2}l+\frac{1}{2}}}^2 \\
				\le\, &C_* e^{Q(L_1, L_2, \mathfrak{s}_0^{-1})t} \qty( \abs{\f_0}_{\H{\frac{3}{2}l+\frac{3}{2}}}^2 + \abs{\f_1}_{\H{\frac{3}{2}l+\frac{1}{2}}}^2 + \int_0^t \abs{\g(t')}^2_{\H{\frac{3}{2}l+\frac{1}{2}}} \dd{t'}).
			\end{split}
		\end{equation}
	\end{prop}
	
	For the linearized current-vorticity systems, the arguments in \textsection~\ref{sec linear current vortex} are still valid.
	
	\subsection{Nonlinear Problems}
	As in \textsection~\ref{sec nonlinear}, take a reference hypersurface $ \Gs \in H^{\kk+\frac{1}{2}} $ and $ \delta_0 > 0  $, so that
	\begin{equation*}
		\Lambda_* \coloneqq \Lambda \qty(\Gs, \kk-\frac{1}{2}, \delta_0)
	\end{equation*}
	satisfies all the properties discussed in the preliminary. Furthermore, assume that there is a constant $ c_0 > 0 $ so that \eqref{surface away condi} holds for $ \Gs $. We shall solve the nonlinear problem by iterations on the linearized problems in the space:
	\begin{equation*}\label{key}
		\begin{split}
			&\ka \in C^0\qty([0, T]; \H{\kk-\frac{3}{2}}) \cap C^1\qty([0, T]; B_{\delta_1}\subset\H{\kk-\frac{5}{2}}) \cap C^2\qty([0, T]; \H{\kk-\frac{7}{2}}); \\
			&\vom_{*}, \vj_{*} \in  C^0\qty([0, T]; H^{\kk-1}(\Om_*^\pm)) \cap C^1\qty([0, T]; H^{\kk-2}(\Om_*^\pm)); \\
			&\va{\mathfrak{v}}, \va{\mathfrak{h}} \in C^1\qty([0, T]; \R^2).
		\end{split}
	\end{equation*}
	
	\subsubsection{Fluid Region, Velocity and Magnetic Fields}
	
	As discussed in \textsection~\ref{section recovery}, the bulk region, velocity and magnetic fields can be obtained via solving the div-curl problems:
	\begin{equation}\label{div-curl nonlinear v""}
		\begin{cases*}
			\div \vv = 0 \qc \curl \vv = \bar{\vom} &in $ \Om^+_t $, \\
			\vv \vdot \vn = \vn \vdot (\pd_t\gt \vnu) \circ\Phi_\Gt^{-1} &on $ \Gt $, \\
			\vv \vdot \wt{\vn}_+ = 0 \qc \int_{\Gamma_+} \vv \dd{S} = \va{\mathfrak{v}}  &on $\Gamma_+ $;
		\end{cases*}
	\end{equation}
	\begin{equation}\label{div-curl nonlinear h""}
		\begin{cases*}
			\div \vh = 0 \qc \curl \vh = \bar{\vj} &in $ \Om_t^+ $, \\
			\vh \vdot \vn = 0 &on $ \Gt $, \\
			\vh \vdot \wt{\vn}_+ = 0 \qc \int_{\Gamma_+} \vh \dd{S} = \va{\mathfrak{h}} &on $ \Gamma_+ $,
		\end{cases*}
	\end{equation}
	where $ \bar{\vom} $ and $ \bar{\vj} $ are given by \eqref{def bar vom vj}, and
	\begin{equation}\label{div-curl nonlinear vhh""}
		\begin{cases*}
			\div \vhh = 0 \qc \curl \vhh = \vb{0} &in $ \Om_t^- $, \\
			\vhh \vdot \vn = 0 &on $ \Gt $, \\
			\wt{\vn}_- \cp \vhh = \vu{J} &on $ \Gamma_- $.
		\end{cases*}
	\end{equation}

	\subsubsection{Iteration Mapping}
	Suppose that  $ \vu{J}$ is a tangential vector field on $ \Gamma_- $ satisfying \eqref{J'}-\eqref{compatibility J_hat}, and there exists a positive constant $M_*$ so that
	\begin{equation}
		\sup_{t\in[0, T^*]} \qty(\abs{\vu{J}}_{H^{\kk-\frac{1}{2}}(\Gamma_-)} + \abs{\pd_t\vu{J}}_{H^{\kk-\frac{3}{2}}(\Gamma_-)}) \le M_*
	\end{equation} 
	Consider the following function space:
	\begin{defi}
		For given positive constants $ T, M_0, M_1, M_2, M_3, c_0$, and $ \mathfrak{s}_0 $, define $ \mathfrak{X} $ to be the collection of $ \qty(\ka, \vom_*, \vj_*, \va{\mathfrak{v}}, \va{\mathfrak{h}}) $ satisfying:
		\begin{equation}
			\abs{\ka(0) - \kappa_{*}}_{\H{\kk-\frac{5}{2}}} \le \delta_1,
		\end{equation}
		\begin{equation}
			\abs{\qty(\pd_t \ka)(0)}_{\H{\kk-\frac{7}{2}}}, \norm{\vom_*(0)}_{H^{\kk-2}(\Om_*^+)}, \norm{\vj_*(0)}_{H^{\kk-2}(\Om_*^+)}, \abs{\va{\mathfrak{v}}(0)}, \abs{\va{\mathfrak{h}}(0)} \le M_0,
		\end{equation}
		\begin{equation}\label{M1}
			\sup_{t \in [0, T]} \qty(\abs{\ka}_{\H{\kk-\frac{3}{2}}}, \abs{\pd_t \ka}_{\H{\kk-\frac{5}{2}}}, \norm{\vom_*}_{H^{\kk-1}(\Om_*^+)}, \norm{\vj_*}_{H^{\kk-1}(\Om_*^+)}, \abs{\va{\mathfrak{v}}}, \abs{\va{\mathfrak{h}}} ) \le M_1,
		\end{equation}
		\begin{equation}\label{M2}
			\sup_{t \in [0, T]} \qty(\norm{\pd_t\vom_*}_{H^{\kk-2}(\Om_*^+)}, \norm{\pd_t\vj_*}_{H^{\kk-2}(\Om_*^+)}, \abs{\pd_t \va{\mathfrak{v}}}, \abs{\pd_t\va{\mathfrak{h}}}) \le M_2,
		\end{equation}
		and
		\begin{equation}\label{M3}
			\sup_{t \in [0, T]} \abs{\pd^2_{tt}\ka}_{\H{\kk-\frac{7}{2}}} \le a^2 M_3 \ (\text{here $ a $ is the constant in the definition of $ \ka $}).
		\end{equation}
		For $ \Upsilon(\vh, \vhh) $ defined by \eqref{def Upsilon vh vhh},
		\begin{equation*}\label{key}
			\Upsilon\qty(\vh, \vhh) \ge \mathfrak{s}_0
		\end{equation*}
		holds uniformly for $ 0 \le t \le T $.
		In addition, $ \dist(\Gamma_t, \Gamma_\pm) \ge c_0 > 0 $, and the compatibility conditions
		\begin{equation}\label{compatibility}
			\int_{\Gamma_+} \wt{\vn}_+ \vdot \vom_{*} \dd{S_+} = \int_{\Gamma_+} \wt{\vn}_+ \vdot \vj_{*} \dd{S_+} = 0
		\end{equation}
		hold for all $ t\in [0, T] $.
	\end{defi}
	
	As for the initial data, we assume that $ 0 < \epsilon \ll \delta_1 $ and $ A > 0 $, and consider:
	\begin{equation*}
		\mathfrak{I}(\epsilon, A) \coloneqq\qty{\qty\big(\ini{\ka}, \ini{\pd_t\ka}, \ini{\vom_*}, \ini{\vj_*}), \ini{\va{\mathfrak{v}}}, \ini{\va{\mathfrak{h}} }},
	\end{equation*}
	for which
	\begin{gather*}
		\abs{\ini{\ka}-\kappa_{*+}}_{\H{\kk-\frac{3}{2}}}<\epsilon; \\ \abs{\ini{\pd_t\ka}}_{\H{\kk-\frac{5}{2}}},\
		\norm{\ini{\vom_*}}_{H^{\kk-1}(\Om_*^+)},\ \norm{\ini{\vj_*}}_{H^{\kk-1}(\Om_*^+)},\ \abs{\ini{\va{\mathfrak{v}}}},\ \abs{\ini{\va{\mathfrak{h}}}} < A,
	\end{gather*}
	\begin{equation*}\label{key}
		\Upsilon\qty(\ini{\vh}, \ini{\vhh}) \ge 2\mathfrak{s}_0,
	\end{equation*}
	and
	\begin{equation*}
		\dist\qty(\Gt, \Gamma_\pm) \ge 2c_0.
	\end{equation*}
	Furthermore, $ \ini{\vom_*} $ and $ \ini{\vj_*} $ satisfy the compatibility conditions:
	\begin{equation*}
		\int_{\Gamma_+} \wt{\vn}_+ \vdot {\ini{\vom_{*}}} \dd{S_+} = \int_{\Gamma_+} \wt{\vn}_+ \vdot {\ini{\vj_{*}}} \dd{S_+} = 0.
	\end{equation*}
	
	Then, as in \textsection~\ref{sec itetarion map}, define the iteration mappings to be:
	\begin{equation}\label{eqn (n+1)ka""}
		\begin{cases*}
			\pd^2_{tt}\itm{\ka} + \opC_0\qty(\itn{\ka}, \itn{\pd_t\ka}, \itn{\vv_{*}}, \itn{\vh_{*}}, \itn{\vhh_*})\itm{\ka} \\ \qquad = \opF\qty(\itn{\ka})\pd_t \itn{\vom_*} + \opG\qty(\itn{\ka}, \itn{\pd_t\ka}, \itn{\vom_{*}}, \itn{\vj_{*}}, \itn{\va{\mathfrak{v}}}, \itn{\va{\mathfrak{h}}}, \vu{J}) + \mathscr{S}\qty(\itn{\ka})\pd_t \itn{\va{\mathfrak{v}}} \\
			\itm{\ka}(0) = \ini{\ka}, \quad \itm{\pd_t\ka}(0) = \ini{\pd_t\ka};
		\end{cases*}
	\end{equation}
	and
	\begin{equation}\label{eqn linear (n+1) vom vj""}
		\begin{cases*}
			\pd_t\itm{\vom} + \DD_{\itn{\vv}}\itm{\vom} - \DD_{\itn{\vh}}\itm{\vj} = \DD_{\itm{\vom}}\itn{\vv} - \DD_{\itm{\vj}}\itn{\vh}, \\
			\pd_t\itm{\vj} + \DD_{\itn{\vv}}\itm{\vj} - \DD_{\itn{\vh}}\itm{\vom} \\
			\qquad  =  \DD_{\itm{\vj}}\itn{\vv} - \DD_{\itm{\vom}}\itn{\vh}
			- 2\tr(\grad\itn{\vv} \cp \grad \itn{\vh}), \\
			\itm{\vom}(0) = \Pb \qty(\ini{\vom_{*}} \circ (\X_{\itn{\Gamma_0}}^{+})^{-1}), \quad \itm{\vj}(0) = \Pb\qty(\ini{\vj_{*}}\circ (\X_{\itn{\Gamma_0}}^{+})^{-1}).
		\end{cases*}
	\end{equation}
	where $ \qty(\itn{\vv}, \itn{\vh}, \itn{\vhh}) $ are induced by $ \qty(\itn{\ka}, \itn{\vom_{*}}, \itn{\vj_{*}}, \itn{\va{\mathfrak{v}}}, \itn{\va{\mathfrak{h}}}, \vu{J}) $ via solving \eqref{div-curl nonlinear v""}-\eqref{div-curl nonlinear vhh""}, the tangential vector fields $ \itn{\vv_{*}} $, $ \itn{\vh_{*}} $, and $ \itn{\vhh_*} $ on $ \Gs $ are defined by \eqref{def itn vv_*},	and the current-vorticity equations are considered in the domain $ \itn{\Om_t^+} $.
	
	Define
	\begin{equation}\label{eqn dvt va v"}
		\itm{\va{\mathfrak{v}}}(t) \coloneqq \ini{\va{\mathfrak{v}}} + \int_0^t\int_{\Gamma_+} - \DD_{\itn{\vv}}\itn{\vv} - \grad\itn{p} + \DD_{\itn{\vh}}\itn{\vh} \dd{S_+} \dd{t'},
	\end{equation}
	\begin{equation}\label{eqn dvt va h"}
		\itm{\va{\mathfrak{h}}}(t) \coloneqq \ini{\va{\mathfrak{h}}} + \int_0^t \int_{\Gamma_+} \DD_{\itn{\vh}}\itn{\vv} - \DD_{\itn{\vv}} \itn{\vh} \dd{S_+} \dd{t'},
	\end{equation}
	\begin{equation}\label{def itm vom vj}
		\itm{\vom_*}\coloneqq \itm{\vom}\circ\X_{\itn{\Gt}}^+, \qand \itm{\vj_*}\coloneqq \itm{\vj}\circ\X_{\itn{\Gt}}^+,
	\end{equation}
	where $ \itn{p} $ is given by \eqref{decom pressure'} with $ \qty(\itn{\ka}, \itn{\vv}, \itn{\vh}, \itn{\vhh}) $ plugged in.
	
	To verify that the iteration mapping is indeed a map from $\mathfrak{X}$ to $\mathfrak{X}$, one can first check that \eqref{M1}-\eqref{compatibility} still hold for the output by the arguments in \textsection~\ref{sec itetarion map}. Due to the different geometrical requirements, it remains only to check that:
	\begin{equation}\label{dist est}
		\dist\qty(\itm{\Gt}, \Gamma_\pm) \ge 2c_0 - C_*T\abs{\pd_t\itm{\ka}}_{C^0_tH^{\kk-\frac{5}{2}}(\Gs)} \ge 2c_0 - TQ(M_1) \ge c_0,
	\end{equation}
	and
	\begin{equation}\label{Upsi est}
		\abs{\Upsilon\qty(\itm{\vh}, \itm{\vhh}) - \Upsilon\qty(\ini{\vh}, \ini{\vhh})} \le T Q(M_1, M_2, M_*) \le \mathfrak{s}_0,
	\end{equation}
	provided that $T$ is sufficiently small.
	
	Then, with the same notations as in \textsection~\ref{sec itetarion map}:
	\begin{equation}\label{key}
		\begin{split}
			&\mathfrak{T}\qty(\qty[\ini{\ka}, \ini{\pd_t\ka}, \ini{\vom_{*}}, \ini{\vj_{*}}, \ini{\va{\mathfrak{v}}}, \ini{\va{\mathfrak{h}}}], \qty[\itn{\ka}, \itn{\vom_{*}}, \itn{\vj_{*}}, \itn{\va{\mathfrak{v}}}, \itn{\va{\mathfrak{h}}}]) \\
			&\coloneqq \qty(\itm{\ka}, \itm{\vom_{*}}, \itm{\vj_{*}}, \itm{\va{\mathfrak{v}}}, \itm{\va{\mathfrak{h}}}),
		\end{split}
	\end{equation}
	the following proposition is obtained:
	\begin{prop}
		Suppose that $ k \ge 3 $. For any $ 0 < \epsilon \ll \delta_0 $, $ A > 0 $, and $M_* \ge 0$, there are positive constants $ M_0, M_1, M_2, M_3, c_0$, and $\mathfrak{s}_0 $, so that for small $ T > 0 $,
		\begin{equation*}
			\mathfrak{T}\qty\Big{\qty[\ini{\ka}, \ini{\pd_t\ka}, \ini{\vom_{*}}, \ini{\vj_{*}}, \ini{\va{\mathfrak{v}}}, \ini{\va{\mathfrak{h}}}], \qty[\ka, \vom_{*}, \vj_{*}, \va{\mathfrak{v}}, \va{\mathfrak{h}}]} \in \mathfrak{X},
		\end{equation*}
		holds for any $ \qty(\ini{\ka}, \ini{\pd_t\ka}, \ini{\vom_{*}}, \ini{\vj_{*}}, \ini{\va{\mathfrak{v}}}, \ini{\va{\mathfrak{h}}}) \in \mathfrak{I}(\epsilon, A) $ and $ \qty(\ka, \vom_{*}, \vj_{*}, \va{\mathfrak{v}}, \va{\mathfrak{h}}) \in \mathfrak{X} $.
	\end{prop}
	
	To verify that the iteration map is indeed a contraction, as in \textsection~\ref{sec contra ite}, we assume that $$ \qty(\itn{\ka}(\beta), \itn{\vom_{*}}(\beta), \itn{\vj_{*}}(\beta), \itn{\va{\mathfrak{v}}}(\beta), \itn{\va{\mathfrak{h}}}(\beta) ) \subset \mathfrak{X} $$ and  $$ \qty(\ini{\ka}(\beta), \ini{\pd_t\ka}(\beta), \ini{\vom_{*}}(\beta), \ini{\vj_{*}}(\beta), \ini{\va{\mathfrak{v}}}(\beta), \ini{\va{\mathfrak{h}}}(\beta)) \subset \mathfrak{I}(\epsilon, A) $$ are two families of data depending on a parameter $ \beta $.
	
	Define $ \qty(\itm{\ka}(\beta), \itm{\vom_{*}}(\beta), \itm{\vj_{*}}(\beta), \itm{\va{\mathfrak{v}}}(\beta), \itm{\va{\mathfrak{h}}}(\beta) ) $ to be the output of the iteration map. Then, by applying $ \pdv*{\beta} $ to \eqref{eqn (n+1)ka""}-\eqref{def itm vom vj}, one has the variational problems \eqref{eqn beta n+1}-\eqref{eqn g2} as well as:
	\begin{equation}\label{key}
		\begin{cases*}
			\pd^2_{tt} \pd_\beta \itm{\ka} + \itn{\opC}\pd_\beta\itm{\ka} \\ 
			\qquad  = - \qty(\pd_\beta\itn{\opC})\itm{\ka} + \pd_\beta \qty(\itn{\opF}\itn{\pd_t\vom_*} + \itn{\opG} + \itn{\mathscr{S}}\pd_t\itn{\va{\mathfrak{v}}}), \\
			\pd_\beta\itm{\ka}(0) = \pd_\beta\ini{\ka}(\beta), \quad \pd_t\qty(\pd_\beta\itm{\ka})(0) = \pd_\beta\ini{\pd_t\ka}(\beta),
		\end{cases*}
	\end{equation}
	\begin{equation}\label{eqn var pdt va v}
		\begin{split}
			\pd_t\pd_\beta \itm{\va{\mathfrak{v}}} =\pd_\beta\ini{\va{\mathfrak{v}}}+\int_0^t \int_{\Gamma_+} \Dbt \qty(- \DD_{\itn{\vv}}\itn{\vv} -\grad\itn{p} + \DD_{\itn{\vh}}\itn{\vh}) \dd{{S}_+} \dd{t'},
		\end{split}
	\end{equation}
	and
	\begin{equation}\label{eqn var pdt va h}
		\pd_t\pd_\beta\itm{\va{\mathfrak{h}}} = \pd_\beta\ini{\va{\mathfrak{h}}} +  \int_0^t \int_{\Gamma_+} \Dbt\qty(\DD_{\itn{\vh}}\itn{\vv} - \DD_{\itn{\vv}} \itn{\vh}) \dd{{S}_+} \dd{t'}.
	\end{equation}
	
	Consider the energy functionals:
	\begin{equation}\label{key}
		\begin{split}
			\itn{\E}(\beta)\coloneqq &\sup_{t \in [0, T]}\left(\abs{\pd_\beta\itn{\ka}}_{\H{\kk-\frac{5}{2}}} + \abs{\pd_\beta\itn{\pd_t\ka}}_{\H{\kk-\frac{7}{2}}} + \right. \\
			&\qquad\qquad \left.+ \norm{\pd_\beta\itn{\vom_{*}}}_{H^{\kk-2}(\Om_*^+)} + \norm{\pd_\beta\itn{\vj_{*}}}_{H^{\kk-2}(\Om_*^+)} +  \right. \\
			&\qquad\qquad\qquad \left. + \norm{\pd_\beta\pd_t\itn{\vom_{*}}}_{H^{\kk-4}(\Om_*^+)} + \abs{\pd_\beta\itn{\va{\mathfrak{v}}}} + \abs{\pd_\beta\itn{\va{\mathfrak{h}}}} \right),
		\end{split}
	\end{equation}
	and
	\begin{equation}\label{key}
		\begin{split}
			\ini{\E}(\beta)\coloneqq &\sup_{t \in [0, T]}\left(\abs{\pd_\beta\ini{\ka}}_{\H{\kk-\frac{5}{2}}} + \abs{\pd_\beta\ini{\pd_t\ka}}_{\H{\kk-\frac{7}{2}}} + \right. \\
			&\qquad\qquad \left.+ \norm{\pd_\beta\ini{\vom_{*}}}_{H^{\kk-2}(\Om_*^+)} + \norm{\pd_\beta\ini{\vj_{*}}}_{H^{\kk-2}(\Om_*^+)} +  \right. \\
			&\qquad\qquad\qquad \left.  + \abs{\pd_\beta\ini{\va{\mathfrak{v}}}} + \abs{\pd_\beta\ini{\va{\mathfrak{h}}}} \right).
		\end{split}
	\end{equation}
	
	It follows from \eqref{eqn var pdt va v}-\eqref{eqn var pdt va h}, \eqref{est var opG""}-\eqref{est var opS""}, and the arguments in \textsection~\ref{sec contra ite} that
	\begin{equation}\label{key}
		\itm{\E} \le \frac{1}{2}\itn{\E} + Q(M_1)\ini{\E},
	\end{equation} 
	provided that $ T $ is sufficiently small. That is, the following proposition holds:
	\begin{prop}\label{prop fixed point""}
		Assume that $ k \ge 3 $. For any $ 0 < \epsilon \ll \delta_0 $, $ A > 0 $, and $M_* \ge 0$, there are positive constants $ M_0, M_1, M_2, M_3, c_0$, and $ \mathfrak{s}_0 $ so that if $ T $ is small enough, there is a map $ \mathfrak{S} : \mathfrak{I}(\epsilon, A) \to  \mathfrak{X} $ such that
		\begin{equation}\label{key}
			\mathfrak{T}\qty{\mathfrak{x}, \mathfrak{S(x)}} = \mathfrak{S(x)},
		\end{equation}
		for each $ \mathfrak{x} = \qty(\ini{\ka}, \ini{\pd_t\ka}, \ini{\vom_{*}}, \ini{\vj_{*}}, \ini{\va{\mathfrak{v}}}, \ini{\va{\mathfrak{h}}}) \in \mathfrak{I}(\epsilon, A)$.
	\end{prop}
	
	\subsubsection{The Original MHD Problem}
	From the fixed point given in Proposition~\ref{prop fixed point""}, one can obtain the quantities $(\Gt, \vv, \vhh, \vhh)$ via solving the div-curl problems \eqref{div-curl nonlinear v""}-\eqref{div-curl nonlinear vhh""}. In order to check that it is indeed the solution to the original plasma-vacuum problem \eqref{MHD}-\eqref{eqn pM}, with the help of those arguments in \S~\ref{sec Back to ori p-v problem}, it remains only to check that
	\begin{equation}
		\int_{\Gamma_+} \vV \dd{S_+} = \vb{0} = \int_{\Gamma_+} \vH \dd{S_+},
	\end{equation}
	which follows from \eqref{eqn dvt va v"}-\eqref{eqn dvt va h"} and \eqref{def vV}-\eqref{def vH} easily.
	
	As a direct result, Theorem \ref{thm alpha=0 case"} holds.
	
	\subsection{Vanishing Surface Tension Limit}
	Assume that $k \ge 3$, $\Om \coloneqq \mathbb{T}\times(-1, 1)$, and the initial data satisfy the requirements in Theorem~\ref{thm vanishing surf limt"}. To show that the a priori estimates are uniform in the surface tension coefficient $\alpha$, one can consider the following energy functionals:
	\begin{equation}\label{key}
		\mathcal{E}_0 (t) \coloneqq \frac{1}{2}\int_{\Om_t^+} \abs{\vv}^2 + \abs{\vh}^2 \dd{x} + \int_\Gt \alpha^2 \dd{S_t} + \frac{1}{2}\int_{\Om_t^-} \abs*{\vhh}^2 \dd{x}, 
	\end{equation}
	\begin{equation}\label{E1}
		\mathcal{E}_1 (t) \coloneqq \int_\Gt \abs{\opd^{k-2}\n_+^{\frac{1}{2}}\Dt\kappa}^2 + \alpha^2 \abs{\opd^{k-1}\n_+^{\frac{1}{2}}\kappa}^2 + \abs{\opd^{k-2}\n_+^{\frac{1}{2}}\DD_\vh\kappa}^2 + \abs{\opd^{k-2}\n_+^\frac{1}{2}\DD_{\vhh}\kappa}^2 \dd{S_t},
	\end{equation}
	\begin{equation}\label{key}
		\mathcal{E}_2 (t) \coloneqq \norm{\vom}_{H^{\kk-1}(\Om_t^+)}^2 + \norm{\vj}_{H^{\kk-1}(\Om_t^+)}^2,
	\end{equation}
	\begin{equation}\label{key}
		\mathcal{E}_3 (t) \coloneqq \abs{\va{\mathfrak{v}}}^2 + \abs{\va{\mathfrak{h}}}^2
	\end{equation}
	and
	\begin{equation}\label{key}
		\mathcal{E} \coloneqq \mathcal{E}_0 + \mathcal{E}_1 + \mathcal{E}_2 + \mathcal{E}_3,
	\end{equation}
	where in \eqref{E1}, $\opd \coloneqq \qty(-\n_+^{\sfrac{1}{2}}\lap_\Gt\n_+^{\sfrac{1}{2}})^{\sfrac{1}{2}}$.
	
	Observe first that if $T $ is small enough (depending only on the initial data, not on $\alpha$), one may derive from \eqref{dist est}-\eqref{Upsi est} that
	\begin{equation}\label{vanish up dist}
		\Upsilon(\vh, \vhh) \ge \mathfrak{s}_0 > 0,
		\qand
		\dist\qty(\Gt, \Gamma_\pm) \ge c_0 > 0
	\end{equation}
	hold uniformly. Therefore, there exists a generic constant determined by $\Lambda_{*}$ and $c_0$, so that the following div-curl estimates hold (c.f. \cite{Cheng-Shkoller2017}):
	\begin{equation}\label{vani est v}
		\norm{\vv}_{H^{\kk}(\Om_t^+)} \lesssim_{\Lambda_{*}, c_0} \abs{\vn \vdot \lap_\Gt \vv}_{H^{\kk-\frac{5}{2}}(\Gt)} + \norm{\vom}_{H^{\kk-1}(\Om_t^+)} + \abs{\va{\mathfrak{v}}} + \norm{\vv}_{L^2(\Om_t^+)},
	\end{equation}
	\begin{equation}\label{vani est h}
		\norm{\vh}_{H^{\kk}(\Om_t^+)} \lesssim_{\Lambda_{*}, c_0} \abs{\vn \vdot \lap_\Gt \vh}_{H^{\kk-\frac{5}{2}}(\Gt)} + \norm{\vj}_{H^{\kk-1}(\Om_t^+)} + \abs{\va{\mathfrak{h}}} + \norm{\vh}_{L^2(\Om_t^+)},
	\end{equation}
	and
	\begin{equation}\label{vani est hh}
		\norm{\vhh}_{H^{\kk}(\Om_t^-)} \lesssim_{\Lambda_{*}, c_0} \abs{\vu{J}}_{H^{\kk-\frac{1}{2}}(\Gamma_-)}.
	\end{equation}
	
	For the estimates of the energies, one can first derive from \eqref{conserv physic energy}, \eqref{eqn pM}, and \eqref{eqn dt vhh} that
	\begin{equation}
		\abs{\dv{t}\mathcal{E}_0(t)} \lesssim_{\Lambda_{*}, c_0} \norm{\vv}_{H^{\kk}(\Om_t^+)}^2 + \abs{\pd_t\vu{J}}_{H^{\kk-\frac{3}{2}}(\Gamma_-)}^2 + \abs{\vu{J}}_{H^{\kk-\frac{1}{2}}(\Gamma_-)}^2.
	\end{equation}
	For $\mathcal{E}_1$, it follows from \eqref{eqn dt2 kappa} and the arguments in \S~\ref{sec linear ka} that
	\begin{equation}
		\abs{\dv{t}\mathcal{E}_1(t)} \le  Q\qty(\alpha\abs{\kappa}_{H^{\kk-1}(\Gt)}, \abs{\kappa}_{H^{\kk-\frac{3}{2}}(\Gt)}, \norm{(\vv, \vh)}_{H^{\kk}(\Om_{t}^+)}, \norm*{\vhh}_{H^{\kk}(\Om_{t}^-)}),
	\end{equation}
	with a generic polynomial $Q$ determined by $\Lambda_{*}, c_0$ and $\mathfrak{s}_0$. On the other hand, it is clear from \eqref{eqn dt kappa}, the fact that $\Gt \in \Lambda_{*}$, \eqref{vanish up dist}, and \eqref{equiv n lap} that
	\begin{equation}
		\abs{\Dt\kappa}_{H^{\kk-\frac{5}{2}}(\Gt)}^2 + \alpha^2\abs{\kappa}_{H^{\kk-1}(\Gt)}^2 + \abs{\kappa}_{H^{\kk-\frac{3}{2}}(\Gt)}^2 \lesssim_{\Lambda_{*}, \mathfrak{s}_0} 1+ \mathcal{E}_1.
	\end{equation}
	The estimates of $\mathcal{E}_2$ and $\mathcal{E}_3$ can be derived from \eqref{eqn pdt vom}-\eqref{eqn pdt vj} and \eqref{eqn dvt va v"}-\eqref{eqn dvt va h"} that
	\begin{equation}
		\abs{\dv{t}\mathcal{E}_2(t)} + \abs{\dv{t}\mathcal{E}_3(t)} \le Q\qty(\abs{\kappa}_{H^{\kk-\frac{3}{2}}(\Gt)}, \norm{\qty(\vv, \vh)}_{H^{\kk}(\Om_t^+)}).
	\end{equation}
	Furthermore, it is not difficult to derive from \eqref{vani est v}-\eqref{vani est h}, \eqref{est lap_gt vv}, and \eqref{est lap_gt vh} that
	\begin{equation}
		\norm{\vv}_{H^{\kk}(\Om_t^+)}^2 + \norm{\vh}_{H^{\kk}(\Om_t^+)}^2 \le Q(\mathcal{E}).
	\end{equation}
	
	In conclusion, the above arguments lead to
	\begin{equation}
		\mathcal{E}(t) \le \mathcal{E}(0) + \int_0^t Q \qty(\mathcal{E}(t'), \abs{\vu{J}(t')}_{H^{\kk-\frac{1}{2}}(\Gamma_-)}, \abs{\pd_t\vu{J}(t')}_{H^{\kk-\frac{3}{2}}(\Gamma_-)}) \dd{t'},
	\end{equation}
	for a generic polynomial $Q$ depending on $\Lambda_{*}, c_0, \mathfrak{s}_0$, but not on $\alpha$, which yields Theorem \ref{thm vanishing surf limt"}.

	\bibliographystyle{alpha}
	\bibliography{ref.bib}
\end{document}